\documentclass[11pt,twoside]{article}

\usepackage{geometry}

\input{commands}

\begin{document}

\begin{center}

  \textbf{\LARGE{Computationally efficient reductions \\ between some statistical models}}

\vspace*{.2in}

{\large{
\begin{tabular}{ccc}
Mengqi Lou$^{\star}$, Guy Bresler$^{\ddagger}$, Ashwin Pananjady$^{\star, \dagger}$
\end{tabular}
}}
\vspace*{.2in}

\begin{tabular}{c}
Schools of $^\star$Industrial and Systems Engineering and
$^\dagger$Electrical and Computer Engineering, \\
Georgia Tech \\
$^\ddagger$Department of Electrical Engineering and Computer Science, MIT
\end{tabular}

\vspace*{.2in}

\today

\vspace*{.2in}

\begin{abstract}
We study the problem of approximately transforming a sample from a source statistical model to a sample from a target statistical model without knowing the parameters of the source model, and construct several computationally efficient such reductions between canonical statistical experiments. In particular, we provide computationally efficient procedures that approximately reduce uniform, Erlang, and Laplace location models to general target families. 
We illustrate our methodology by establishing nonasymptotic reductions between some canonical high-dimensional problems, spanning mixtures of experts, phase retrieval, and signal denoising. Notably, the reductions are structure-preserving and can accommodate missing data. 
We also point to a possible application in transforming one differentially private mechanism to another.
\end{abstract}
\end{center}


\section{Introduction} \label{sec:intro}

A statistical model or experiment is a collection of probability measures defined on a common sample space. The notation $(\mathbb{X}, \{ \mathcal{P}_{\theta} \}_{\theta \in \Theta})$ is typically used to denote a parametric statistical model on the sample space $\mathbb{X}$, where $\theta$ is an unknown parameter and $\mathcal{P}_{\theta}$ is the corresponding probability model from which a random sample of observations is drawn.
The following classical question, first asked by~\citet{bohnenblust1949reconnaissance}, forms the starting point of this paper:
\begin{center}
\qquad \qquad \quad  \emph{When is one statistical experiment more informative than another?} \hfill (Q1)
\end{center}

 The question can be reframed colloquially for two parametric statistical models given by $\mathcal{U} = (\mathbb{X}, \{ \mathcal{P}_{\theta} \}_{\theta \in \Theta})$ (the source) and $\mathcal{V} = (\mathbb{Y}, \{ \mathcal{Q}_{\theta} \}_{\theta \in \Theta})$ (the target), defined on possibly different sample spaces $\mathbb{X}$ and $\mathbb{Y}$, respectively. For every $\theta \in \Theta$ and \emph{without knowledge of $\theta$}, can a sample from the source distribution $\mathcal{P}_{\theta}$ be transformed into a sample from the distribution $\mathcal{Q}_{\theta}$? In other words, does there exist some randomized procedure $\textsc{K}$ (a \emph{reduction}) such that for every value of $\theta \in \Theta$, if $X \sim \mathcal{P}_{\theta}$, then the random variable $\textsc{K}(X)$ has the distribution $\mathcal{Q}_{\theta}$. If so, the problem of recovering the unknown parameter $\theta$ from a sample of the source statistical model $\mathcal{U}$ can now be solved by a recovery procedure that uses a sample from the target statistical model $\mathcal{V}$.

The concept was further built on by Blackwell in a series of papers~\citep{blackwell1951comparison,blackwell1953equivalent}, where he formalized so-called comparison criteria between experiments in the context of the minimum risk attainable by statistical procedures. There followed a beautiful line of follow-up work (see, e.g.,~\citet{boll1955comparison,lindley1956measure,strassen1965existence,hansen1974comparison} and \citet{le1996comparison} for a survey). 
Somewhat unfortunately, however, it was found that most experiments were incomparable (see, e.g.,~\citet{stone1961non} for some examples). In fact, very few exact comparisons between experiments are known---and all explicit examples of such exact reductions (amounting to less than five families of procedures) can be found in the books by~\citet{torgersen1991comparison} and~\citet{shiryaev2000statistical}.

Accordingly, various approximate notions of comparisons between experiments were introduced (e.g.~\citet{karlin1956theory,brown1976complete,goel1979comparison,lehmann2011comparing}). The most popular among these is the notion of approximate sufficiency formalized by the \emph{total variation deficiency} between experiments~\citep{le1964sufficiency}. Instead of asking when one could exactly reduce the source experiment to the target, Le Cam asked a relaxations of this question, asking when one could produce an \emph{inexact} reduction that is accurate up to some total variation distance. In other words, he was interested in a reduction attaining small total variation deficiency, thereby relaxing the above question (Q1) to:
\begin{center}
\qquad \qquad \qquad \qquad \emph{How far is one statistical experiment from another?} \hfill (Q2)
\end{center}
Equivalently, what is the smallest value $\delta$ so that there exists a randomized procedure $\textsc{K}$ (an \emph{approximate reduction}) with the property that for all $\theta \in \Theta$, if $X \sim \mathcal{P}_{\theta}$, then the distribution of $\textsc{K}(X)$ is $\delta$-close in total variation to $\mathcal{Q}_{\theta}$? Similarly to its exact analog, an approximate reduction allows us to transfer procedures and risk guarantees from the source to taget model (see Eq.~\eqref{eq:pointwise-risk-transfer} to follow).
Le Cam then used the notion of deficiency to define a distance between experiments, which has come to be known as the Le Cam distance. Note that the question (Q2)---as with the original question (Q1) of exact comparisons---remains reasonable for any pair of statistical experiments, even those with a single observation. A specific form of this question forms the focus of the present paper, since it is an important consideration in high dimensions.

To see this more clearly in an example, consider the problem of producing approximate reductions between experiments when there is a \emph{high-dimensional} parameter and independent noise across coordinates. A prototypical example is the structured mean estimation problem, where $\Theta \subseteq \mathbb{R}^n$ is some structured set, and we obtain a single noisy observation of a vector $\theta \in \Theta$ from the source model $\mathcal{P}_{\theta} = \mathcal{P}^1_{\theta_1} \times \cdots \mathcal{P}^n_{\theta_n}$, where $\theta_k$ denotes the $k$-th entry of $\theta$ and the noisy observations are independent across coordinates given the parameter. Other isomorphic examples are those in which the parameter set is a class of matrices or tensors. Now suppose we are interested in reducing our observation to some target model $\mathcal{Q}_{\theta} = \mathcal{Q}^1_{\theta_1} \times \cdots \mathcal{Q}^n_{\theta_n}$, effectively changing the noise distribution while preserving the underlying parameter. Then one conceptually robust way to do so is to produce \emph{entry-by-entry} approximate reductions that are agnostic to the structure in the set $\Theta$, i.e., to approximately reduce source $\mathcal{P}^k_{\theta_k}$ to target  $\mathcal{Q}^k_{\theta_k}$ for each $k = 1, \ldots, n$, thereby producing $n$ \emph{scalar} approximate reductions. An important feature of this problem that is worth emphasizing is that, owing to the high-dimensional setting, we obtain just a \emph{single observation} per coordinate, and so our reduction\footnote{In the sequel, we will largely deal with approximate reductions, but refer to these as reductions for convenience.} must be accurate in a regime where the parameter itself cannot be accurately estimated.

%


While single-sample reductions of the above flavor were the original focus of classical papers (e.g.~\cite{hansen1974comparison,lehmann2011comparing}), they have seen renewed interest in the context of high-dimensional problems in statistics and applied probability, especially for transferring \emph{average-case} computational hardness of estimation and testing tasks so as to provide lower bounds on the (computationally) constrained minimax risk~\citep{berthet2013complexity,wainwright2014constrained,brennan2020reducibility,wu2021statistical}. The idea of transferring computational hardness via reduction has been foundational in theoretical computer science since the start of the field, and in our setting goes as follows: Suppose we believe that estimating $\theta$ from the statistical model $\mathcal{U}$ 
is computationally difficult (in that a computationally efficient algorithm does not exist), and also that we have a computationally efficient, approximate reduction from the source  $\mathcal{U}$ to target $\mathcal{V}$ that attains arbitrarily small total variation (TV) deficiency. Then estimating $\theta$ from the statistical model $\mathcal{V}$ must also be difficult. To see this, suppose for the sake of contradiction that we indeed have a computationally efficient estimation algorithm in the statistical model $\mathcal{V} = (\mathbb{Y}, \{ \mathcal{Q}_{\theta} \}_{\theta \in \Theta})$. Then we could perform estimation in $\mathcal{U} = (\mathbb{X}, \{ \mathcal{P}_{\theta} \}_{\theta \in \Theta})$ by drawing our source sample, performing the reduction, and using our estimation algorithm for $\mathcal{V} = (\mathbb{Y}, \{ \mathcal{Q}_{\theta} \}_{\theta \in \Theta})$ to recover $\theta$. Note that this yields a computationally efficient procedure for estimating $\theta$ using the source experiment $\mathcal{U}$, which is a contradiction. Similar ideas apply to the testing problem. A key desideratum that these reductions must satisfy---in addition to the TV closeness property alluded to above---is that they must be implementable in a computationally efficient manner.
This property distinguishes this line of work from (Q2) above. Indeed, there exists source/target experiments such that their Le Cam distance is small but the reduction procedure cannot be computed in polynomial time under natural conjectures in average-case complexity. We provide such an example in Appendix~\ref{sec:examples-inefficient}.

With this setup in place, let us now discuss a few lines of related work to provide context, followed by our contributions.

\subsection{Related work}

We split our discussion of related work into a few verticals.

\paragraph{Single-sample reductions:} With the motivation sketched above, \citet{hajek2015computational,ma2015computational} provided computationally efficient reductions from ``two-point" source experiments---where both $\mathbb{X}$ and $\Theta$ are sets containing two discrete values---to a family of target experiments having two distinct means. Both papers took the approach of entry-by-entry, single-sample reductions to transforming the planted clique model, which is believed to be computationally hard from an average-case complexity standpoint, to variants of the planted submatrix detection problem. Specifically, \citet{hajek2015computational} provided a single-sample reduction from a pair of Bernoulli models with different means to a pair of Binomial models with different means, and~\cite{ma2015computational} gave a single-sample reduction from a pair of Bernoulli models with different means to a pair of unit-variance Gaussian models with different means. Following up on this work, \citet{brennan2018reducibility} gave generalizations of these reductions 
to univariate target distributions along with a number of other results. \citet{brennan2019universality} also considered two-point sources and general, multivariate target distributions, showing a universality property for computational hardness in submatrix detection problems. Subsequently,~\citep{brennan2020reducibility} demonstrated a universality result for estimation in sparse mixture problems by devising a reduction from \emph{three-point} source distributions---where both $\mathbb{X}$ and $\Theta$ are sets containing three discrete values. 

The aforementioned reductions between statistical models have been used as ingredients---sometimes alongside 
modified conjectures in average-case complexity and 
careful transformations of structure between problem instances~\citep{brennan2020reducibility}---to show computational hardness for a host of high-dimensional statistical problems with combinatorial structure, including (but not limited to) principal component analysis with structure~\citep[e.g.,][]{brennan2019optimal,wang2023algorithms}, location and mixture models with adversarial corruptions~\citep{brennan2020reducibility}, block models in both matrix~\citep{brennan2018reducibility} and tensor~\citep[e.g.,][]{luo2022tensor,han2022exact} settings, tensor PCA~\citep{brennan2020reducibility}, and adaptation lower bounds in shape-constrained regression problems~\citep[e.g.,][]{pananjady2022isotonic}. 
However, the reductions between statistical models are limited to those with either two or three unknown parameters. In this paper, we will handle several cases where the parameter set is \emph{continuous}.

\paragraph{Asymptotic perspectives:} A complementary literature to the non-asymptotic perspective sketched above---with several classical results---is one that views this problem through an asymptotic lens. Indeed, after asking question (Q2) above, Le Cam
 considered asymptotic convergence, in his Le Cam distance, of a \emph{sequence} of experiments indexed by the size of an i.i.d. sample. He asked if such a sequence converges to a limit, 
and in a set of landmark results, showed that under some regularity conditions and with appropriate rescaling, one has convergence to a Gaussian experiment, i.e., a remarkable strengthening of the central limit theorem. This result was a cornerstone for the theory of local asymptotic normality, which along with its counterpart theory of local asymptotic minimaxity has resulted in several groundbreaking discoveries in mathematical statistics over the years; see~\citet{le2000asymptotics,van2000asymptotic} for a more thorough exposition.
Notably, Le Cam's asymptotic theory has an analog for nonparametric models\footnote{See \citet{mariucci2016cam} for a modern survey of this literature.}, as was proved in a line of influential work relating nonparametric regression in the white noise and fixed design models, as well as nonparametric density estimation~\citep{nussbaum1996asymptotic,brown1996asymptotic,brown2002asymptotic,carter2002deficiency,brown2004equivalence}.

In the context of the present paper, it is important to note that this line of work has a complementary focus and is not applicable to the structured high-dimensional settings that motivate our work. First, asymptotic equivalence of experiments has been used mainly in the \emph{analysis} of natural estimators---for example, to show that the MLE in a quadratic mean differentiable model is asymptotically Gaussian~\citep{van2000asymptotic}. 
Second, and on a related note, the reductions inherently rely on large sample sizes and some form of convergence in distribution guaranteed by the central limit theorem. In particular, they only succeed in regimes where the parameter itself is estimable from the data, and so
the reductions used in these techniques typically 
do not perform any sophisticated processing of the source sample, even in complex models~\citep{shiryaev2000statistical}.
Consequently, while they are useful for analytical purposes, these reductions do not provide a path to characterizing the deficiency between general statistical models (as framed in question (Q2)) when we have a single (or fixed-size) observation, which, as alluded to before, is the key issue in high-dimensional problems.

\paragraph{Other related papers:} 
Besides these two lines of literature dealing explictly with reductions, we also mention some related investigations. See~\citet{huber1965robust} for a classical two-point reduction inspired by robust hypothesis testing. The theory of channel comparisons is a related concept, and has seen a line of important work~\citep{shannon1958note,raginsky2011shannon,polyanskiy2017strong,makur2018comparison}; see \citet{polyanskiy2023book} for an introduction. The problem of approximate reductions between statistical models is also related to the recently introduced concept of simultaneous optimal transport~\citep{wang2022simultaneous}.

\subsection{Contributions, motivating examples, and organization}

In this paper, we revisit the problem of constructing non-asymptotic, computationally efficient, approximate reductions between families of source and target statistical models. In doing so, we consider source distributions with continuous parameter spaces, going well beyond the two and three-point cases considered in prior work. In particular, we introduce two general-purpose techniques to construct these reductions, and showcase them by giving explicit reductions between Laplace, Erlang, and Uniform location models and general families of target models, all of which are new. In the process, we show some exact reductions (in the sense of~\cite{blackwell1951comparison}) that, to our knowledge, had not been pointed out in the literature. 

Our techniques demonstrate that it is possible to achieve non-asymptotic approximate reductions for several natural statistical models with continuous parameters.
These can in turn show near-equivalence results for several problems in high-dimensional statistics.
As motivation, we next give some examples of high-dimensional problems for which a formal reduction algorithm would be meaningful. These examples serve as an informal preview: They are formally presented again in Section~\ref{sec:consequences} along with explicit guarantees that result from our general theory. 

\paragraph{Motivation 1} The mixture of experts model has been influential in studying heterogeneous regression problems with so-called ``gating" functions. In a canonical version of the model, the scalar response $Y \in \real$ and high-dimensional covariates $X \in \real^d$ are related via the model $\mathbb{E}[Y | X] = \sgn( g(X) ) \cdot f(X)$ for some unknown regression function $f \in \mathcal{F}$ and gating function $g \in \mathcal{G}$. Variants of this model have been extensively studied under several, often parametric, assumptions on the noise~\citep[e.g.,][]{jordan1994hierarchical,song2014robust}. A related model is one of signless regression, in which we have $\EE[Y | X] = | f(X) |$, which has also been studied under various models for the regression function and noise process~\citep[e.g.,][]{cai2015optimal,duchi2019solving}. When $\mathcal{F}$ is the class of linear functions, the two statistical models go by the names of ``mixture of linear regressions" and ``phase retrieval", respectively. 
Given the similarity between the models, it is natural to ask: Is there a reduction that would allow us to solve a noisy mixture of experts problem through a phase retrieval algorithm when the dimension $d$ is large?

\paragraph{Motivation 2} In signal denoising problems of the form $Y = \theta^* + \epsilon$, we often have sharp characterizations of the error of estimators when the noise $\epsilon$ is Gaussian. For example, if the signal $\theta^*$ is known to belong to a closed convex set in $\real^n$, we can exactly characterize (with both upper and lower bounds that match up to a universal constant) the $\ell_2$ risk of the least squares estimator under zero-mean Gaussian noise~\citep{chatterjee2014new}. While upper bounds of this type can sometimes be proved for other noise distributions (e.g., when $\epsilon$ is drawn from a zero-mean sub-Gaussian), it is not clear if such bounds are sharp in the same way. With the goal of producing an estimator with sharp error guarantees, can we design a reduction from a given statistical model of interest---where $\theta^*$ is known to obey some structure and the noise is known to be drawn from some distribution~\citep[e.g.,][]{sambasivan2018minimax,mcrae2021low}---to one that is arbitrarily close to a Gaussian model? Can we do so in natural settings with structured signals and missing data?

\paragraph{Motivation 3} There are several privacy-preserving mechanisms for data analysis, and many rely on the notion of differential privacy, which is often guaranteed by adding noise from a certain distribution to the output of a database query~\citep{dwork2006differential}. Several such noise addition mechanisms have been proposed and studied in the literature, and each comes with its own set of operational advantages. A natural question is whether one can transform between these mechanisms in a downstream fashion, that is, without needing to de-privatize the query itself. For two popular privacy preserving mechanisms, this natural question is equivalent to one covered by our setting: Can a Laplace random variable (or vector) of unknown mean can be transformed into a Gaussian random variable (or vector) of the same mean? How does the resulting mechanism compare with a direct Gaussian mechanism in terms of its privacy and accuracy guarantee? 

\medskip

The rest of this paper is organized as follows. In Section~\ref{sec:setup}, we formally present the problem of constructing reductions with minimal total variation deficiency, and discuss why a natural reduction may be a poor candidate for this problem. In Section~\ref{sec:rej}, we make the oracle assumption that we are given some candidate ``signed kernel" between the source and target experiments, and provide a computationally efficient algorithm for generating a sample that resembles one from the target experiment. 
In Section~\ref{sec:reductions}, we use this result alongside explicit constructions of signed kernels for source models given by the multivariate Laplace and univariate Erlang and Uniform location models, and provide explicit reductions for general target models. We present applications to high-dimensional problems in Section~\ref{sec:consequences}, including answers to the questions posed in Motivations 1-3 above. 
We conclude with a discussion of open problems in Section~\ref{sec:discussion}. Proofs of all our results can be found in Section~\ref{sec:proofs}.


\section{Background and preliminary observations} \label{sec:setup}

We begin by defining the reduction problem for general statistical models for completeness, and note simplifications of our setup when we have standard probability mass and density functions. 

\subsection{Definitions}
To begin, note that all the spaces considered in this paper are standard Borel, which are Borel subspaces of Polish spaces. This means any such space $\mathbb{U}$ will be endowed with its Borel sigma-algebra $\mathcal{B}(\mathbb{U})$. The space of all probability measures on $\mathbb{U}$ will be denoted by $\Delta(\mathbb{U})$ and the space of all signed measures on $\mathbb{U}$ by $\Sigma(\mathbb{U})$.
Following the notation in~\cite{raginsky2011shannon}, we use $\| \mu - \nu \|_{\mathsf{TV}}$ to denote the total variation distance between two measures $(\mu, \nu)$ defined on the same space. If $X \sim \mu$ and $Y \sim \nu$, then we use the notation $\mathsf{d_{TV}}(X, Y) := \| \mu - \nu \|_{\mathsf{TV}}$. A Markov kernel between any two spaces $\mathbb{X}$ and $\mathbb{Y}$ is mapping $\mathcal{T}: \mathcal{B}(\mathbb{Y}) \times \mathbb{X} \to [0, 1]$ such that $\mathcal{T}(\cdot | x) \in \Delta(\mathbb{Y})$ for all $x \in \mathbb{X}$ and $\mathcal{T}(B | \cdot)$ is a measurable function on $\mathbb{X}$ for any $B \in \mathcal{B}(\mathbb{Y})$. When $\mathbb{X}$ and $\mathbb{Y}$ are both finite, $\mathcal{T}$ is a stochastic matrix. Let $\mathsf{M}(\mathbb{Y} | \mathbb{X})$ denote the space of all such Markov kernels $\mathcal{T}$. Analogously, define a signed kernel between $\mathbb{X}$ and $\mathbb{Y}$ as a mapping $\mathcal{S}: \mathcal{B}(\mathbb{Y}) \times \mathbb{X} \to \real$ such that $\mathcal{S}(\cdot | x) \in \Sigma(\mathbb{Y})$ and $\mathcal{S}(B | \cdot)$ is a measurable function on $\mathbb{X}$ for any $B \in \mathcal{B}(\mathbb{Y})$. Let $\mathsf{S}(\mathbb{Y} | \mathbb{X})$ denote the space of all such signed kernels.

From this point onward, we reserve the notation $\mathbb{X}$ for the input sample space and $\mathbb{Y}$ for the output sample space. We use $\Theta$ to denote the space of (unknown) parameters. With this notation, one can consider some special Markov kernels. The source kernel (or experiment) is denoted by $\mathcal{U} : \mathcal{B}(\mathbb{X}) \times \Theta \to [0, 1]$---thus, $\mathsf{M}(\mathbb{X} | \Theta)$ denotes the family of all possible source (or input) experiments. Analogously, the target experiment is given by $\mathcal{V} :  \mathcal{B}(\mathbb{Y}) \times \Theta \to [0, 1]$ and the family of all possible target (or output) experiments is $\mathsf{M}(\mathbb{Y} | \Theta)$.
In statistics, it is common to use $u( \cdot ; \theta) \in \Delta(\mathbb{X})$ for $\mathcal{U}(\cdot | \theta)$ and $v( \cdot ; \theta) \in \Delta(\mathbb{Y})$ for $\mathcal{V}(\cdot | \theta)$, and we often use this notation.

It is natural to define a (maximal) TV distance between two Markov kernels $\mathcal{T}, \mathcal{T}' \in \mathsf{M}(\mathbb{Y} | \mathbb{X})$:
$
	\| \mathcal{T} - \mathcal{T}' \|_\infty := \sup_{x \in \mathbb{X}}\| \mathcal{T}( \cdot | x ) -  \mathcal{T}'( \cdot | x ) \|_{\mathsf{TV}}.
$
	This definition also extends to signed kernels via transformation to the $\ell_1$ norm. For signed kernels $\mathcal{S}, \mathcal{S}' \in \mathsf{S}(\mathbb{Y} | \mathbb{X})$, we write 
$
	\| \mathcal{S} - \mathcal{S}' \|_\infty := \sup_{x \in \mathbb{X}} \; \frac{1}{2} \| \mathcal{S}( \cdot | x ) -  \mathcal{S}'( \cdot | x ) \|_{1}.
$
It will be useful to define compositions of Markov kernels and signed kernels. For two Markov kernels $\mathcal{U} \in \mathsf{M}(\mathbb{X} | \Theta)$ and $\mathcal{T} \in \mathsf{M}(\mathbb{Y} | \mathbb{X})$, the object $\mathcal{T} \circ \mathcal{U} \in \mathsf{M}(\mathbb{Y} | \Theta)$ also a Markov kernel. In particular, for $C \in \mathcal{B}(\mathbb{Y})$ and $\theta \in \Theta$, we have
$
	(\mathcal{T} \circ \mathcal{U}) (C | \theta) = \int_{\mathbb{X}} T(C | x) \cdot \mathcal{U}( \mathrm{d} x| \theta).
$
Similarly, one can define the signed kernel $\mathcal{S} \circ \mathcal{W} \in \mathsf{S}(\mathbb{Y} | \Theta)$ for $\mathcal{S} \in \mathsf{S}(\mathbb{Y}| \mathbb{X})$ and $\mathcal{W} \in \mathsf{S}(\mathbb{X} | \Theta)$.

For a Markov kernel $\mathcal{T} \in \mathsf{M}(\mathbb{Y} | \mathbb{X})$, let
\begin{subequations}
\begin{align} \label{eq:one-way-fixed}
	\delta( \mathcal{U}, \mathcal{V}; \mathcal{T}) := \sup_{\theta \in \Theta} \; \left\| 
	\mathcal{V}( \cdot | \theta) - \int_{\mathbb{X}}  \mathcal{T}( \cdot| x) \cdot \mathcal{U}( \mathrm{d} x | \theta)
	\right\|_{\mathsf{TV}} = \| \mathcal{V} - (\mathcal{T} \circ \mathcal{U}) \|_\infty
	\end{align}
	be the deficiency attained by $\mathcal{T}$ when mapping the source experiment $\mathcal{U}$ to the target experiment $\mathcal{V}$. 
	The experiment $\mathcal{U}$ is said to be $\epsilon$-deficient relative to $\mathcal{V}$ if there exists a Markov kernel $\mathcal{T}$ such that $\delta( \mathcal{U}, \mathcal{V}; \mathcal{T}) \leq \epsilon$. Markov kernel $\mathcal{T}$ is then said to \emph{witness} $\epsilon$-deficiency between $\mathcal{U}$ and $\mathcal{V}$. The global (one-way) $\delta$-deficiency between the source and target experiments is given by
	\begin{align} \label{eq:one-way-def}
	\delta( \mathcal{U}, \mathcal{V}) &:= \inf_{\mathcal{T} \in \mathsf{M}(\mathbb{Y} | \mathbb{X})} \delta( \mathcal{U}, \mathcal{V}; \mathcal{T}) = \inf_{\mathcal{T} \in \mathsf{M}(\mathbb{Y} | \mathbb{X})} \; \| \mathcal{V} - (\mathcal{T} \circ \mathcal{U}) \|_\infty.
	\end{align}
\end{subequations}
The above infimum is achieved; this result is known as the Le Cam randomization theorem.

One can define an analogous notion over signed kernels. 
For each $\mathcal{S} \in \mathsf{S}(\mathbb{Y} | \mathbb{X})$, define
\begin{subequations}
	\begin{align} \label{eq:one-way-sign-fixed}
	\SignedDef( \mathcal{U}, \mathcal{V}; \mathcal{S}) := \sup_{\theta \in \Theta} \; \frac{1}{2} \left\| 
	\mathcal{V}( \cdot | \theta) - \int_{\mathbb{X}}  \mathcal{S}( \cdot| x) \cdot \mathcal{U}( \mathrm{d} x | \theta)
	\right\|_{1} = \| \mathcal{V} - (\mathcal{S} \circ \mathcal{U}) \|_\infty
	\end{align}
	and the global version
	\begin{align} \label{eq:one-way-sign}
	\SignedDef( \mathcal{U}, \mathcal{V}) := \inf_{\mathcal{S} \in \mathsf{S}(\mathbb{Y} | \mathbb{X})} \; \SignedDef( \mathcal{U}, \mathcal{V}; \mathcal{S}) = \inf_{\mathcal{S} \in \mathsf{S}(\mathbb{Y} | \mathbb{X})} \; \| \mathcal{V} - (\mathcal{S} \circ \mathcal{U}) \|_\infty.
	\end{align}
\end{subequations}
	This just amounts to dropping a set of linear constraints (positivity and summation to one) from the optimization problem defining~\eqref{eq:one-way-def}. In contrast to the program~\eqref{eq:one-way-def}, there is no guarantee that the infimum is achieved in Eq.~\eqref{eq:one-way-sign} since the set $\mathsf{S}(\mathbb{Y} | \mathbb{X})$ is not compact. As a consequence of the inclusion $\mathsf{M}(\mathbb{Y} | \mathbb{X}) \subseteq \mathsf{S}(\mathbb{Y} | \mathbb{X})$, we have the immediate inequality
	\begin{align} \label{eq:domination}
	\SignedDef( \mathcal{U}, \mathcal{V}) \leq \delta ( \mathcal{U}, \mathcal{V}).
	\end{align}

When the space $\mathbb{X}$ is a measurable subset of $\mathbb{R}^d$ (as will be in all the particular cases to follow), we use the standard notation $u(x; \theta)$ to denote the source density with respect to Lebesgue measure. Similar conventions will apply---when $\mathbb{Y}$ is a measurable subset of $\mathbb{R}^d$---for the output model $v(y; \theta)$, Markov kernel $\mathcal{T}(y|x)$, and signed kernel $\mathcal{S}(y|x)$. This notation will be used from Section~\ref{sec:reductions} onward.

\subsection{Objective}

Our overall goal is to, in polynomial time, design and sample from a Markov kernel witnessing that $\mathcal{U}$ is $\epsilon$-deficient relative to $\mathcal{V}$. Operationally, there is some unknown $\theta \in \Theta$ fixed by Nature. We draw a sample from the source channel $\mathcal{U}(\cdot | \theta)$; let us denote this by the random variable $X_{\theta}$. We would like to transform this source sample to a sample over the target sample space and obtain via an algorithm \textsc{K} the random variable $\textsc{K}(X_\theta)$, with the following twofold objective:
\begin{itemize}
\item[(a)] \textbf{Statistical:} Uniformly over $\theta \in \Theta$, the random variable $\textsc{K}(X_\theta)$ must be $\epsilon$-close in total variation distance to an output random variable $Y_{\theta} \sim \mathcal{V}(\cdot| \theta)$, i.e.,
\[
\sup_{\theta \in \Theta} \; \mathsf{d_{TV}}(\textsc{K}(X_\theta), Y_\theta) \leq \epsilon, \text{ and}
\] 
\item[(b)] \textbf{Computational:} The transformation algorithm $\textsc{K}$ is ``computationally efficient", in that it must run in time\footnote{Throughout, we assume access to real arithmetic and ignore issues associated with discretization via dyadic expansions; a detailed treatment of this step is presented by~\cite{ma2015computational}.} polynomial in $1/\epsilon$ and a reasonable notion\footnote{The notion of ``size'' should be intuitive for natural source, target, and parameter spaces. For example, if these spaces are given by subsets of $\real^d$, then we would like our reduction to run in time polynomial in $d$.} of ``size'' of $(\Theta, \mathbb{X}, \mathbb{Y})$. 
\end{itemize}
There are several consequences of having such a reduction algorithm $\textsc{K}$. The most classical is the consequence for risk of estimation.  Concretely, suppose that $L: \Theta \times \Theta \to [0, 1]$ is some bounded loss function. Now suppose we have an estimator $\widehat{\theta}: \mathbb{Y} \to \Theta$ for the target model. 
Then the estimator $\widetilde{\theta} := \widehat{\theta} \circ \textsc{K}: \mathbb{X} \to \Theta$ is an estimator for the source model. \cite{le1964sufficiency} showed that if the statistical property (a) above holds, then for all $\theta \in \Theta$, the estimator $\widetilde{\theta}$ obeys 
\begin{align} \label{eq:pointwise-risk-transfer}
\mathbb{E}[L(\theta, \widetilde{\theta})] \leq \EE[L(\theta, \widehat{\theta})] + \epsilon.
\end{align}
While such a property is true for the pointwise risk at parameter $\theta$, similar statements can be made about the source and target minimax risks, given respectively by
\begin{align} \label{eq:minimax-def}
\mathfrak{M}_{\mathcal{U}} :\overset{\1}{=} \inf_{\widehat{\theta} \in \Theta} \;  \sup_{\theta \in \Theta} \; \EE[L(\theta, \widehat{\theta})] \quad \text{ and } \quad \mathfrak{M}_{\mathcal{V}} :\overset{\2}{=} \inf_{\widetilde{\theta} \in \Theta} \; \sup_{\theta \in \Theta} \; \EE[L(\theta, \widetilde{\theta})].
\end{align}
In case $\1$, the expectation is taken over $X_{\theta} \sim \mathcal{U}(\cdot | \theta)$, and the infimum over $\widehat{\theta}$ over all measurable functions of the observation $X_{\theta}$. In case $\2$, the expectation is taken over $Y_{\theta} \sim \mathcal{V}(\cdot | \theta)$, and the infimum over $\widetilde{\theta}$ is taken over all measurable functions of the observation $Y_{\theta}$. If there is a Markov kernel $\textsc{K}$ obeying the statistical property (a) above, then it is straightforward to show that
\begin{align} \label{eq:minimax-transfer}
\mathfrak{M}_{\mathcal{V}} \geq \mathfrak{M}_{\mathcal{U}} - \epsilon.
\end{align}
Furthermore, define $\mathfrak{M}^{\mathsf{comp}}_{\mathcal{U}}$ and $\mathfrak{M}^{\mathsf{comp}}_{\mathcal{V}}$ as the respective \emph{constrained} minimax risks~\citep{wainwright2014constrained} in which we take infima in Eq.~\eqref{eq:minimax-def} only over computationally efficient estimators, in the heuristic sense of property (b) above. Then if $\textsc{K}$ additionally obeys the computational property (b), we have the same relation between the constrained minimax risks
$
\mathfrak{M}^{\mathsf{comp}}_{\mathcal{V}} \geq \mathfrak{M}^{\mathsf{comp}}_{\mathcal{U}} - \epsilon.
$
Similar statements to the above can also be made for unbounded losses; see Appendix~\ref{sec:loss-unbounded}.

\subsection{Additional notation} 

We use $\mathcal{L}[X]$ to denote the law of a random variable $X$. For two (possibly signed) measures $\mu$ and $\nu$, we use $\frac{\mathrm{d} \nu}{\mathrm{d} \mu}$ to denote the Radon--Nikodym derivative of $\nu$ with respect to $\mu$. We use $\mathsf{Unif}(S)$ to denote the uniform distribution on a set $S$ and $\mathsf{Exp}(\lambda)$ to denote an exponential distribution of rate $\lambda$. The notation $\mathsf{Lap}(\theta, b)$ is used to denote a Laplace distribution with mean $\theta$ and scale $b$, and $\NORMAL(\theta, \sigma^2)$ to denote a normal distribution with mean $\theta$ and variance $\sigma^2$. For a function of two variables $v(y; \theta)$, we write $\nabla_{\theta} v(y; \theta) \vert_{\theta = \widetilde{\theta}}$ to denote the gradient of $v$ with respect to its second argument $\theta$, evaluated at $\theta = \widetilde{\theta}$. Similar conventions apply to $k$-th order partial derivatives $\nabla^{(k)}_\theta$ for $k \geq 2$. Occasionally, we will use operator notation. For example, $I + \nabla_{\theta}$ should be viewed as an operator, with $\left((I + \nabla_{\theta}) v(y; \theta) \right)\mid_{\theta = \widetilde{\theta}} \; := v(y; \theta) \mid_{\theta = \widetilde{\theta}} + \nabla_{\theta} v(y; \theta) \mid_{\theta = \widetilde{\theta}}$. With this notation, one can define the composition of operators, e.g., $\nabla_{\theta_1} \circ \nabla_{\theta_2}$ denotes the successive partial derivative operation with respect to the second entry and then the first entry of~$\theta$.

We use $(c, c', C, C')$ to denote universal positive constants that may take different values in each instantiation. 
For two nonnegative sequences $f_n$ and $g_n$, we say that $f_n = \mathcal{O}(g_n)$ if there is a universal constant $C > 0$ such that $f_n \leq C g_n$ for all $n \in \mathbb{N}$. We say that $f_n = \Omega(g_n)$ if $g_n = \mathcal{O}(f_n)$. Logarithms are taken to the base $e$. We write $\sgn(z) = 1$ if $z \geq 0$ and $-1$ otherwise. We use $\ind{\cdot}$ to denote the indicator function.



\subsection{How good is a natural plug-in approach?} \label{sec:eq-plugin}

Having set up the reduction problem formally, we first address a basic question that will set the stage for our development. In particular, 
we use the specific example of mapping a source location model to a Gaussian location model in dimension $1$ to highlight issues with a natural approach to reductions.
%
To be concrete, consider a specific case in which the output space is given by the real line, i.e., $\mathbb{Y} = \mathbb{R}$. Suppose we are interested in implementing a reduction from a source statistical model $\mathcal{U} = \left( \mathbb{X}, \{ \mathcal{P}_{\theta} \}_{\theta \in \Theta} \right)$ to the target model $\mathcal{V} = \left( \mathbb{Y}, \{\mathcal{Q}_{\theta} \}_{\theta \in \Theta} \right)$ when both models are location models, i.e., the parameter $\theta$ denotes the unknown mean, and the parameter set is a subset of real values, i.e., we have $\Theta \subseteq \mathbb{R}$. 
Furthermore, suppose that the target probability models of interest are Gaussian with variance $\sigma^2$, i.e., we have $\mathcal{Q}_{\theta} = \mathcal{N}(\theta, \sigma^2)$.

A natural approach\footnote{Another natural plug-in approach is to add to $X_{\theta}$ a zero-mean Gaussian of variance $\widetilde{\sigma}^2 \leq \sigma^2$, so as to match the variance of the target experiment. A similar lower bound to Proposition~\ref{prop:plugin} can also be established for this approach.} to performing a reduction to a Gaussian random variable of mean $\theta$ and variance $\sigma^2$ is to ``plug in'' $X_{\theta}$ as a surrogate for the mean $\theta$, and output the target random variable 
\begin{align} \label{eq:reduction-plugin}
\textsc{K}_{\mathsf{plugin}}(X_{\theta}) = X_{\theta} + \sigma Z, \text{ where } Z \sim \mathcal{N}(0, 1).
\end{align}
Note that this reduction is very robust---indeed, it does not rely on any distributional properties of the source experiment! However, the following proposition provides a lower bound on its TV deficiency as a function of $\sigma$ for some canonical source experiments. 

\begin{proposition} \label{prop:plugin}
There is a universal constant $c > 0$ such that the following is true. Suppose $\theta \in \Theta$ and the source location model generates the random variable $X_{\theta} = \theta + W$. Consider any of the following cases:
\begin{enumerate}
	\item[(i)] Laplace: $\Theta = \mathbb{R}$ and $W \sim \mathsf{Lap}(0, 1)$,
	\item[(ii)] Shifted exponential: $\Theta = \mathbb{R}$ and $W + 1 \sim \mathsf{Exp}(1)$
	\item[(iii)] Uniform: $\Theta = [-1/2, 1/2]$ and $W \sim \mathsf{Unif}([-1/2, 1/2])$.
\end{enumerate}
Recall the plugin reduction from Eq.~\eqref{eq:reduction-plugin}, and let $Y_{\theta} \sim \mathcal{N}(\theta, \sigma^2)$ denote the desired target for some $\sigma \geq 1$. Then in cases (i) and (iii) for the source model, we have
$\mathsf{d_{TV}}(\textsc{K}_{\mathsf{plugin}}(X_{\theta}), Y_{\theta}) \geq c \sigma^{-2}$,
and in case (ii), we have $\mathsf{d_{TV}}(\textsc{K}_{\mathsf{plugin}}(X_{\theta}), Y_{\theta}) \geq c \sigma^{-1}$.
\end{proposition}


Proposition~\ref{prop:plugin} is proved in Section~\ref{sec:pf-prop-plugin}. It shows that for several source experiments, the plug-in reduction attains TV deficiency that is \emph{lower-bounded} by a \emph{polynomially} scaling function in $\sigma^2$. Thus, one would need $\sigma^2$ to scale at least polynomially in $1/\epsilon$ to be able to attain TV deficiency $\epsilon$ with the plugin approach. 
In the following sections, we build up to a much better reduction procedure that attains TV deficiency decaying \emph{exponentially} in $\sigma^2$.

\section{An oracle rejection sampling algorithm} \label{sec:rej}

In this section, we prove an algorithmic lemma, showing that if we have oracle access to some $\mathcal{S}^*$ that is a feasible solution to the optimization problem~\eqref{eq:one-way-sign} over signed kernels,
then we can efficiently sample from a nearby Markov kernel and explicitly bound its deficiency.

\begin{assumption} \label{assptn:exact}
We can evaluate some $\mathcal{S}^* \in \mathsf{S}(\mathbb{Y} | \mathbb{X})$ in time $T_{eval}$ for each pair $(x, C) \in \mathbb{X} \times \mathcal{B}(\mathbb{Y})$.
\end{assumption}

For notational convenience, define for each $(x, C) \in \mathbb{X} \times \mathcal{B}(\mathbb{Y})$ the truncation $\overline{\mathcal{S}}(C|x) := \mathcal{S}^*(C|x) \vee 0$.
By Assumption~\ref{assptn:exact}, we have oracle access to $\mathcal{S}^*$ and thus to $\overline{\mathcal{S}}$.

\begin{assumption} \label{assptn:RND}
There exists $M > 0$ and a family of base probability measures $\{ \mathcal{P}(\cdot | x) \}_{x \in \mathbb{X}}$ s.t. \\
\noindent (a) For any $x \in \mathbb{X}$, we can sample from $\mathcal{P}(\cdot | x) \in \Delta(\mathbb{Y})$ in time $T_{samp}$. \\
\noindent (b) For all $x \in \mathbb{X}, y\in \mathbb{Y}$, the Radon--Nikodym derivative $\frac{\mathrm{d} \overline{\mathcal{S}}}{\mathrm{d} \mathcal{P}}( y |  x) \leq M$.
\end{assumption}

We now present our randomized reduction algorithm under the two assumptions above.
%

\noindent \hrulefill

\noindent \underline{{\bf Algorithm: General rejection kernel} $\textsc{rk}(x, N, M, y_0)$}

\smallskip

\noindent \textit{Parameters}: Input $x \in \mathbb{X}$, number of iterations $N$, parameter $M$ satisfying Assumption~\ref{assptn:RND}(b), initialization $y_0 \in \mathbb{Y}$. 
\begin{itemize} \label{algo:rk}
\item Initialize $t = 0$, $Y = y_{0}$, and mark ``unset''.
\item Until $Y$ is set or $t\geq N+1$ iterations have elapsed:
\begin{itemize}
\item[(1)] Sample $U_{t} \sim \mathsf{Unif}([0, 1])$ and $Y_{t} \sim \mathcal{P}(\cdot | x)$ mutually independently from the past.
\item[(2)] If $U_{t} \leq \frac{1}{M} \frac{\mathrm{d} \overline{\mathcal{S}}}{ \mathrm{d} \mathcal{P}} ( Y_{t} | x)$, then set $Y \gets Y_{t}$. 
\item[(3)] Set $ t \gets t+1$.
\end{itemize}
\end{itemize}
\noindent \emph{Output} $Y$.

\noindent \hrulefill

The goal of Algorithm \textsc{rk} is to sample from the Markov kernel $\widehat{\mathcal{T}}$ defined via
\begin{align} \label{eq:MK-thresh}
\widehat{\mathcal{T}}( C | x) = \frac{\mathcal{S}^*( C | x) \vee 0}{\int_{\mathbb{Y}} (\mathcal{S}^*( \mathrm{d} y' | x) \vee 0)} \quad \text{for each } (x, C) \in \mathbb{X} \times \mathcal{B}(\mathbb{Y}).
\end{align}
Algorithm \textsc{rk} uses rejection sampling to accomplish this goal. Our bound on its total variation deficiency requires some definitions.
Define the functions $p, q: \mathbb{X} \to [0, \infty)$ via 
\begin{align} \label{eq:pq}
p(x) := \int_{\mathbb{Y}} [\mathcal{S}^* (\mathrm{d} y|x) \vee 0] \quad \text{ and } \quad q(x) := \int_{\mathbb{Y}} - [\mathcal{S}^* (\mathrm{d} y|x) \land 0].
\end{align}
In principle, the functions $p$ and $q$ ought to also be indexed by the kernel $\mathcal{S}^*$, but we suppress this notation for ease of exposition. Recall that $X_{\theta} \sim u(\cdot ;\theta)$ denotes the input random variable when the (unknown) true parameter is $\theta$. 
\begin{lemma} \label{lem:rej-sampling}
Suppose Assumptions~\ref{assptn:exact} and~\ref{assptn:RND} hold. Then for any $\epsilon \in (0, 1)$ and any $x \in \mathbb{X}, y_0 \in \mathbb{Y}$, the randomized algorithm $x \mapsto \textsc{rk}(x, N, M, y_0)$ runs in time $O\big(N (T_{eval} + T_{samp}) \big)$ and has deficiency bounded as
\begin{align} \label{eq:def-ub}
&\sup_{\theta \in \Theta} \; \| \mathcal{L} \left[ \textsc{rk}(X_{\theta}, N, M, y_0) \right] - v( \cdot ; \theta) \|_{\mathsf{TV}} \nonumber \\
&\qquad \qquad  \leq 2 \exp \left\{ -\frac{N}{M} \cdot \inf_{x\in \mathbb{X}} p(x) \right\} + \SignedDef(\mathcal{U}, \mathcal{V}; \mathcal{S}^*) +  \| \mathcal{S}^* \circ \mathcal{U} - \widehat{\mathcal{T}} \circ \mathcal{U} \|_{\infty} \nonumber \\
&\qquad \qquad  \leq 2 \exp \left\{ -\frac{N}{M}\cdot \inf_{x\in \mathbb{X}} p(x)  \right\} + \SignedDef(\mathcal{U}, \mathcal{V}; \mathcal{S}^*) + \frac{1}{2} \sup_{\theta \in \Theta} \; \int_{\mathbb{X}} \left( |p(x) - 1| + q(x) \right) \cdot u(\mathrm{d} x; \theta).
\end{align}
\end{lemma}

Lemma~\ref{lem:rej-sampling} is proved in Section~\ref{sec:pf-lemma1} and decouples the rejection sampling algorithm from the quality of the kernels involved. In particular, the first term on the RHS of Eq.~\eqref{eq:def-ub} bounds the distortion introduced due to the sampling step, and is small provided $N \gg M (\inf_x p(x))^{-1}$. Since the algorithm has runtime linear in $N$, we would like $M$ and $(\inf_x p(x))^{-1}$ to be small in order for the reduction to run computationally efficiently. Now considering the remaining terms, the overall distortion is small provided we have access to a ``good'' signed kernel $\mathcal{S}^*$. 
In particular, we require a signed kernel $\mathcal{S}^*$ that (a) has small deficiency with low $\SignedDef(\mathcal{U}, \mathcal{V}; \mathcal{S}^*)$ and (b) is minimally distorted under the transformation~\eqref{eq:MK-thresh} in that $\| \mathcal{S}^* \circ \mathcal{U} - \widehat{\mathcal{T}} \circ \mathcal{U} \|_{\infty}$ is small. The latter condition can in turn be ensured if $|p(x) -1|$ and $q(x)$ are small.

Algorithm~\textsc{rk} itself is similar in spirit to the rejection kernels introduced in~\citet{brennan2018reducibility,brennan2019universality,brennan2020reducibility}, in that it uses a form of rejection sampling in order to approximately sample from a distribution. However, it is distinct in two crucial ways: (a)~It accommodates a possibly infinite input sample space $\mathbb{X}$ whereas existing rejection kernels are specifically designed for the binary (Bernoulli)~\citep{brennan2018reducibility,brennan2019universality} or ternary~\citep{brennan2020reducibility} case; and (b)~It accommodates a general family of base measures $\{ \mathcal{P}(\cdot | x) \}_{x \in \mathbb{X}}$, and its running time depends explicitly on the parameter $M$ from Assumption~\ref{assptn:RND}(b). In that sense, one can choose a family of base measures that is computationally easy to sample from per iteration (i.e., having a small $T_{samp}$) at the cost of a larger parameter $M$.

\section{Concrete reductions} \label{sec:reductions}

In this section, we provide a few concrete reductions for some source and target pairs of statistical experiments. In view of Lemma~\ref{lem:rej-sampling}, it suffices to show instances in which signed kernels $\mathcal{S}^*$ can be constructed, and to evaluate the bound~\eqref{eq:def-ub}. We will consider cases where the source and target are densities on $\real^d$, and our notation will reflect this (see Section~\ref{sec:setup}). 

We will frame our goal as one of constructing some signed kernel $\mathcal{S}^*$ such that $\SignedDef(\mathcal{U}, \mathcal{V}; \mathcal{S}^*) = 0$, i.e., finding an exact solution to the infinite dimensional optimization problem over signed kernels~\eqref{eq:one-way-sign}. In particular, suppose the source experiment is a location model, in which case we have $u(x; \theta) = f(x - \theta)$ for some density $f$. If $\SignedDef(\mathcal{U}, \mathcal{V}; \mathcal{S}^*) = 0$, then the following (infinite) system of equations must hold for all $\theta \in \Theta$ and $y \in \mathbb{Y}$:
\begin{align} \label{eq:infinite-deconvolution}
\int_{\real^d} f(x - \theta) \cdot \mathcal{S}^*(y| x) \mathrm{d} x = v(y; \theta).
\end{align}
We demonstrate two techniques to compute such a kernel $\mathcal{S}^*$, both of which are based on viewing Eq.~\eqref{eq:infinite-deconvolution} as an infinite family of deconvolution problems---one for each $y \in \mathbb{Y}$. The first, natural technique performs deconvolution via Fourier inversion, which we rigorously verify using a delicate integration by parts argument. This technique is showcased on the multivariate Laplace and univariate Erlang location models as the source, with $\Theta$ given by $\real^d$ and $\real$, respectively. The second technique is to perform deconvolution via differentiation, akin to \citet{masry1992gaussian}. We showcase this latter technique on the uniform location model as the source, with $\Theta = [-1/2,1/2]$.

\subsection{Laplace location source experiment}

Suppose the source statistical model $\mathcal{U}(\cdot | \theta)$ is given by a $d$-dimensional Laplace distribution
\begin{align} \label{eq:Laplace-location}
u(x; \theta) = \prod_{j = 1}^d \frac{1}{2b_{j}} \exp \left\{ - \frac{| x_j - \theta_j | }{b_{j}} \right\},
\end{align} 
so that the source sample space is given by $\mathbb{X} = \real^d$. Here $\theta_j$, $x_j$ denote the $j$-th entry of $\theta, x\in \real^d$. Consider any $\Theta \subseteq \real^d$; our goal is to map to a general statistical model $\mathcal{V}$ defined on the sample space $\mathbb{Y}$.  

Recall that we interpret $\left(I - b_{1}^{2}\nabla^{(2)}_{\theta_1} \right) \circ \cdots \circ \left(I - b_{d}^{2}\nabla^{(2)}_{\theta_d} \right)$ as a composition of operators, where $\theta_{j}$ denotes the $j$-th entry of $\theta \in \real^{d}$. Let $v(\cdot;\theta)$ be the target model. For each $1\leq k \leq d$ and $y\in \mathbb{Y}$, we define a function $v_{k}(y;\cdot): \real^{d} \rightarrow \real$ via
\begin{subequations}\label{eq:v-funcs-laplace}
\begin{align}
	&v_{k}(y;x) = \left( \left(I - b_{k+1}^{2}\nabla^{(2)}_{\theta_{k+1}} \right) \circ \cdots \circ \left(I - b_{d}^{2}\nabla^{(2)}_{\theta_d} \right) \; v(y; \theta) \right) \Big \vert_{\theta = x} \text{ for all} \quad  x \in \real^{d}\\
	& \text{ with the convention that } \quad v_{d}(y;x) = v(y;\theta) \mid_{\theta = x} \quad \text{for all} \quad  x \in \real^{d}.
\end{align} 
\end{subequations}
For each $1\leq k \leq d$, $y\in \mathbb{Y}$, and $x \in \real^{d}$, we assume that
\begin{subequations}\label{assump-target-laplace}
\begin{align}\label{assump1-target-laplace}
	&\lim_{x_{k} \rightarrow -\infty} \exp \{x_{k}/b_k \} \cdot v_{k}(y;x) = 0; \qquad  &&\lim_{x_{k} \rightarrow -\infty} \exp \{x_{k}/b_k \} \cdot \frac{\partial v_{k}(y;x)}{\partial x_{k}} = 0;  \\
	&\lim_{x_{k} \rightarrow +\infty} \exp \{-x_{k}/b_k \} \cdot v_{k}(y;x) = 0; \qquad  &&\lim_{x_{k} \rightarrow +\infty} \exp \{-x_{k}/b_{k} \} \cdot \frac{\partial v_{k}(y;x)}{\partial x_{k}} = 0. \label{assump2-target-laplace}
\end{align}
\end{subequations}
Note that we only require pointwise convergence (not uniform convergence) in Eq.~\eqref{assump-target-laplace} and there exists target density $v(y|\theta)$ satisfying Eq.~\eqref{assump-target-laplace}, e.g., $v(y|\theta) = \NORMAL(\theta,\sigma^2 I)$.
We are now ready to state the result.
\begin{proposition} \label{prop:Laplace}
Consider the source model $u(\cdot; \theta)$~\eqref{eq:Laplace-location} and a general target model $v(\cdot; \theta)$ satisfying the assumptions given in Eq.~\eqref{eq:v-funcs-laplace}-\eqref{assump-target-laplace}. Then the signed kernel
\begin{align} \label{S-star-Laplace}
\mathcal{S}^*(y | x ) =  \left( \left(I - b_{1}^{2}\nabla^{(2)}_{\theta_1} \right) \circ \cdots \circ \left(I - b_{d}^{2}\nabla^{(2)}_{\theta_d} \right) \; v(y ; \theta) \right)\Big \vert_{\theta = x}
\end{align}
attains $\SignedDef(\mathcal{U}, \mathcal{V}; \mathcal{S}^*) = 0$ for any arbitrary parameter set $\Theta \subseteq \real^d$.
\end{proposition}

The proof can be found in Section~\ref{sec:pf-propLap}. We first provide a short proof sketch under stronger conditions required to apply the Fourier inversion theorem, and 
then prove the result under the weaker conditions~\eqref{assump-target-laplace} via a delicate integration by parts argument. 
Condition~\eqref{assump-target-laplace} places various differentiability assumptions on $v(y; \cdot)$ and also quantitative conditions on its behavior at (plus and minus) infinity. It disallows any function for which $v_k(y; \cdot)$ or its partial derivative with respect a particular coordinate grows super-exponentially in the value of that coordinate. Operationally, such a bad situation will only occur if the density is parameterized in some ``unstable'' way, which will not be the case in most applications of interest. Aside from Eq.~\eqref{assump-target-laplace}, note that in defining the signed kernel $\mathcal{S}^*$, we have implicitly assumed that the target experiment is twice differentiable in each entry $\theta_j$. Besides these requirements, there are no other constraints---the signed kernel attains zero deficiency provided Eq.~\eqref{assump-target-laplace} holds and the RHS of Eq.~\eqref{S-star-Laplace} is well-defined.
In order for this signed kernel to be useful however, we need to bound its associated functionals $p(x)$ and $q(x)$, and show that they are close to $1$ and $0$, respectively (see Lemma~\ref{lem:rej-sampling}).

\subsubsection{A deficiency bound for Gaussian target families}

Proposition~\ref{prop:Laplace} is stated for the multivariate Laplace source and general targets, but for simplicity we illustrate it for mapping from a univariate Laplace location model $\mathsf{Lap}(\theta,b)$, i.e., with $d = 1$. In this case, the signed kernel~\eqref{S-star-Laplace} takes the simplified form
\begin{align} \label{S-star-Laplace-scalar} 
\mathcal{S}^*(y | x) = \left( v(y ; \theta) - b^2\nabla^{(2)}_{\theta} v(y ; \theta) \right)\mid_{\theta = x}
\end{align}
and yields the following theorem for a univariate Gaussian location target model.

\begin{theorem} \label{thm:Lap-Gaussian}
Suppose the source experiment $\mathcal{U}(\cdot | \theta)$ is given by the Laplace location model $\mathsf{Lap}(\theta,b)$ and the target experiment $v( \cdot ; \theta)$ is given by the pdf of a normal distribution $\NORMAL(\theta, \sigma^2)$ for some $\sigma \geq b$. Choose $\mathcal{P}(\cdot | x)$ to be the pdf of the normal distribution $\NORMAL(x, \sigma^2)$ for each $x \in \real$. Then for each $\epsilon \in (0,1)$ and $\theta \in \real$, Algorithm $\textsc{rk}$ using $\mathcal{S}^*$~\eqref{S-star-Laplace-scalar}, with the parameter settings $M = 2$ and $N \geq 2 \log (4/\epsilon)$ and any $y_0 \in \real$ satisfies
\begin{align} \label{eq:Lap-Gaussian-TV}
\| \mathcal{L} \left[ \textsc{rk}(X_{\theta}, N, M, y_0) \right] - \NORMAL(\theta, \sigma^2) \|_{\mathsf{TV}} \leq \frac{\epsilon}{2} + 6 \exp \left\{ - \frac{\sigma^2}{2b^2} \right\}.
\end{align}

Consequently, for any $\Theta \subseteq \real$ and provided $\sigma^2 \geq 2b^2 \log \left( \frac{12}{\epsilon} \right)$, the reduction algorithm $\textsc{rk}$ runs in time $\mathcal{O}(\log(4/\epsilon)(T_{samp} + T_{eval}))$ and satisfies
\begin{align} \label{eq:final-Lap}
\delta(\mathcal{U}, \mathcal{V}; \textsc{rk}) \leq \epsilon.
\end{align}
\end{theorem}

\begin{figure}[] 
	\centering
	\begin{subfigure}[b]{0.4\textwidth}
		\centering
		\includegraphics[width=\textwidth]{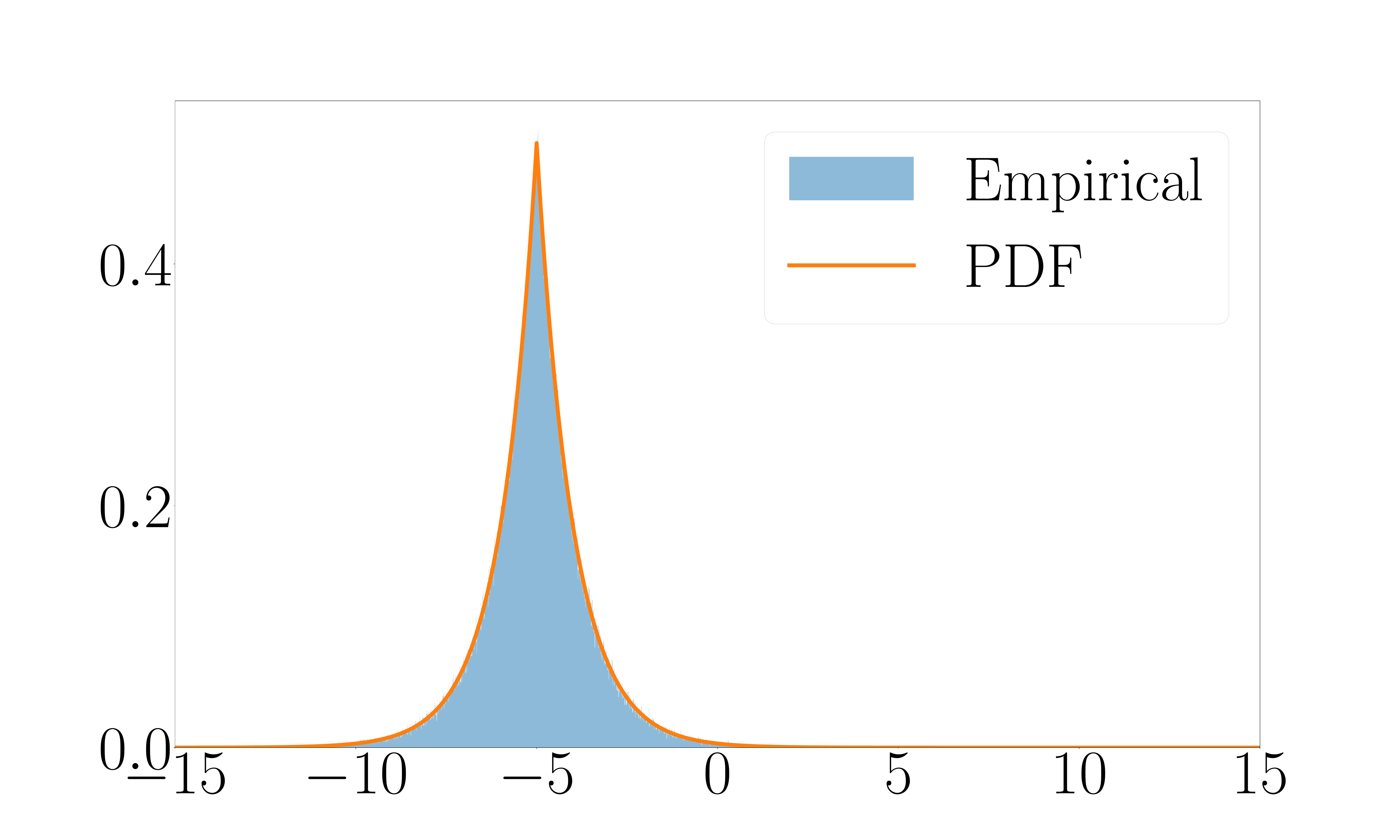}
		\caption{Source pdf and histogram, $\theta = -5$} 
		\label{fig:lap-gaussian1}
	\end{subfigure}
	\hfill
	\begin{subfigure}[b]{0.4\textwidth}  
		\centering 
		\includegraphics[width=\textwidth]{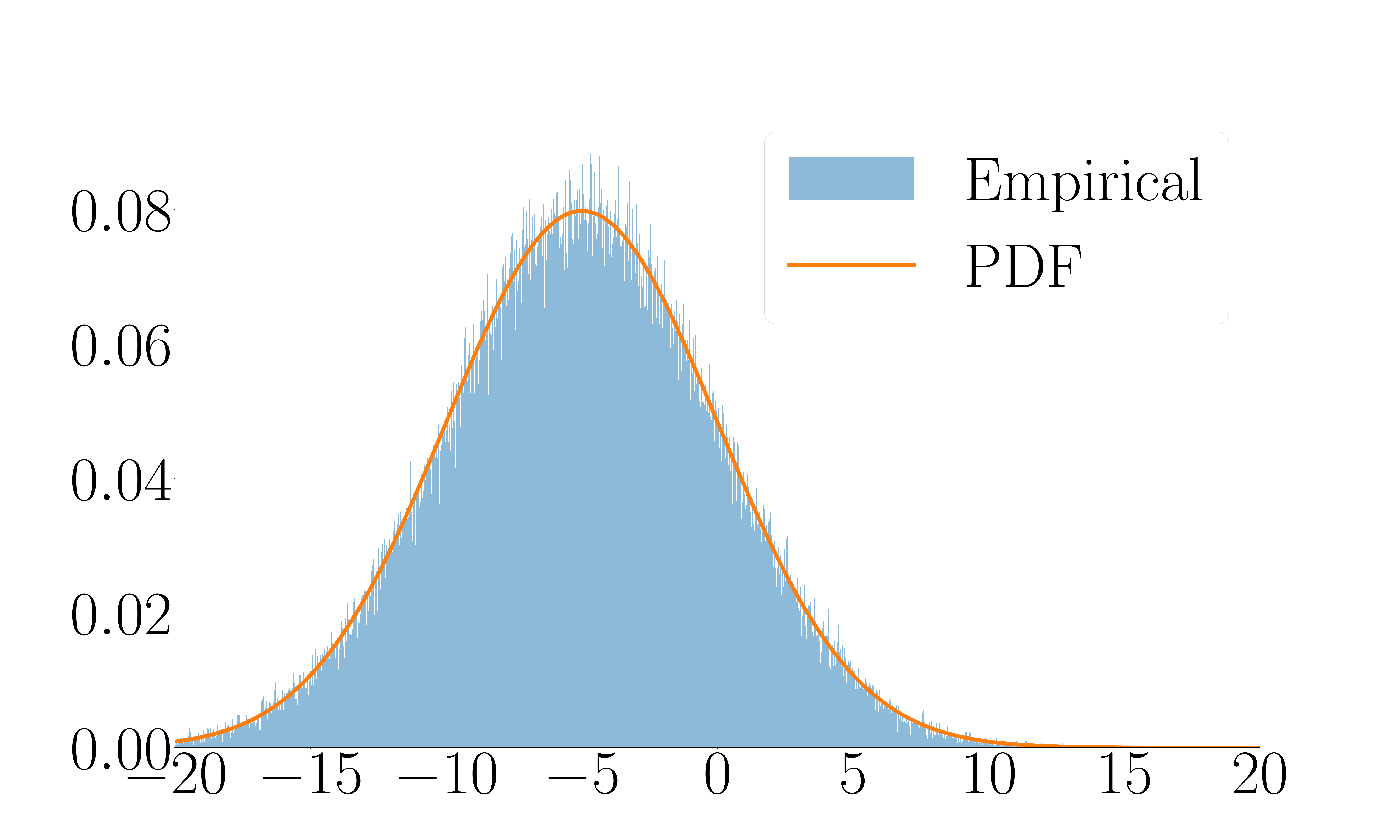}
		\caption{Target pdf and histogram, $\theta = -5$}
		\label{fig:lap-gaussian2}
	\end{subfigure}
	\vskip\baselineskip
	\vspace{-0.56cm}
	\begin{subfigure}[b]{0.4\textwidth}
		\centering
		\includegraphics[width=\textwidth]{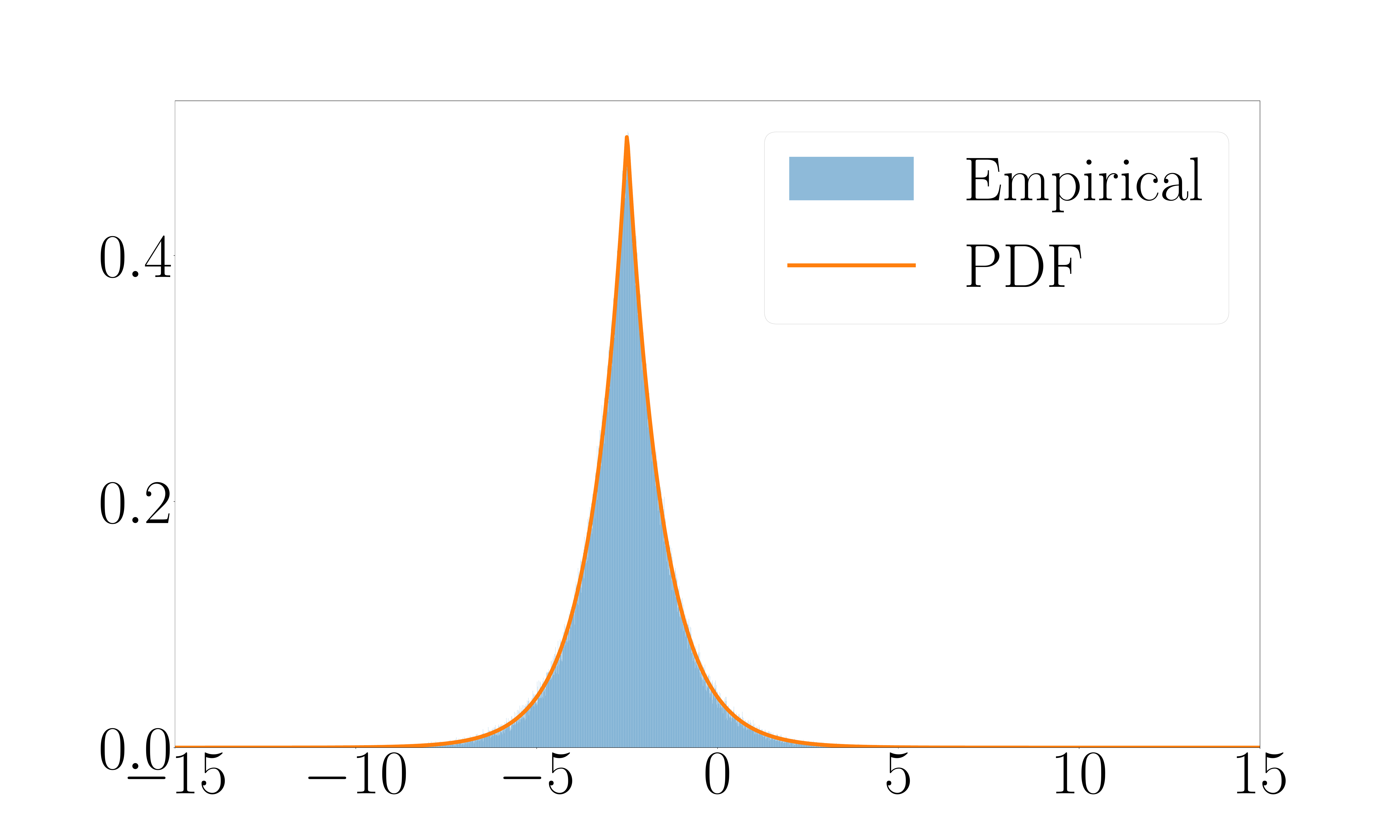}
		\caption{Source pdf and histogram, $\theta = -2.5$} 
		\label{fig:lap-gaussian3}
	\end{subfigure}
	\hfill
	\begin{subfigure}[b]{0.4\textwidth}  
		\centering 
		\includegraphics[width=\textwidth]{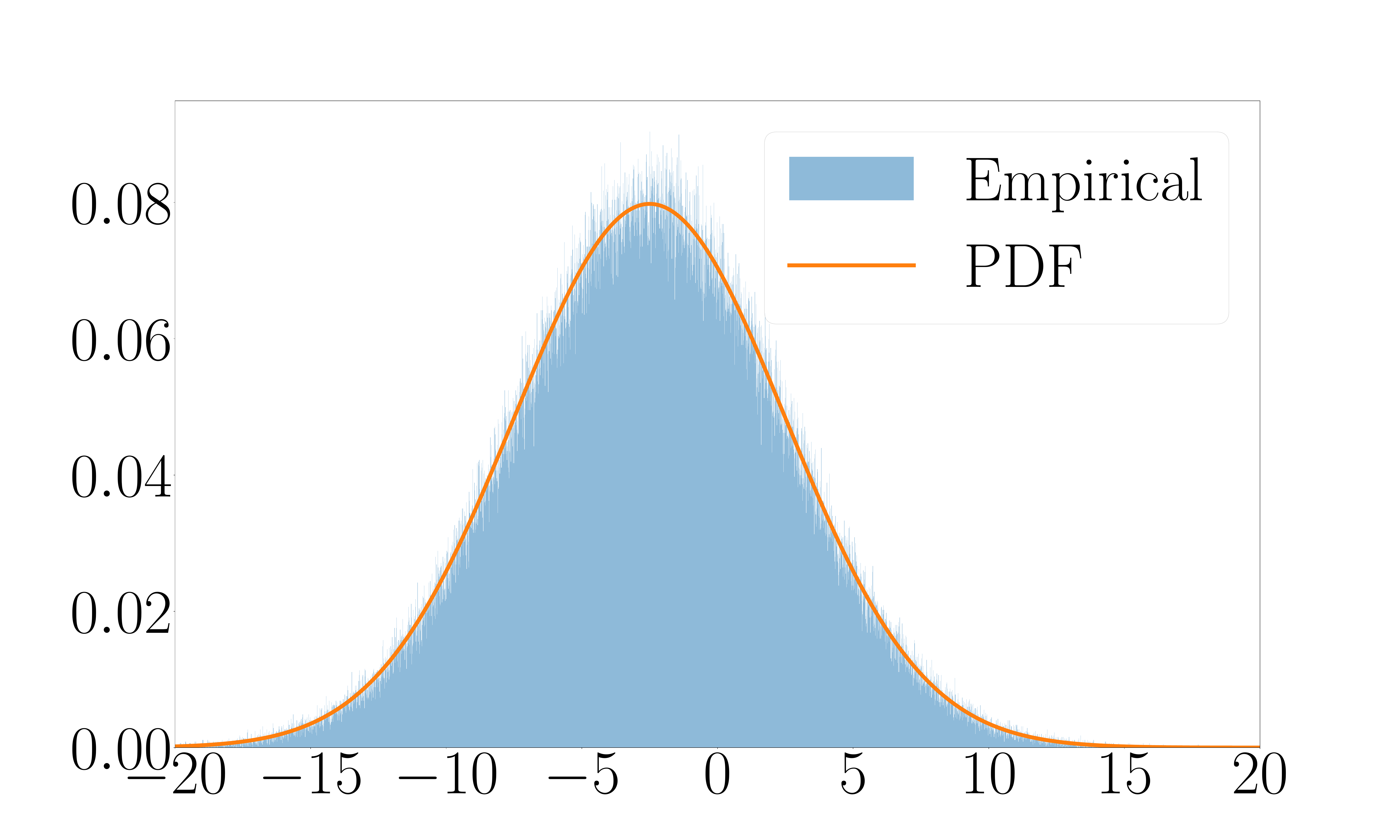}
		\caption{Target pdf and histogram, $\theta = -2.5$}
		\label{fig:lap-gaussian4}
	\end{subfigure}
	\vskip\baselineskip
	\vspace{-0.56cm}
	\begin{subfigure}[b]{0.4\textwidth}
		\centering
		\includegraphics[width=\textwidth]{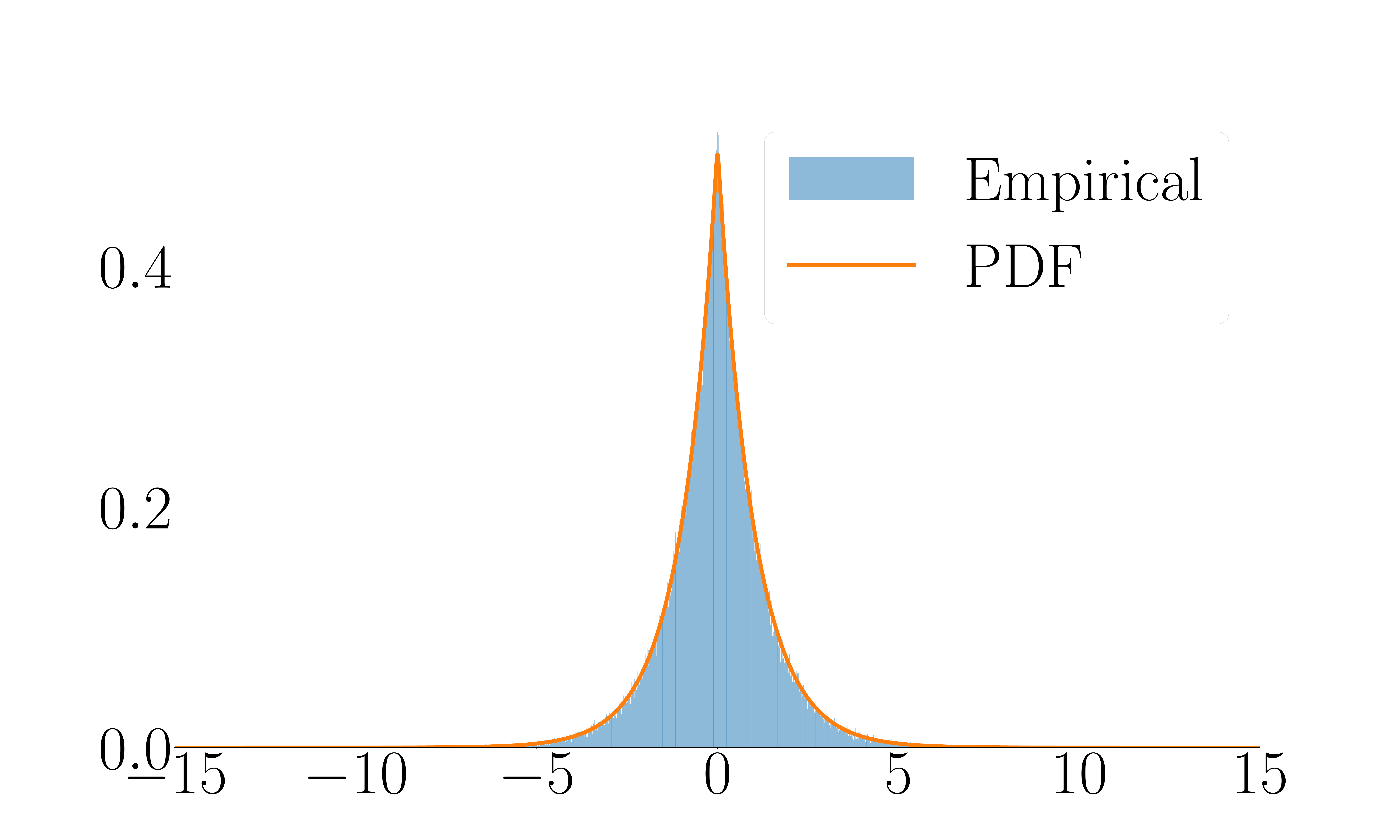}
		\caption{Source pdf and histogram, $\theta = 0$} 
		\label{fig:lap-gaussian5}
	\end{subfigure}
	\hfill
	\begin{subfigure}[b]{0.4\textwidth}  
		\centering 
		\includegraphics[width=\textwidth]{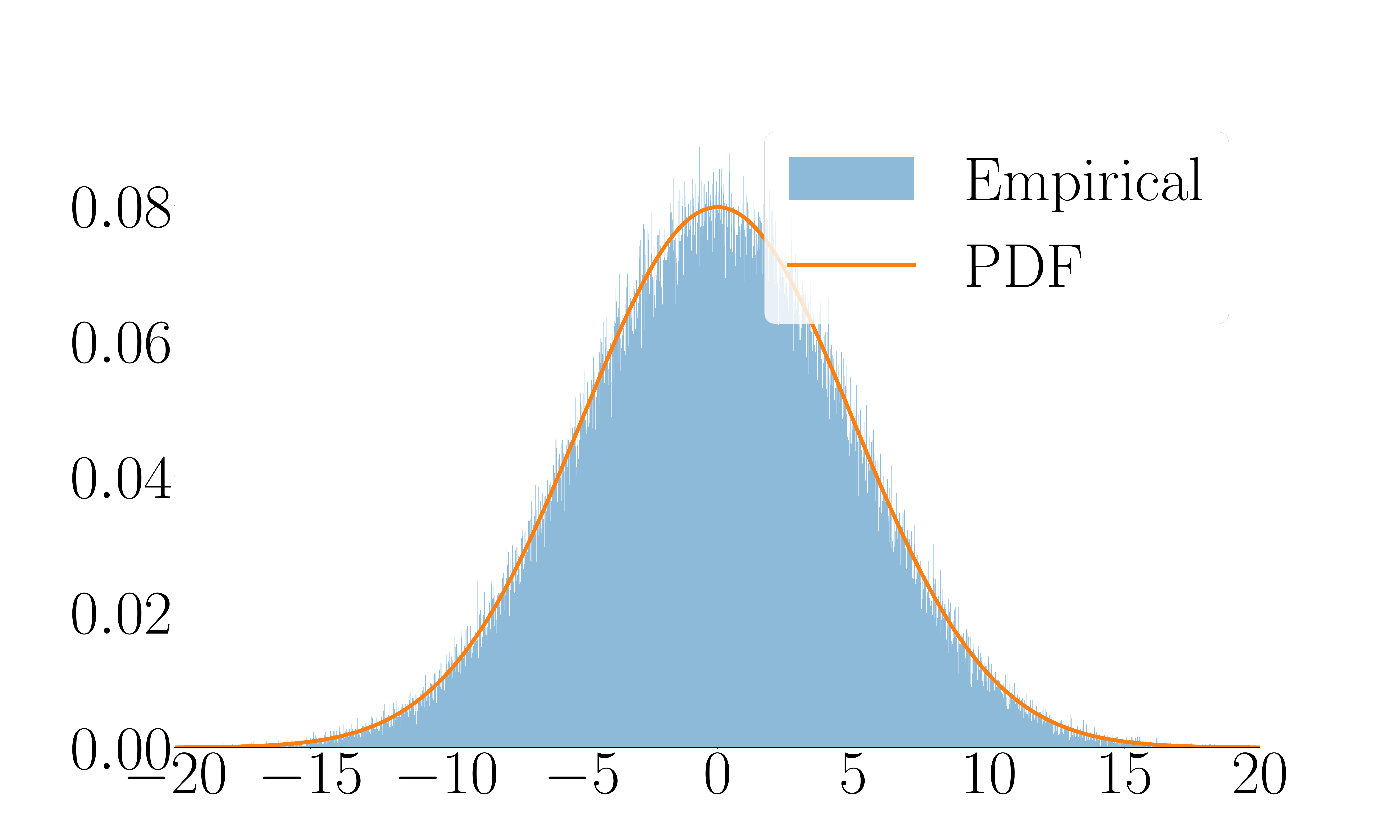}
		\caption{Target pdf and histogram, $\theta = 0$}
		\label{fig:lap-gaussian6}
	\end{subfigure}
	\vskip\baselineskip
	\vspace{-0.56cm}
	\begin{subfigure}[b]{0.4\textwidth}
		\centering
		\includegraphics[width=\textwidth]{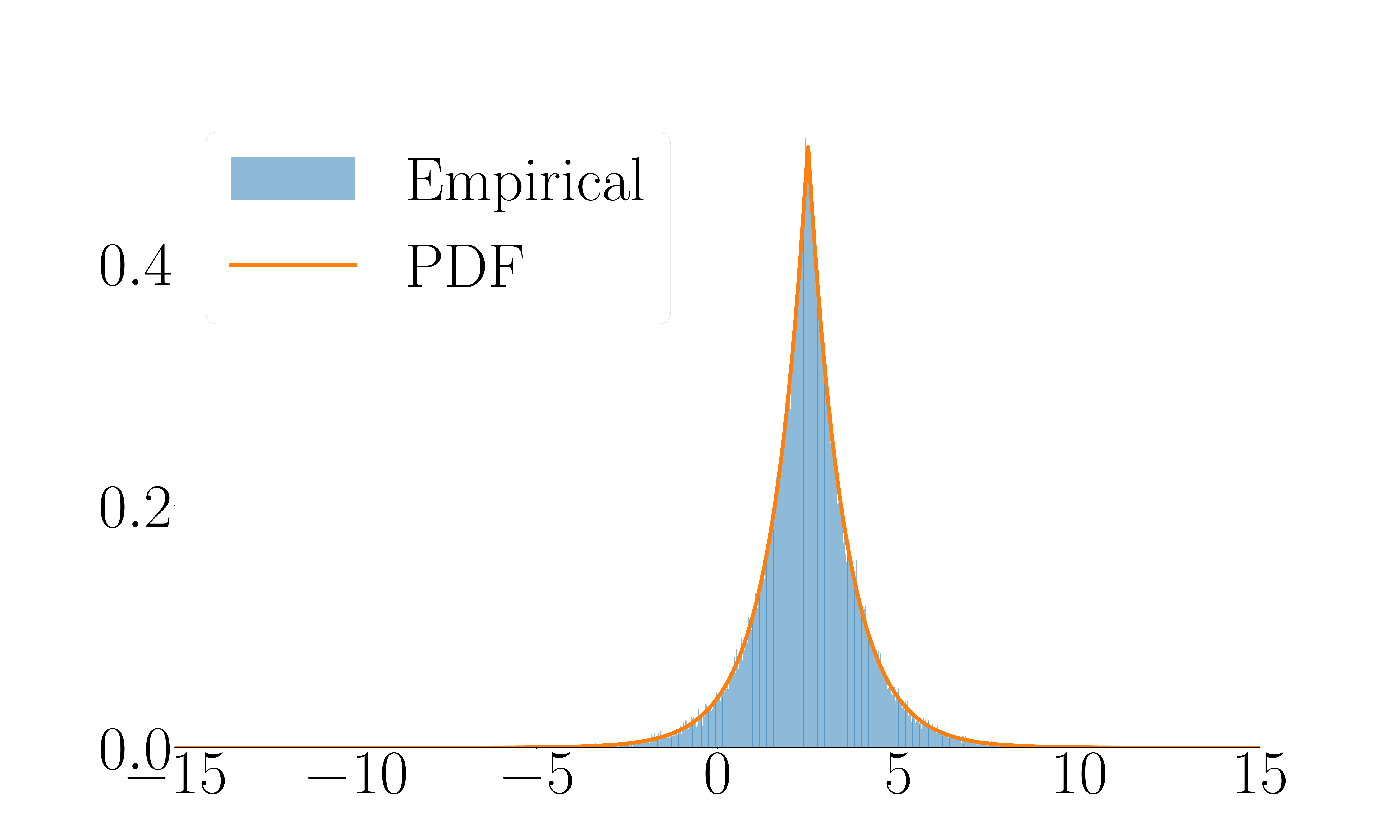}
		\caption{Source pdf and histogram, $\theta = 2.5$} 
		\label{fig:lap-gaussian7}
	\end{subfigure}
	\hfill
	\begin{subfigure}[b]{0.4\textwidth}  
		\centering 
		\includegraphics[width=\textwidth]{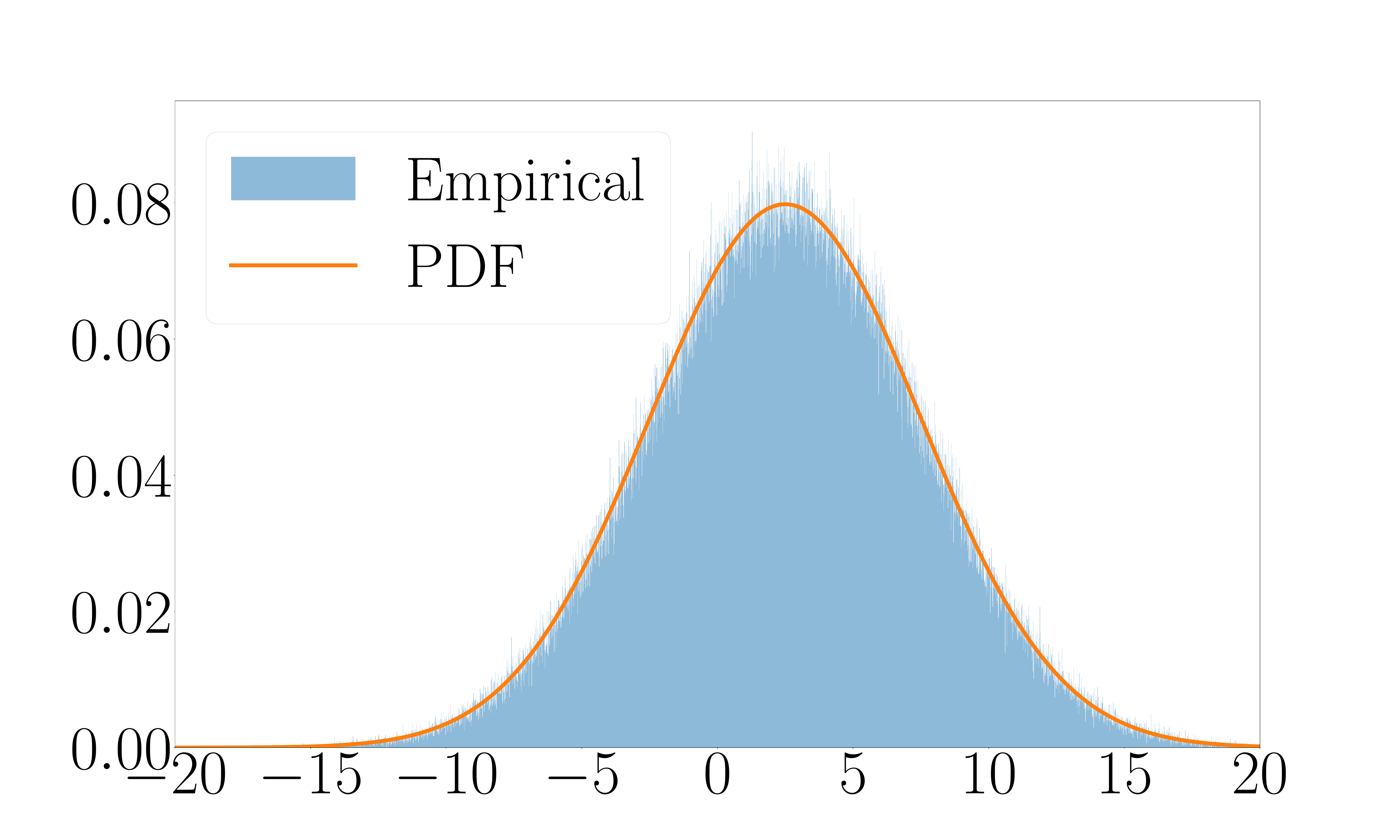}
		\caption{Target pdf and histogram, $\theta = 2.5$}
		\label{fig:lap-gaussian8}
	\end{subfigure}
	\vskip\baselineskip
	\vspace{-0.56cm}
	\begin{subfigure}[b]{0.4\textwidth}
		\centering
		\includegraphics[width=\textwidth]{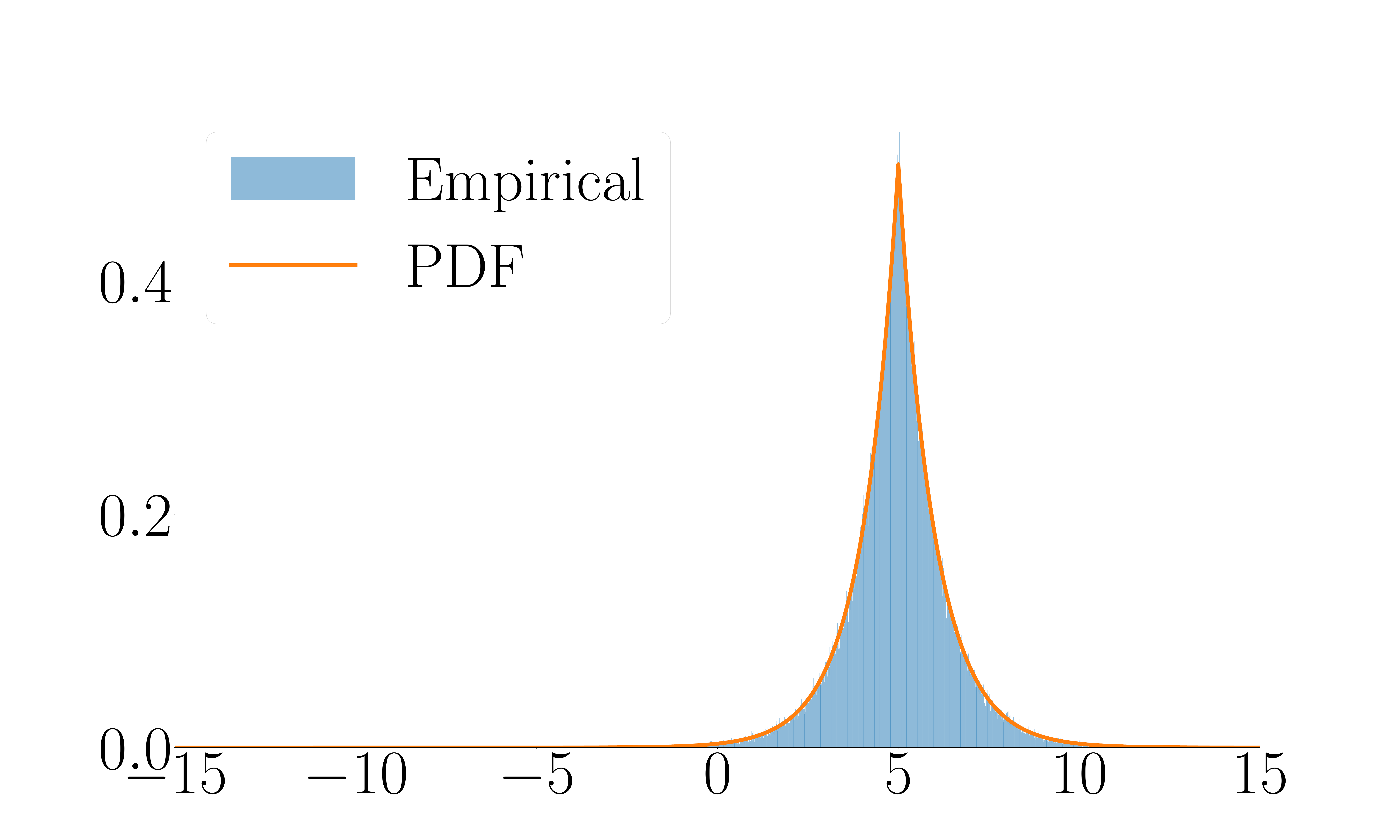}
		\caption{Source pdf and histogram, $\theta = 5$} 
		\label{fig:lap-gaussian9}
	\end{subfigure}
	\hfill
	\begin{subfigure}[b]{0.4\textwidth}  
		\centering 
		\includegraphics[width=\textwidth]{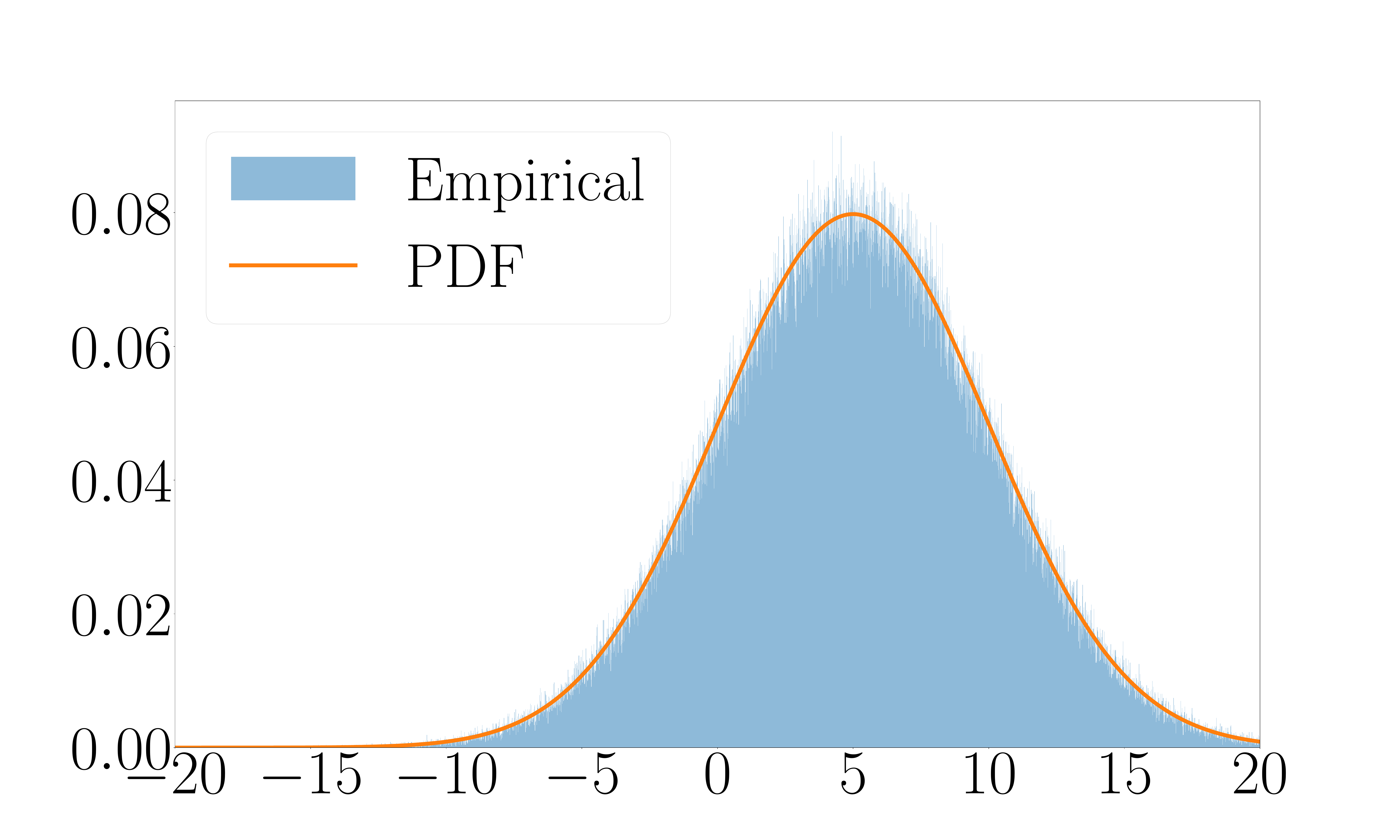}
		\caption{Target pdf and histogram, $\theta = 5$}
		\label{fig:lap-gaussian10}
	\end{subfigure}
	\caption{Simulation of $\mathsf{Lap}(\theta, 1)$ to $\mathcal{N}(\theta, \sigma^2)$ reduction, plotted for $\theta \in \{-5, -2.5, 0, 2.5, 5\}$ and $\sigma = 5$.}
	\label{fig:lap-gaussian}
\end{figure}

The proof of Theorem~\ref{thm:Lap-Gaussian} can be found in Section~\ref{sec:pf-thmLapGauss}. A few comments are in order. First, in order to satisfy the total variation guarantee in Eq.~\eqref{eq:final-Lap}, the variance of the Gaussian must grow logarithmically in $1/\epsilon$, i.e., we only have a reduction to a Gaussian experiment whose variance is $\sigma^2 = \Omega\left(\log(1/\epsilon)\right)$. Furthermore, the amount of computational resource required by the algorithm grows logarithmically in $1/\epsilon$. In that sense, there is a cost associated with transforming a Laplace experiment into a Gaussian one. On the other hand, this is only a logarithmic inflation of both the variance and running time, and so in both statistical and computational terms, this blow-up should not affect the efficiency of any procedure significantly. Contrast this with the plug-in reduction, c.f. Proposition~\ref{prop:Laplace}. Second, note that the base measure $\mathcal{P}(\cdot | x)$ is simply a Gaussian measure centered at $x$, which is very easy to obtain a sample from. We also see that Assumption~\ref{assptn:RND} is satisfied with $M = 2$, and Algorithm~\textsc{RK} thus runs very efficiently.
Third, we note that while showing that Laplace experiments are $\epsilon$-deficient relative to appropriately defined Gaussian experiments is of conceptual interest in its own right, Theorem~\ref{thm:Lap-Gaussian} can be potentially applied to differential privacy---we present this concrete consequence in Corollary~\ref{cor:privacy} to follow.
Finally, note that a multivariate version of the reduction can be constructed via similar techniques by using the signed kernel in Eq.~\eqref{S-star-Laplace}.


Figure~\ref{fig:lap-gaussian} provides a numerical illustration of the reduction. We set parameters $M = 2$, $N = 20$, $\sigma = 5$, $y_{0} = 0$, and set the base measure $\mathcal{P}(y | x)$ and the kernel $\mathcal{S}^{*}(y|x)$ as in Theorem~\ref{thm:Lap-Gaussian}. For each $\theta \in \{-5,-2.5,0,2.5,5\}$, we generate $K = 5\times 10^{5}$ source samples $\{X^{i}\}_{i=1}^{K} \overset{\mathsf{i.i.d.}}{\sim} \mathsf{Lap}(\theta,1)$ and generate target samples $Y^{i} = \textsc{rk}(X^{i}, N, M, y_0)$ for each $i\in \{1,2,\dots,K\}$. We emphasize that algorithm $\textsc{rk}$ has no access to the true parameter $\theta$ and that a large $K$ is chosen so as to plot a reasonable histogram (even though Theorem~\ref{thm:Lap-Gaussian} holds for $K  = 1$). In Figure~\ref{fig:lap-gaussian}, we plot the histograms of $\{X^{i}\}_{i=1}^{K}$ and $\{Y^{i}\}_{i=1}^{K}$ in color blue and plot the probability density functions in color orange. We see that the target distribution almost follows the pdf $\NORMAL(\theta, \sigma^2)$, i.e., the mean $\theta$ is perserved.

\subsection{Erlang location source experiment} 

An Erlang distribution with parameters $(\lambda, k)$ models the time taken for $k$ events of a Poisson process with rate $\lambda$, i.e., it is the distribution of the sum of $k$ i.i.d. $\mathsf{Exp}(\lambda)$ random variables. Consider the shifted $(1, k)$ Erlang with a mean parameter, i.e., a location model, with density
\begin{align}\label{eq:Erlang-location}
u(x; \theta) = \frac{\lambda^k}{(k - 1)!} \cdot (x - \theta)^{k - 1} e^{-\lambda (x - \theta)}.
\end{align}
A natural question is whether such a family of distribution with unknown mean $\theta$ can be mapped to a general target model. For example, the centered exponential $Y = X - 1$ for $X \sim \mathsf{Exp}(1)$ is a canonical zero-mean, unit variance, log-concave random variable that has been studied extensively in nonparametric statistics~\citep{dumbgen2011approximation}, and it is natural to ask if its location model can be transformed to another target model.

Once again, we will consider the parameter set $\Theta = \real$, so that the source sample space is given by $\mathbb{X} = \real$. Suppose that our goal is to map to a general statistical model $\mathcal{V}$ having output sample space $\mathbb{Y}$. As before, we require a technical assumption. Suppose that for each $y \in \mathbb{Y}$, $w \in \real$, and $0\leq j \leq k-1$, 
\begin{subequations}\label{eq:assump-prop-Erlang}
\begin{align}
	\label{eq1:assump-prop-Erlang}
	&\lim_{x \rightarrow +\infty}\; x^{k-1} e^{-\lambda x} \cdot \nabla_{\theta}^{(j)} v(y;\theta) \big \vert_{\theta=x+w} = 0,
\end{align}
and additionally, if $k \geq 2$, assume that
\begin{align}
	\label{eq2:assump-prop-Erlang}
	&\lim_{x \rightarrow 0}\; x \cdot v(y;\theta) \big \vert_{\theta=x + w} = 0, \\
	\label{eq3:assump-prop-Erlang}
	&\lim_{x \rightarrow 0}\; x^{k-1} \cdot \nabla_{\theta}^{(j)} v(y;\theta) \big \vert_{\theta=x + w} = 0.
\end{align}
\end{subequations}
Note that we only require pointwise convergence (not uniform convergence) in Eq.~\eqref{eq:assump-prop-Erlang} and there exists target density $v(y|\theta)$ satisfying Eq.~\eqref{eq:assump-prop-Erlang}, e.g., $v(y|\theta) = \NORMAL(\theta,\sigma^2)$. Then the following result holds.

\begin{proposition} \label{prop:Erlang}
Consider the Erlang source model $u(\cdot;\theta)$~\eqref{eq:Erlang-location} and a general target model $v(\cdot; \theta)$ satisfying the assumption given in Eq.\eqref{eq:assump-prop-Erlang}.
Then the signed kernel
\begin{align}\label{eq:signed-kernel-Erlang}
\mathcal{S}^*(y | x ) =  \sum_{j = 0}^k \binom{k}{j} \lambda^{-j} \cdot \nabla^{(j)}_t v(y; -t) \Big|_{t = -x} 
\end{align}
attains $\SignedDef(\mathcal{U}, \mathcal{V}; \mathcal{S}^*) = 0$ for any arbitrary parameter set $\Theta \subseteq \real$.
\end{proposition}

See Section~\ref{sec:pf-Erlang} for the proof of Proposition~\ref{prop:Erlang}. As before, we present both a Fourier-theoretic proof sketch
as well as a rigorous proof under the differentiability conditions~\eqref{eq:assump-prop-Erlang} on the target experiment. Roughly speaking, if the source is an Erlang $(1, k)$ distribution, then we require the target experiment to be $k$ times differentiable in $\theta$, with bounded growth properties at zero and infinity. Below, we present an example reduction when $k = 1$, i.e., for the exponential distribution.

\subsubsection{Deficiency between an exponential source and log-concave target}

As a special case of Proposition~\ref{prop:Erlang}, suppose we are interested in mapping the exponential location model to a general log-concave target location model. In particular, suppose the source experiment takes the form of an $\mathsf{Exp}(1)$ random variable with unknown shift $\theta$, having density
\begin{align}\label{eq:exponential-location}
	u(x;\theta) = \begin{cases} \exp(-(x-\theta+1)) &\text{for} \quad x \geq \theta-1, 
								\\ 0 &\text{for} \quad x<\theta-1. \end{cases}
\end{align}
Note that the source distribution has unknown mean $\theta$ and unit variance.

To define the target experiment, let $\psi:\real \rightarrow \real$ be a convex and differentiable function that satisfies
\begin{subequations} \label{eq:exp-target}
\begin{align} \label{eq:psi-density-condition}
	\int_{\real} e^{-\psi(z)} \mathrm{d}z = 1, \quad \int_{\real} z\cdot e^{-\psi(z)} \mathrm{d}z = 0 \quad \text{and} \quad \int_{\real}z^{2} \cdot e^{-\psi(z)} \mathrm{d}z = 1.
\end{align}
Suppose the target experiment is given by 
\begin{align}
    v(y ; \theta) = \frac{1}{\sigma} \exp\left\{ - \psi\Big(\frac{y-\theta}{\sigma} \Big) \right\} \text{ for some } \sigma \geq 1. 
\end{align}
\end{subequations}
\begin{subequations} \label{eq:functionals}
Also define the functionals 
\begin{align} \label{eq:kappa-tau}
    \kappa(\sigma) &:= \begin{cases} \inf\{z\in \real: \psi'(z) = \sigma \}, &\text{if there exists }  z \in \real \text{ such that } \psi'(z) = \sigma 
    \\ + \infty, & \text{otherwise, and}\end{cases} \\
    \tau(\sigma) &:= 2\int_{\kappa(\sigma)}^{+\infty}e^{-\psi(z)} \mathrm{d}z + \frac{2}{\sigma}e^{ -\psi(\kappa(\sigma))}. \label{eq:functional2}
\end{align}
\end{subequations}

Applying Proposition~\ref{prop:Erlang}, we see that the following signed kernel attains $\SignedDef(\mathcal{U}, \mathcal{V}; \mathcal{S}^*) = 0$:
\begin{align}\label{eq:signed-kernel-Exp}
\mathcal{S}^{*}(y|x) =  v(y;x+1) - \nabla_{\theta}v(y;\theta) \mid_{\theta = x+1}.
\end{align}
\begin{theorem}\label{thm:exp-log-concave}
Suppose the source experiment $\mathcal{U}(\cdot | \theta)$ is given by the exponential location model~\eqref{eq:exponential-location} and the target experiment by the log-concave location model~\eqref{eq:exp-target}.
Choose the family of base measures 
\begin{align}\label{eq:base-measure-log-concave}
\mathcal{P}(y|x) = \frac{1}{2\sigma} \exp\left\{ -\psi\Big( \frac{y-x-1}{2\sigma} \Big) \right\}.
\end{align}
Then for each $\sigma>0$ and $\theta \in \real$, Algorithm $\textsc{rk}$ using $\mathcal{S}^*$~\eqref{eq:signed-kernel-Exp}, with any $y_0 \in \real$,  $N \geq 1$ and
\begin{align} \label{ineq:M-lower-bound-exp}
M \geq 2 \exp\left\{ \frac{\psi(0)}{2}\right\} \cdot \sup_{z \leq \kappa(\sigma)} \exp\left\{-\frac{\psi(z)}{2} \right\} \Big(1- \frac{\psi'(z)}{\sigma} \Big)
\end{align}
satisfies
\begin{align} \label{ineq:exp-log-concave}
\left\| \mathcal{L} \left[ \textsc{RK}(X_{\theta}, N, M, y_0) \right] - v(\cdot; \theta) \right\|_{\mathsf{TV}} &\leq 2 \exp \left\{ -\frac{N}{M} \big(1-\tau(\sigma)\big) \right\} + \tau(\sigma),
\end{align}
where $\tau(\sigma)$ is defined in Eq.~\eqref{eq:functional2}.
\end{theorem}

Theorem~\ref{thm:exp-log-concave} is proved in Section~\ref{sec:pf-ExpLog}, and covers many log-concave target distributions. Operationally speaking, it suggests that the exponential distribution is ``extremal'' among log-concave distributions in some approximate sense: If we have a computationally efficient parameter estimation algorithm for any log-concave location model for which the RHS of Eq.~\eqref{ineq:exp-log-concave} is small, then this algorithm can be used to perform parameter estimation under the exponential location model (see Eq.~\eqref{eq:pointwise-risk-transfer} and Appendix~\ref{sec:loss-unbounded}). Similarly, lower bounds on the minimax risk for the exponential model imply minimax lower bounds for all log-concave measures covered by the theorem (see Eq.~\eqref{eq:minimax-transfer}). It is worth noting that extremal properties of the exponential distribution among log-concave measures have also been observed in distinct settings~\citep[see, e.g., the full version of][]{dumbgen2011approximation}.
We now present two prototypical examples of log-concave targets and explicitly evaluate the functionals $\kappa(\sigma)$ and $\tau(\sigma)$~\eqref{eq:functionals}. Proofs of the associated claims can be found in the Sections~\ref{sec:pf-example1}--\ref{sec:pf-example3}.

\begin{example}[Gaussian Target]\label{example-exp-gaussian}
Consider the function $\psi(z) = z^{2}/2 + \log(2\pi)/2$, so that the target model is given by the pdf of of a normal distribution $\NORMAL(\theta,\sigma^{2})$. Then $\kappa(\sigma) = \sigma$, $\tau(\sigma) \leq 2e^{-\sigma^{2}/2}$,
\[
	\text{and} \quad 2 \exp\left\{ \frac{\psi(0)}{2}\right\} \cdot \sup_{z \leq \kappa(\sigma)} \exp\left\{-\frac{\psi(z)}{2} \right\} \big(1-\psi'(z) /\sigma \big) \leq 4, \quad \text{ for all }  \sigma \geq 1.
\]
Thus, by letting $\sigma \geq \sqrt{2\log(4/\epsilon)}$, $M=4$, $N = 2M \log(4/\epsilon)$ and applying Theorem~\ref{thm:exp-log-concave}, we obtain that algorithm $\textsc{rk}$ runs in time $\mathcal{O}\big(\log(4/\epsilon)(T_{samp} + T_{eval})\big)$ and $\delta(\mathcal{U}, \mathcal{V}; \textsc{rk}) \leq \epsilon$.
\hfill $\clubsuit$
\end{example} 

\begin{example}[Logistic Target]\label{example-exp-logistic}
Consider the target distribution $v(y;\theta)$ to be the logistic distribution with mean $\theta$ and variance $\sigma^2$, i.e., 
\begin{align}\label{eq:psi-logistic}
	\psi(z) = \frac{\pi z}{\sqrt{3}} + 2\log \Big( 1 + e^{-\frac{\pi z}{\sqrt{3}} } \Big) - \log\big(\pi/\sqrt{3}\big) \quad \text{and} \quad v(y;\theta) = \frac{\pi}{\sqrt{3}\sigma} \cdot \frac{ e^{-\frac{\pi(y-\theta)}{\sqrt{3}\sigma}}}{ \Big( 1 + e^{-\frac{\pi(y-\theta)}{\sqrt{3}\sigma}} \Big)^{2}}.
\end{align}
Suppose $\sigma \geq \pi/\sqrt{3}$. Then $\kappa_{\sigma} = +\infty$, $\tau_{\sigma} = 0$, and 
\[
	2 \exp\left\{ \frac{\psi(0)}{2}\right\} \cdot \sup_{z \leq \kappa(\sigma)} \exp\left\{-\frac{\psi(z)}{2} \right\} \big(1-\psi'(z) /\sigma \big) \leq 4.
\]
Thus, by letting $\sigma \geq \pi/\sqrt{3}$, $M=4$, $N = M \log(2/\epsilon)$ and applying Theorem~\ref{thm:exp-log-concave}, we obtain that algorithm $\textsc{rk}$ runs in time $\mathcal{O}\big(\log(2/\epsilon)(T_{samp} + T_{eval}) \big)$ and $\delta(\mathcal{U}, \mathcal{V}; \textsc{rk}) \leq \epsilon$.
\hfill $\clubsuit$
\end{example}

\begin{remark}[Exact comparison of experiments: Exponential to logistic]
In Theorem~\ref{thm:Lap-Gaussian} and Example~\ref{example-exp-gaussian}, the TV deficiency parameter $\epsilon$ controls both the scale parameter $\sigma$ of the target and the running time of Algorithm~\textsc{rk}. However, in Example~\ref{example-exp-logistic}, the parameter $\epsilon$ only influences the running time through the iteration count $N$, \emph{but not the scale parameter}. In particular, provided $\sigma \geq \pi/\sqrt{3}$, one can take $\epsilon \downarrow 0$ to conclude that there exists a perfect Markov kernel (i.e., $\mathcal{S}^*$~\eqref{eq:signed-kernel-Exp} is in itself a Markov kernel having deficiency zero) from the exponential source location model to the logistic target model with scale $\sigma$, thereby providing an exact comparison in the sense of~\citet{blackwell1951comparison}.
\end{remark}
At first glance, Theorem~\ref{thm:exp-log-concave} appears to require that $\psi$ is differentiable on $\real$. However, such a condition can be weakened via a mollification technique, which we now showcase when the target is the Laplace model
\begin{align}\label{eq:laplace-target}
	v(y;\theta) = \frac{1}{2\sigma} \exp\left\{ - \psi\Big( \frac{y-\theta}{\sigma} \Big) \right\}, \quad \text{where} \quad \psi(z) = |z|,\; \text{for all} \; z\in\real.
\end{align}
Clearly, $\psi(z)$ is non-differentiable at $z=0$. Recall that $\phi_{\eta}(\cdot)$ is the pdf of a Gaussian random variable with zero mean and variance $\eta^2$. Let $Y \sim \mathsf{Laplace}(0,1)$ and $G \sim \NORMAL(0,1)$ be independent. For each $\eta \in (0,1)$, the random variable $Y+\eta G$ is log-concave (see, e.g.~\citet{saumard2014log}) and has probability density function
\begin{align}\label{eq:psi-epsilon}
	f_{Y+\eta G}(z) = e^{ - \psi_{\eta}(z)}, \quad \text{where} \quad \psi_{\eta}(z) = -\log \Big( \int_{y \in \real} \phi_{\eta}(z-y) \cdot \frac{1}{2}e^{-|y|} \mathrm{d}y \Big).
\end{align}
Note that $\psi_{\eta}(z)$ is differentiable for all $z \in \real$, and define a surrogate target model
\begin{align}\label{eq:smooth-laplace-target}
	v_{\eta}(y;\theta) = \frac{1}{\sigma} \exp\left\{ -\psi_{\eta} \Big( \frac{y-\theta}{\sigma} \Big) \right\}.
\end{align}
The following result, which we prove in Section~\ref{sec:pf-example3} for completeness, shows that this surrogate model closely resembles the Laplace model, i.e.,
\begin{align}\label{ineq:approx-error}
	\big\| v(\cdot;\theta) - v_{\eta}(\cdot;\theta) \big\|_{\mathsf{TV}} \leq \frac{\eta}{2} \quad \text{for all} \quad \eta \in (0,1).
\end{align}
Motivated by this observation, we map the exponential source $u(\cdot;\theta)$~\eqref{eq:exp-target} to target $v_{\eta}(\cdot;\theta)$~\eqref{eq:smooth-laplace-target} instead of the original Laplace target $v(\cdot;\theta)$~\eqref{eq:laplace-target}. For $\eta$ chosen small enough, we then obtain a reduction to the Laplace model.

\begin{example}\label{example-exp-laplace}
Consider $\psi_{\eta}$ in Eq.~\eqref{eq:psi-epsilon} and the target $v_{\eta}(y;\theta)$ in Eq~\eqref{eq:smooth-laplace-target}.
Choose 
\begin{align*}
	\mathcal{S}^{*}(y|x) = v_{\eta}(y;x+1) - \nabla_{\theta} v_{\eta}(y;\theta) \mid_{x+1} \quad \text{and} \quad \mathcal{P}(y|x) = \frac{1}{4\sigma} \exp\left\{ - \frac{|y-x-1|}{2\sigma} \right\}.
\end{align*}
Suppose $\sigma \geq 1$. Then $\kappa(\sigma) = +\infty$, $\tau(\sigma) = 0$, and $\frac{\mathcal{S}^{*}(y|x) \vee 0}{\mathcal{P}(y|x)} \leq 35$. Thus, for all $\theta\in \real$ and $\epsilon \in (0,1)$, the reduction algorithm $\textsc{RK}$ with the parameter settings $M=35$ and $N = 35 \log(4/\epsilon)$ with any $y_0 \in \real$ runs in time $O(\log(4/\epsilon)(T_{samp} + T_{eval}))$ and satisfies
\[
	\left\| \mathcal{L} \left[ \textsc{RK}(X_{\theta}, N, M, y_0) \right] - v_{\eta}(\cdot; \theta) \right\|_{\mathsf{TV}} \leq \frac{\epsilon}{2}.
\]
Consequently, by letting $\eta = \epsilon$, using bound~\eqref{ineq:approx-error} and applying the triangle inequality, we obtain for all $\theta\in \real$ and $\epsilon \in (0,1)$, the guarantee
\[
	\left\| \mathcal{L} \left[ \textsc{RK}(X_{\theta}, N, M, y_0) \right] - v(\cdot; \theta) \right\|_{\mathsf{TV}} \leq \epsilon,
\]
where $v(\cdot|\theta)$ is the Laplace target~\eqref{eq:laplace-target}.
\end{example}

\begin{remark}[Exact comparison of experiments: Exponential to Laplace]
As in Example~\ref{example-exp-logistic}, the TV deficiency parameter $\epsilon$ controls only the running time of Algorithm~\textsc{rk} through the iteration count $N$. Thus, provided $\sigma \geq 1$, one can take $\epsilon \downarrow 0$ to conclude that there exists a perfect Markov kernel (i.e., $\mathcal{S}^*$~\eqref{eq:signed-kernel-Exp} is in itself a Markov kernel having deficiency zero) from the exponential source location model to the Laplace target model with scale $\sigma$, thereby providing an exact comparison in the sense of~\citet{blackwell1951comparison}. 
\end{remark}

\begin{remark} \label{rem:uniform-target}
In principle, the mollification technique accommodates target log-concave densities with discontinuous $\psi$, such as the uniform location model. Suppose we carry out Gaussian mollification of the uniform---i.e., by constructing a surrogate statistical model with infinitely differentiable density by adding a Gaussian of small variance to the uniform distribution in question and considering this mollified distribution to be the target. Then carrying out our proof as in Section~\ref{sec:pf-example3}, we obtain an $\epsilon$-deficient reduction only for scale parameter $\sigma = \Omega(\epsilon^{-1})$. It is not clear if this is a limitation of our techniques, or if it fundamentally difficult to map an exponential location model to a uniform location model having scale parameter growing more slowly (say polylogarithmically) with $1/\eps$.
\end{remark}

The reductions presented in Theorems~\ref{thm:Lap-Gaussian} and~\ref{thm:exp-log-concave} relied on the Fourier technique to solve for $\mathcal{S}^*$ in Eq.~\eqref{eq:infinite-deconvolution}.
A prototypical example in which the Fourier technique does not apply is uniform location families. In the next section, we present a different technique to obtain an exact signed kernel $\mathcal{S}^*$ for this case (i.e., one satisfying $\SignedDef(\mathcal{U}, \mathcal{V}; \mathcal{S}^*)) = 0$) via successive differentiation of both sides of Eq.~\eqref{eq:infinite-deconvolution}. In the process, we will also show how one can automatically allow for some degree of non-smoothness in the target model.

\subsection{Uniform to general distribution} 

Suppose $\Theta = [-1/2, 1/2]$, and we are interested in mapping the uniform location family to a general target distribution. 
In particular, the source experiment is given by
\begin{align} \label{eq:Uniform-location}
u(x; \theta) = \ind{ \theta -1/2 \leq x \leq \theta + 1/2}
\end{align}
and the target model by some general $\mathcal{V}$ with sample space $\mathbb{Y}$.

\begin{proposition}\label{prop:uniform-S-star}
Suppose we have the above setup with parameter space $\Theta = [-1/2, 1/2]$, uniform source~\eqref{eq:Uniform-location}, and target model given by $v(\cdot; \theta)$. Further suppose that for each $y \in \mathbb{Y}$, $v(y; \theta)$ is a continuous function of $\theta$, and $\nabla_{\theta} v(y;\theta)$ exists for all $\theta \in [-1/2, \theta_{0}) \cup (\theta_{0}, 1/2]$. Let $g^+, g^-: \mathbb{Y} \to \mathbb{R}$ denote two functions satisfying
\begin{align}\label{eq:g-plus-g-minus-y}
g^+(y)\cdot(1-2\theta_{0}) + g^{-}(y) \cdot (1 + 2\theta_{0}) = 2v(y; 1/2) + 2v(y; -1/2) - 2 v(y; \theta_{0}).
\end{align}
Then the signed kernel
\begin{align}\label{S-star-uniform}
\begin{split}
	\mathcal{S}^*(y| x) = 
	\begin{cases}
	g^-(y) \quad &\text{ for all } x \leq \theta_{0} -1/2 \\
	g^+(y) - \nabla_{\theta} v(y; \theta)\mid_{\theta = x + 1/2} \quad &\text{ for all } -1/2 + \theta_{0} < x \leq 0 \\
	\nabla_{\theta} v(y; \theta)\mid_{\theta = x - 1/2} + g^-(y) \quad &\text{ for all } 0 < x < 1/2 + \theta_{0} \\
	g^+(y) \quad &\text{ for all } x \geq 1/2 + \theta_{0}
	\end{cases}
\end{split}
\end{align}
attains $\SignedDef(\mathcal{U}, \mathcal{V}; \mathcal{S}^*) = 0$.
\end{proposition}

The proof of this proposition can be found in Section~\ref{sec:pf-Unif}. As previously mentioned, the conditions required for the target model are now weaker than before. Most importantly, we allow for the function $v(y; \cdot)$ to have a point of non-differentiability at some $\theta_0$. Note that our choice of allowing just one such kink is arbitrary---when there are multiple kinks, a variant of Proposition~\ref{prop:uniform-S-star} can be proved either using the same techniques, or---in some cases---by the mollification technique introduced just before Example~\ref{example-exp-laplace}. We now present an explicit reduction from a uniform location model to a Gaussian model with a nonlinearly transformed mean.

To define the target experiment, let $f:[-1/2,1/2] \rightarrow \real$ be continuous on its domain and differentiable on $[-1/2, \theta_{0}) \cup (\theta_{0}, 1/2]$. Define the maximum magnitude of the function and its derivative as
\begin{align}\label{eq:alpha0-alpha1-f}
	\alpha_{0} := \max_{-1/2 \leq t \leq 1/2} \big|f(t)\big| \qquad \text{and} \qquad \alpha_1 := \sup_{t \in [-1/2,\theta_{0}) \cup (\theta_{0},1/2]} \big|f'(t) \big|,
\end{align}
respectively. Suppose the target experiment is given by a normal with mean $f(\theta)$, i.e.,
\begin{align}\label{eq:gaussian-target-f}
	v(y;\theta) = \frac{1}{\sqrt{2\pi} \sigma} \exp\left\{ -\frac{(y-f(\theta))^{2}}{2\sigma^2} \right\}.
\end{align}

Applying Proposition~\ref{prop:uniform-S-star}, we see that the kernel $\mathcal{S}^{*}(y|x)$~\eqref{S-star-uniform} with
\begin{align}\label{eq:unif-g-func}
	g^{+}(y) = g^{-}(y) = v(y;1/2) + v(y;-1/2) - v(y;\theta_{0})
\end{align}
attains $\SignedDef(\mathcal{U}, \mathcal{V}; \mathcal{S}^*) = 0$. 
Equipped with this notation, we are ready to state our theorem.
\begin{theorem} \label{thm:Unif-Gauss}
Suppose the source experiment $\mathcal{U}(\cdot | \theta)$ is given by the uniform location model~\eqref{eq:Uniform-location} and the target experiment $v( \cdot ; \theta)$ is given by the Gaussian nonlinear location model~\eqref{eq:gaussian-target-f}. Recall the functionals $\alpha_{0}$ and $\alpha_{1}$ defined in Eq.~\eqref{eq:alpha0-alpha1-f}. Choose $\mathcal{P}(\cdot | x)$ to be the pdf of the Gaussian distribution $\NORMAL(0,2\sigma^2)$ for each $x \in \mathbb{X}$.  Then for each $\epsilon \in (0,1)$, parameter $\theta \in [-1/2,1/2]$, and standard deviation $\sigma \geq 40 \max\{\alpha_{0}, \alpha_{1}\}$, Algorithm $\textsc{rk}$ using $\mathcal{S}^*$ from Eqs.~\eqref{S-star-uniform} and~\eqref{eq:unif-g-func}, with the parameter settings 
$M = 30$ and $N = 60\log(4/\epsilon)$ and any $y_0 \in \real$ satisfies
\begin{align} \label{eq:claim-unif-Gauss}
	\left\| \mathcal{L} \left[ \textsc{rk}(X_{\theta}, N, M, y_0) \right] - \NORMAL\big(f(\theta), \sigma^2\big) \right\|_{\mathsf{TV}} \leq \frac{\epsilon}{2} + 10 \exp\bigg\{-\frac{\sigma^2}{200\alpha_{1}^{2}}\bigg\}.
\end{align}
Consequently, for $\Theta = [-1/2,1/2]$ and provided $\sigma \geq 10\alpha_{1} \sqrt{2 \log (20 / \epsilon)}$, the reduction algorithm $\textsc{rk}$ runs in time $\mathcal{O}(\log(4/\epsilon)(T_{samp} + T_{eval}))$ and satisfies
\begin{align}
\delta(\mathcal{U}, \mathcal{V}; \textsc{rk}) \leq \epsilon.
\end{align}
\end{theorem}

Theorem~\ref{thm:Unif-Gauss} is proved in Section~\ref{sec:pf-UnifGauss}. Our choice of target experiment---Gaussian location models with nonlinearly transformed mean---is motivated by a concrete reduction to follow, between symmetric mixtures of experts and phase retrieval (see Section~\ref{sec:MOE-PR}). In this reduction, we choose $f$ to be the absolute value function, which has a single point of nondifferentiability, but Theorem~\ref{thm:Unif-Gauss} also applies for any differentiable function $f$. Properties of $f$ enter the theorem through the parameters $\alpha_0$ and $\alpha_1$, but note that these only depend on the zeroth- and first-order properties of the function.

\begin{remark} \label{rem:nonsmooth-target}
As previously mentioned, suppose $f$ is continuous but nondifferentiable at a finite number of points. Then it is easy to see that there must exist a differentiable function $\widetilde{f}$ such that for all $\eta > 0$, we have $\sup_{z \in [-1/2, 1/2]} |\widetilde{f}(z) - f(z) | \leq \eta$ and the functionals $\alpha_0$ and $\alpha_1$ of $\widetilde{f}$ are bounded by the corresponding functionals of $f$. Thus, as in Example~\ref{example-exp-laplace}, one can design a reduction from the uniform location model to the surrogate target model $\NORMAL(\widetilde{f}(\theta), \sigma^2)$ and apply Theorem~\ref{thm:Unif-Gauss} to obtain a TV deficiency guarantee. The additional TV distortion to the model $\NORMAL(f(\theta), \sigma^2)$ will be bounded\footnote{To see this, bound the KL-divergence between $\NORMAL(\widetilde{f}(\theta), \sigma^2)$ and $\NORMAL(f(\theta), \sigma^2)$ and then use Pinsker's inequality.} proportionally to $\eta$, and one may then take $\eta \downarrow 0$ to obtain a reduction to $\NORMAL(f(\theta), \sigma^2)$.
\end{remark}

\begin{figure}[] 
	\centering
	\begin{subfigure}[b]{0.4\textwidth}
		\centering
		\includegraphics[width=\textwidth]{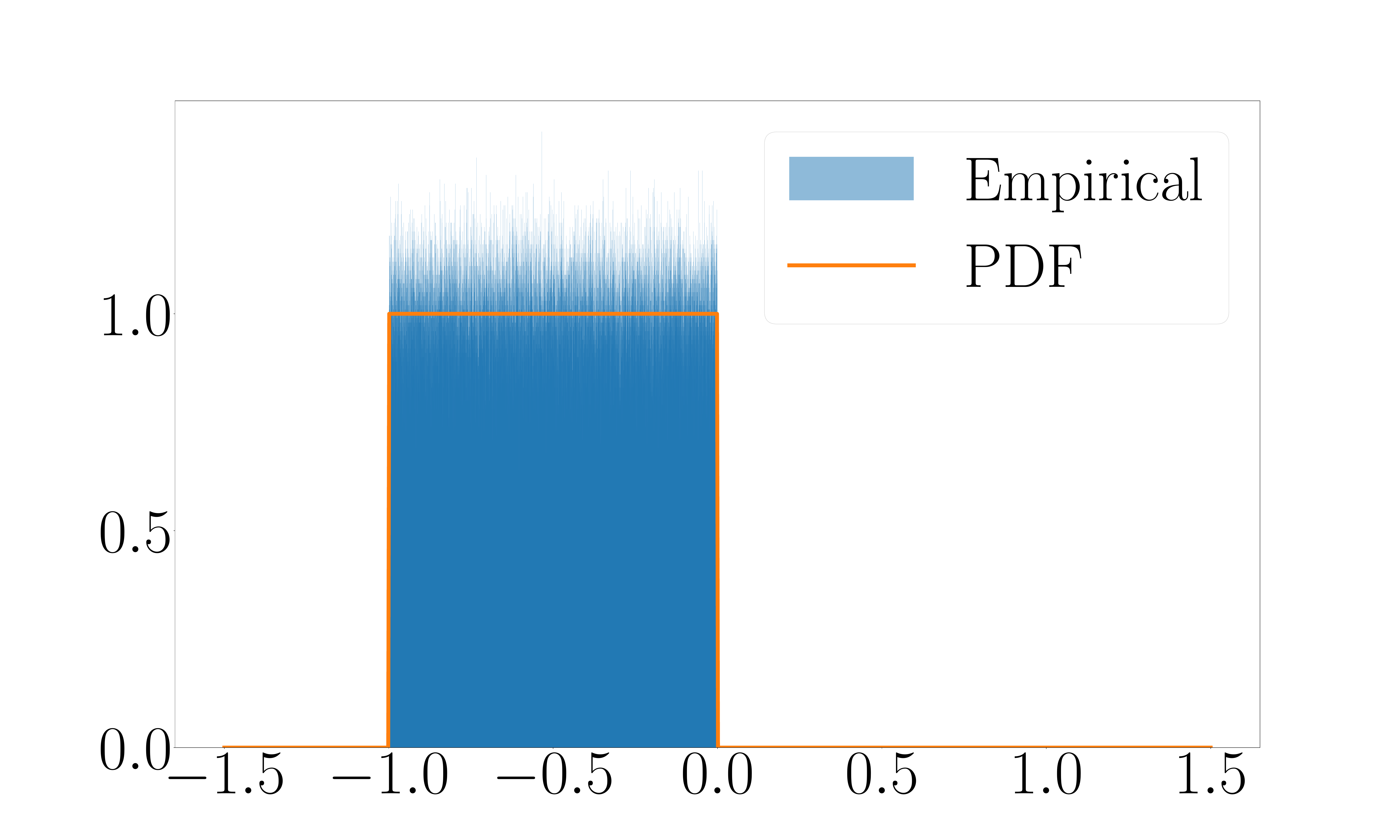}
		\caption{Source pdf and histogram, $\theta = -1/2$} 
		\label{fig:unif-gaussian1}
	\end{subfigure}
	\hfill
	\begin{subfigure}[b]{0.4\textwidth}  
		\centering 
		\includegraphics[width=\textwidth]{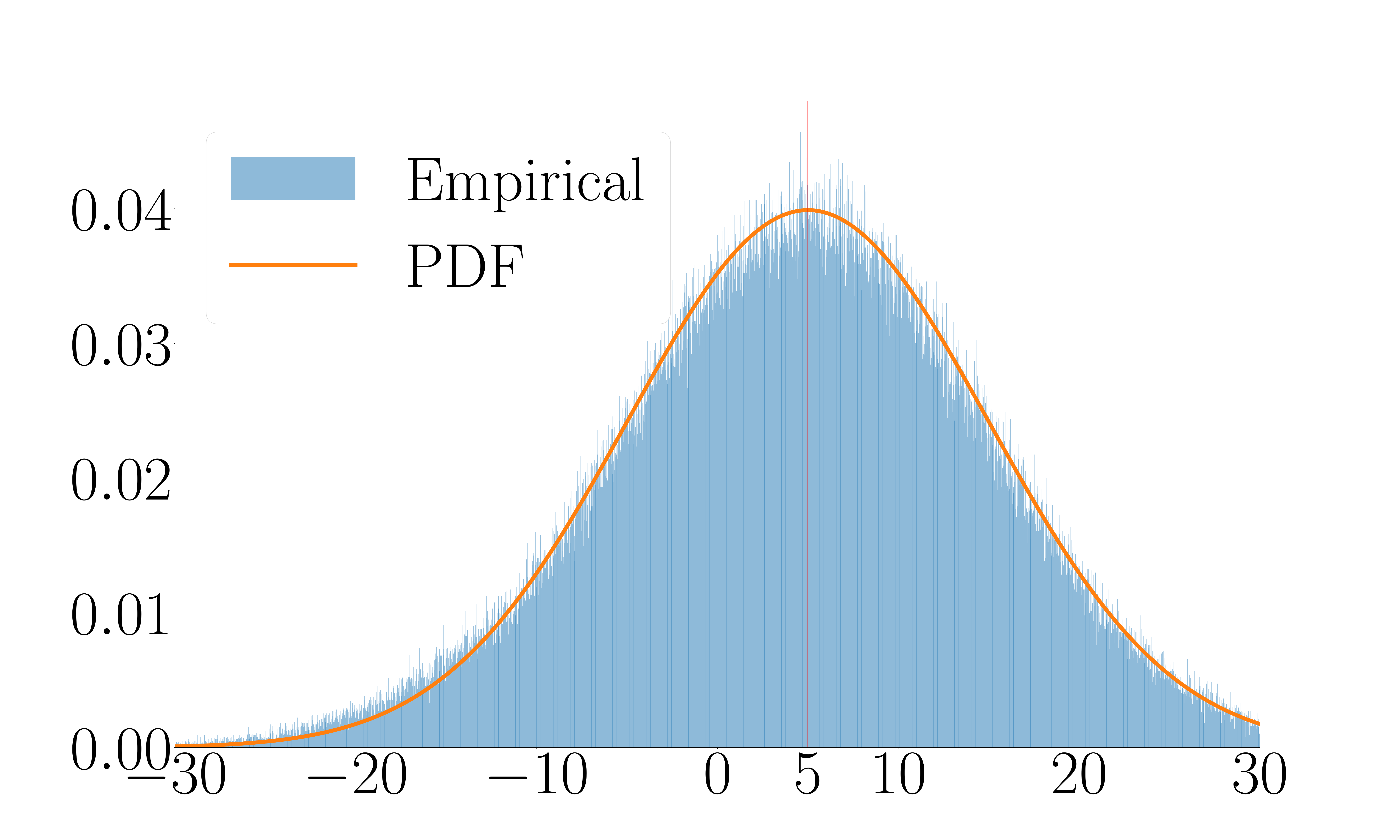}
		\caption{Target pdf and histogram, $\theta = -1/2$}
		\label{fig:unif-gaussian2}
	\end{subfigure}
	\vskip\baselineskip
	\vspace{-0.56cm}
	\begin{subfigure}[b]{0.4\textwidth}
		\centering
		\includegraphics[width=\textwidth]{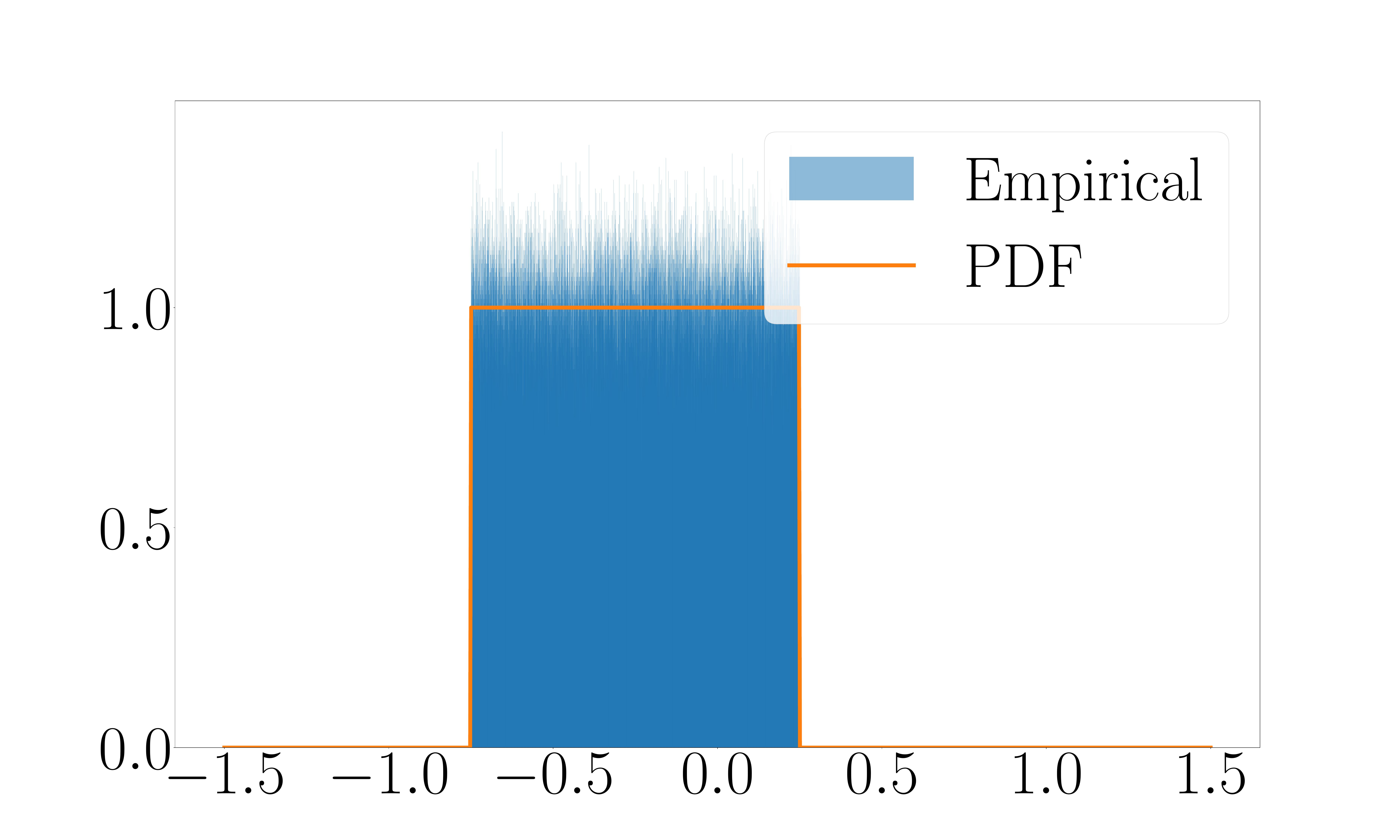}
		\caption{Source pdf and histogram, $\theta = -1/4$} 
		\label{fig:unif-gaussian3}
	\end{subfigure}
	\hfill
	\begin{subfigure}[b]{0.4\textwidth}  
		\centering 
		\includegraphics[width=\textwidth]{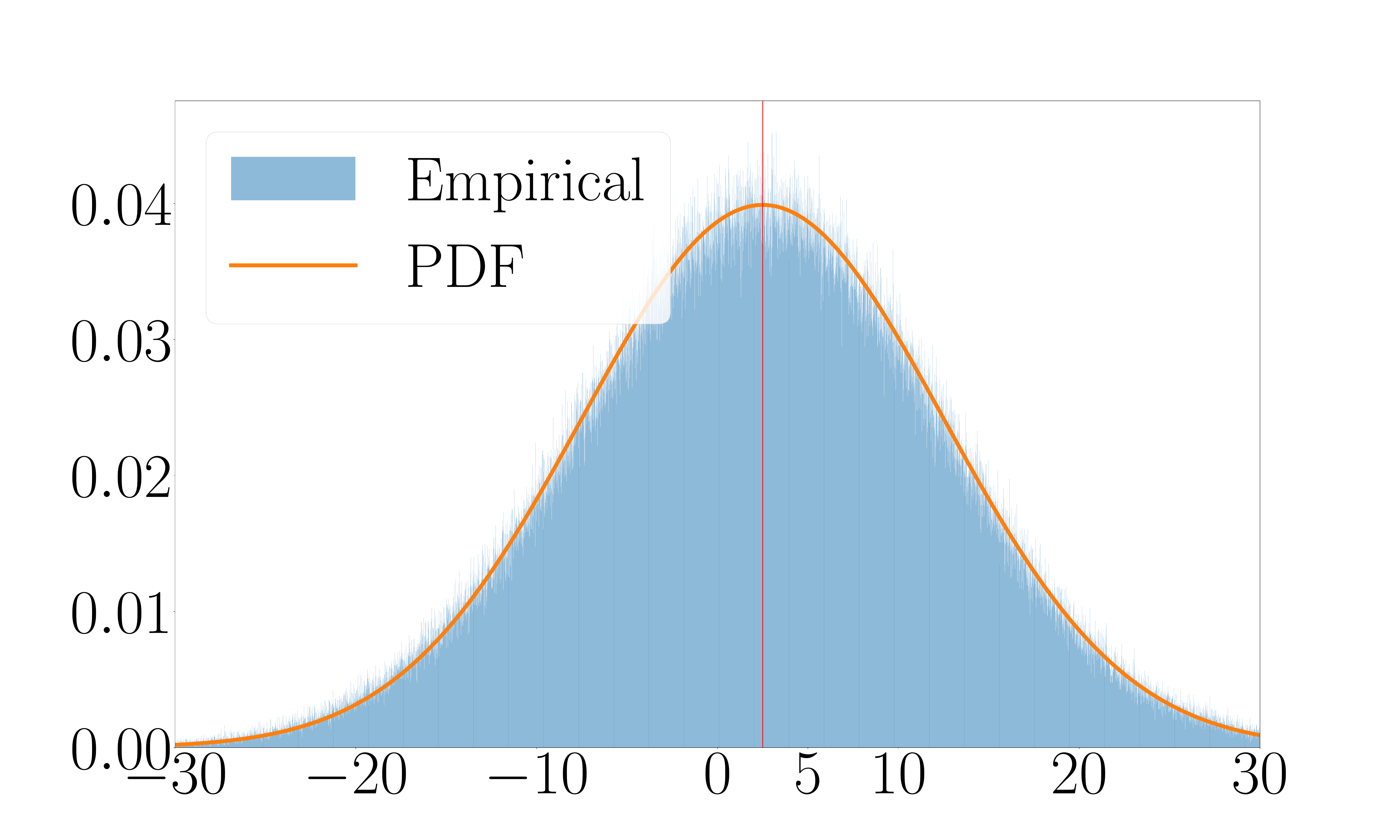}
		\caption{Target pdf and histogram, $\theta = -1/4$}
		\label{fig:unif-gaussian4}
	\end{subfigure}
	\vskip\baselineskip
	\vspace{-0.56cm}
	\begin{subfigure}[b]{0.4\textwidth}
		\centering
		\includegraphics[width=\textwidth]{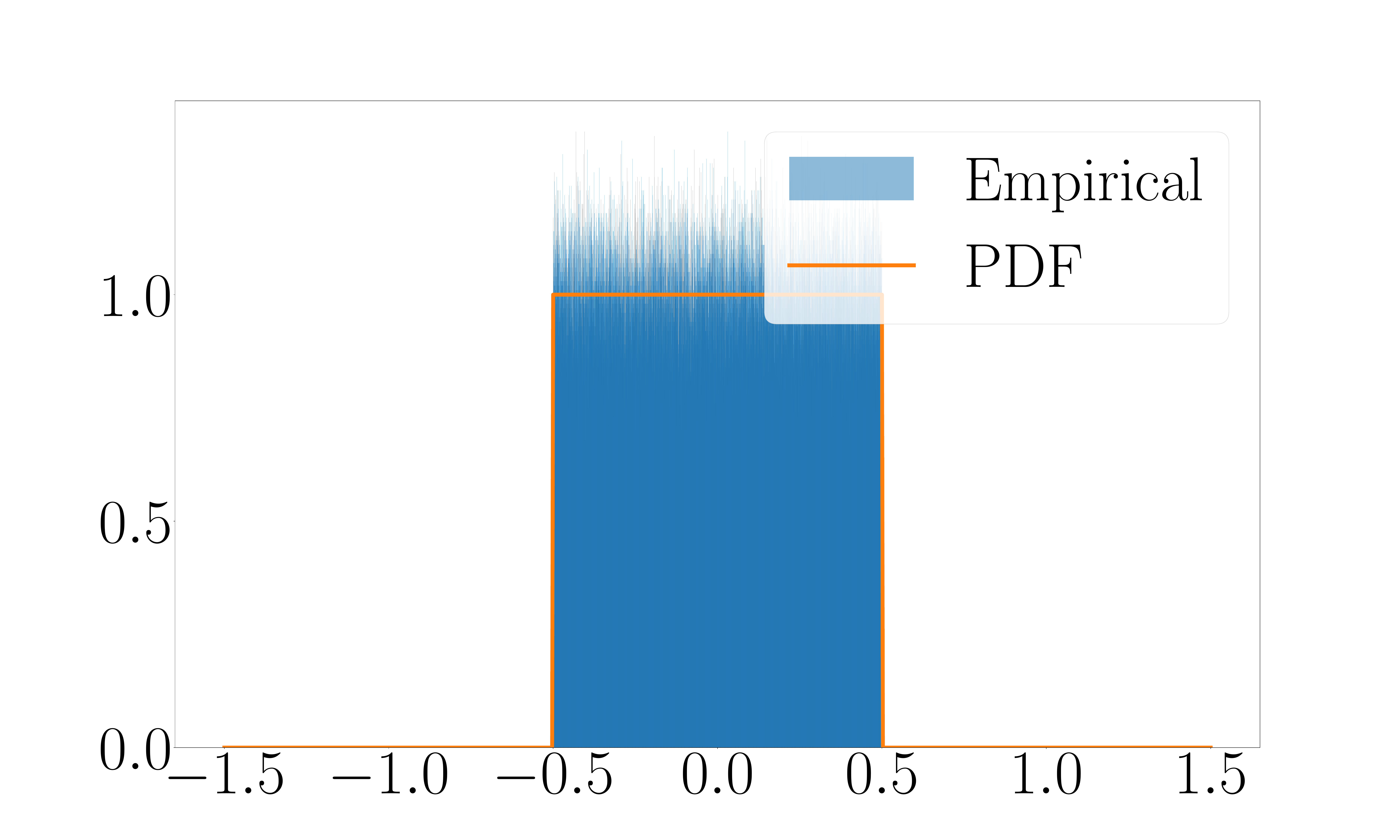}
		\caption{Source pdf and histogram, $\theta = 0$} 
		\label{fig:unif-gaussian5}
	\end{subfigure}
	\hfill
	\begin{subfigure}[b]{0.4\textwidth}  
		\centering 
		\includegraphics[width=\textwidth]{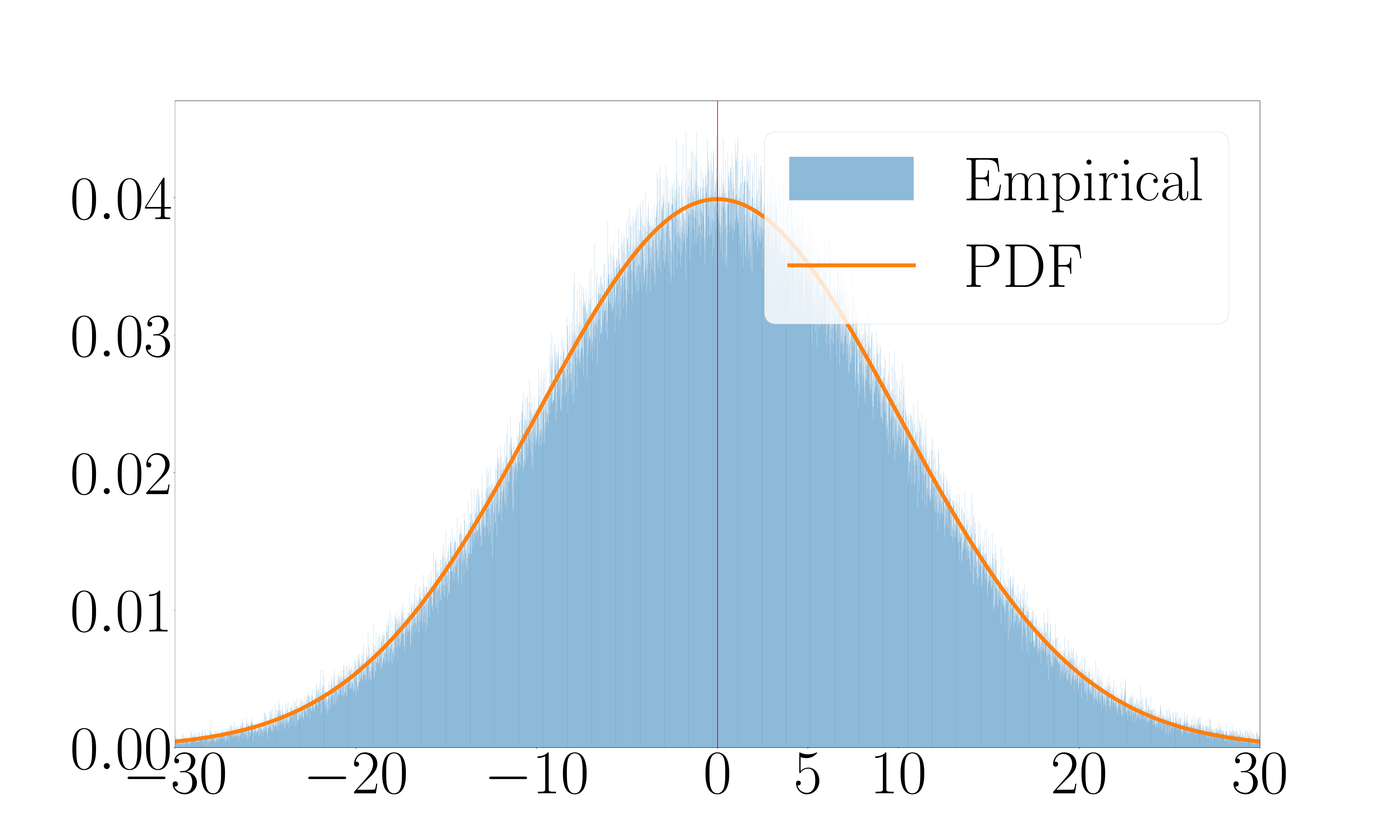}
		\caption{Target pdf and histogram, $\theta = 0$}
		\label{fig:unif-gaussian6}
	\end{subfigure}
	\vskip\baselineskip
	\vspace{-0.56cm}
	\begin{subfigure}[b]{0.4\textwidth}
		\centering
		\includegraphics[width=\textwidth]{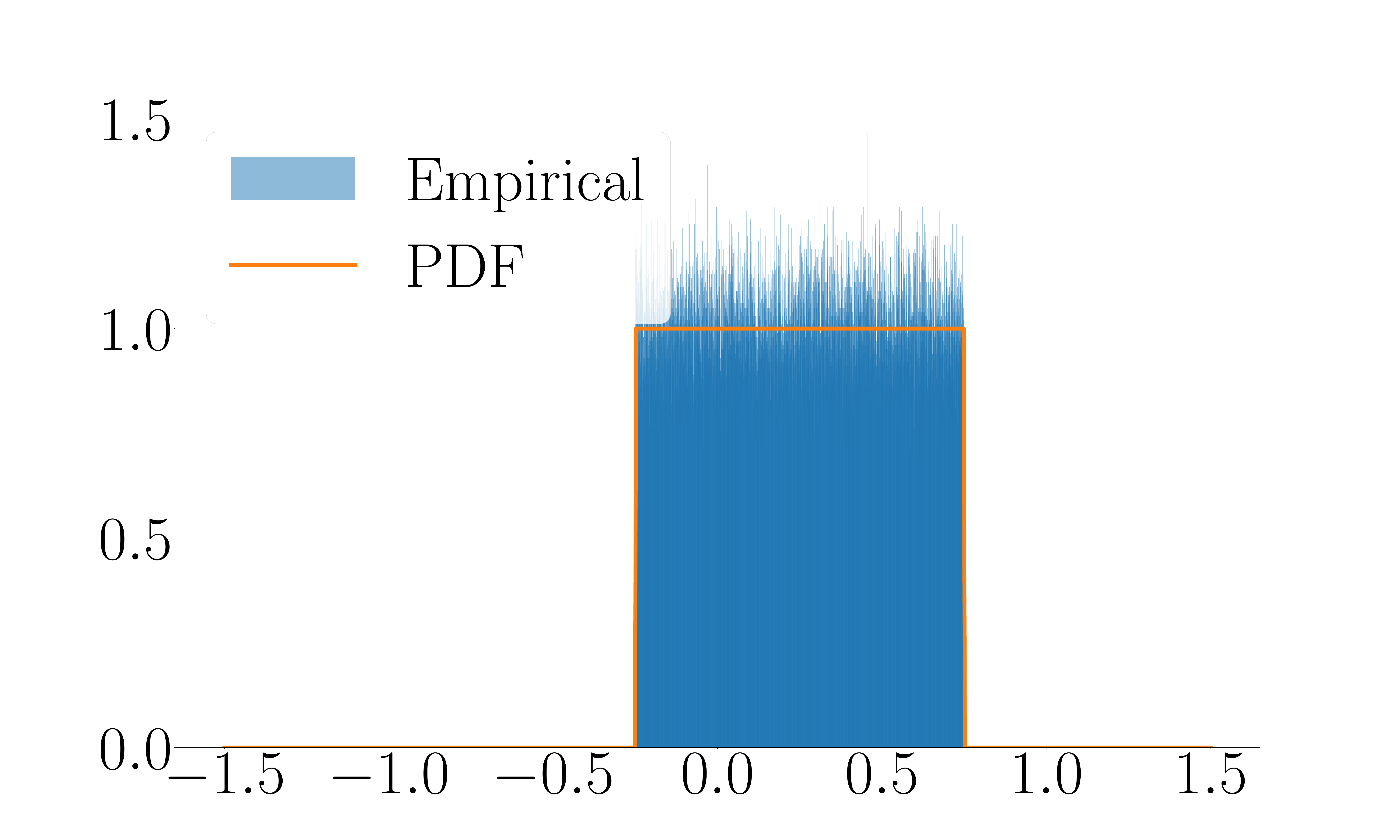}
		\caption{Source pdf and histogram, $\theta = 1/4$} 
		\label{fig:unif-gaussian7}
	\end{subfigure}
	\hfill
	\begin{subfigure}[b]{0.4\textwidth}  
		\centering 
		\includegraphics[width=\textwidth]{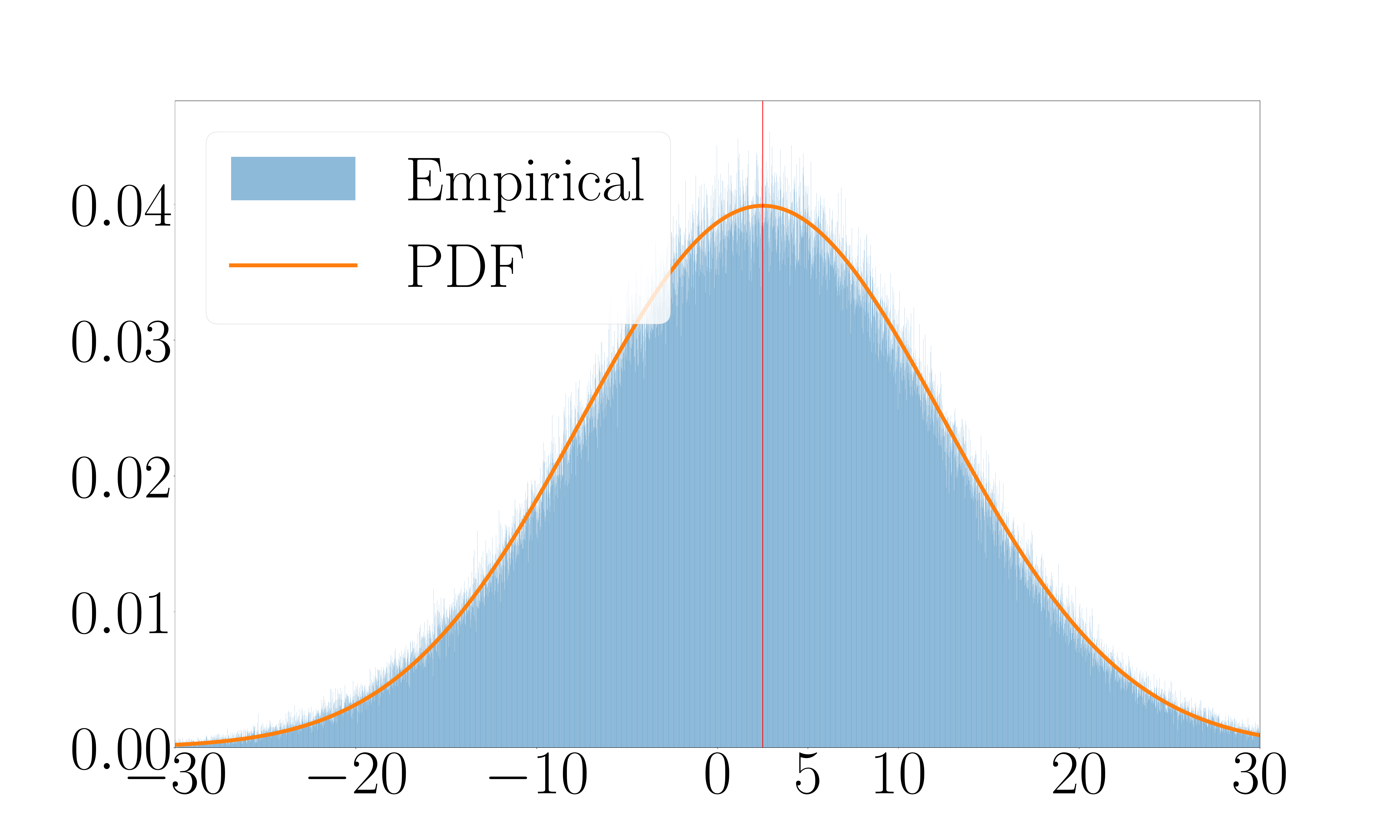}
		\caption{Target pdf and histogram, $\theta = 1/4$}
		\label{fig:unif-gaussian8}
	\end{subfigure}
	\vskip\baselineskip
	\vspace{-0.56cm}
	\begin{subfigure}[b]{0.4\textwidth}
		\centering
		\includegraphics[width=\textwidth]{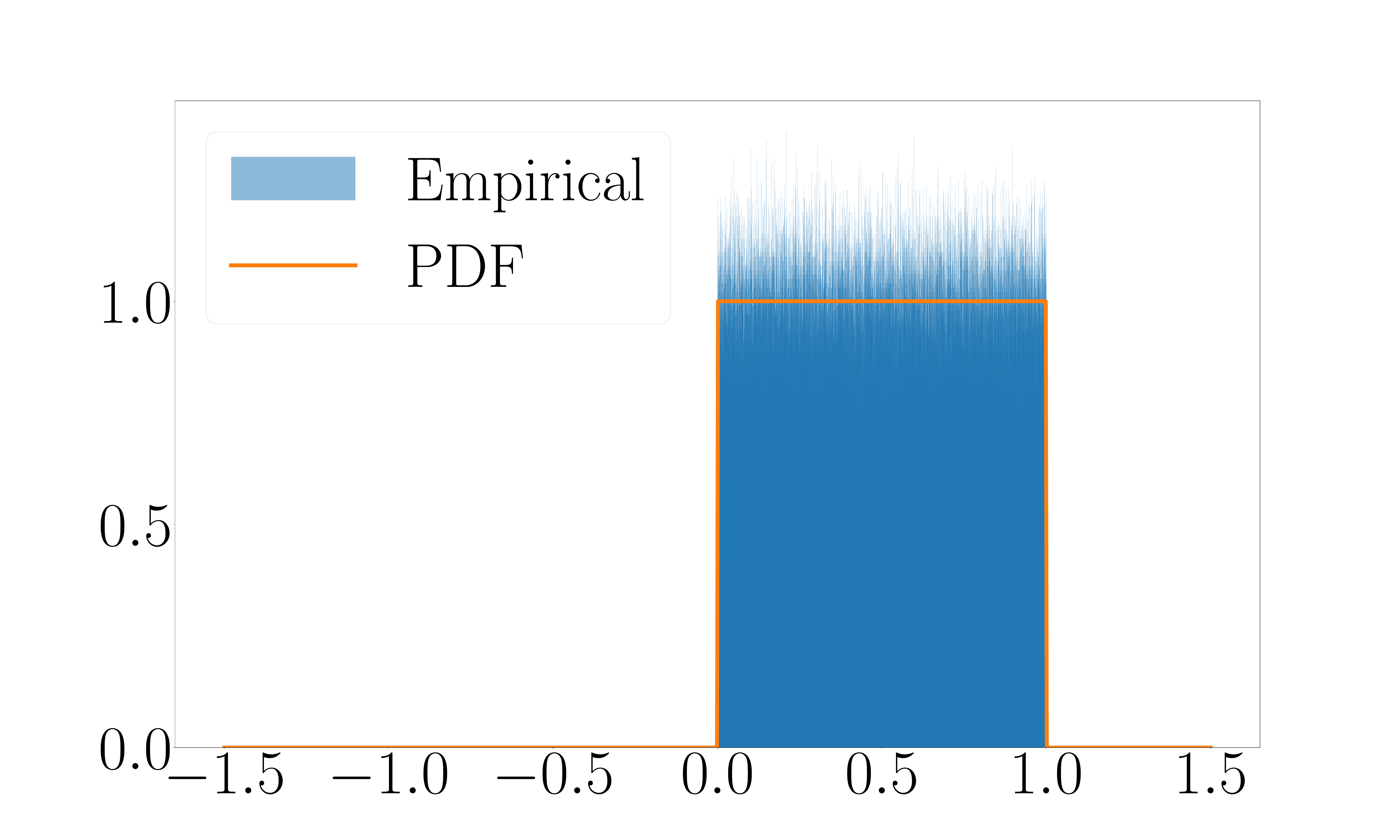}
		\caption{Source pdf and histogram, $\theta = 1/2$} 
		\label{fig:unif-gaussian9}
	\end{subfigure}
	\hfill
	\begin{subfigure}[b]{0.4\textwidth}  
		\centering 
		\includegraphics[width=\textwidth]{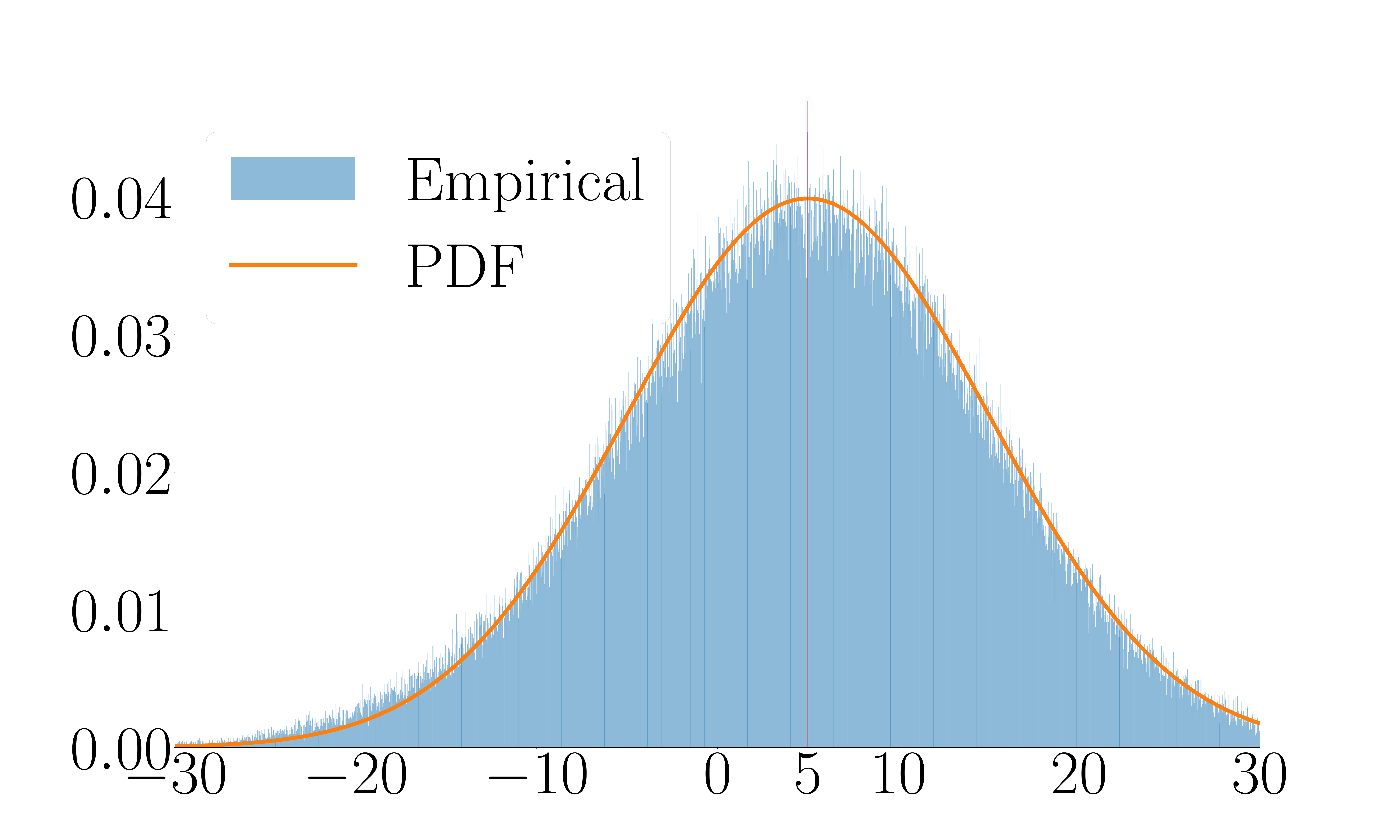}
		\caption{Target pdf and histogram, $\theta = 1/2$}
		\label{fig:unif-gaussian10}
	\end{subfigure}
	\caption{Simulation of $\mathsf{Unif}([\theta-1/2,\theta + 1/2])$ to $\NORMAL(10\cdot|\theta|,\sigma^2)$ reduction, plotted for parameters $\theta \in\{ -1/2, -1/4, 0, 1/4, 1/2\}$ and $\sigma = 10$.}
	\label{fig:unif-gaussian}
\end{figure}

Figure~\ref{fig:unif-gaussian} provides a numerical illustration of Theorem~\ref{thm:Unif-Gauss}
with $f(\theta) = 10 \cdot |\theta|$ so that our target distribution is
$\NORMAL(10\cdot |\theta|,\sigma^2)$. We set $M = 30$, $N = 3000$, $\sigma = 10$, $y_0 = 0$ and set the base distribution $\mathcal{P}(y|x)$ and the kernel $\mathcal{S}^{*}(y|x)$ as in Theorem~\ref{thm:Unif-Gauss}. For each $\theta \in \{-1/2,-1/4,0,1/4,1/2\}$, 
we generate $K = 10^{6}$ source samples $\{X^{i}\}_{i=1}^{K} \overset{\mathsf{i.i.d.}}{\sim} \mathsf{Unif}([\theta-1/2,\theta+1/2])$ and generate target samples $Y^{i} = \textsc{rk}(X^{i}, N, M, y_0)$ for each $i\in \{1,2,\dots,K\}$. From Figure~\ref{fig:unif-gaussian}, we see that the target sample almost follows distribution $\NORMAL(10 \cdot |\theta|,\sigma^2)$, i.e., the mean $\theta$ is transformed to be $10\cdot |\theta|$.

\section{Consequences} \label{sec:consequences}

We now present three concrete consequences of our reductions in Section~\ref{sec:reductions} for some canonical problems in high-dimensional statistics and machine learning. The following three sections mirror the short motivating paragraphs in Section~\ref{sec:intro}.

\subsection{Mixtures of experts and phase retrieval} \label{sec:MOE-PR}
Suppose there is some unknown $\beta_{\star} \in \real^d$. We observe i.i.d. data $(x_{i},y_{i})_{i=1}^{n}$ generated from a symmetric \emph{mixture of linear experts} model
\begin{align}\label{eq:mix-regression-model}
	y_{i} = R_{i} \cdot \langle x_{i}, \beta_{\star} \rangle + \xi_{i}, \quad i = 1, \ldots, n.
\end{align}
Here, the random variable $R_{i} \in \{-1,+1\}$ and additive noise $\xi_{i} \in \real$ are unobserved. Suppose that the noise variables are drawn as $\{\xi_{i}\}_{i=1}^{n} \overset{\mathsf{i.i.d.}}{\sim} \mathsf{Unif}([-1/2,1/2])$ and independent of the pair $(R_{i},x_i)$. We allow $R_i$ to depend on $x_i$ such as in the general definition of the mixture of experts model~\citep{jordan1994hierarchical,jacobs1997bias}. When $R_i$ is independent of $x_i$ and drawn from a Rademacher distribution, we obtain the well-studied \emph{mixture of linear regressions} model~\citep{quandt1978estimating,stadler2010l1}. 

We now show that any mixture of linear experts model can be reduced to the phase retrieval model, which is given by i.i.d. data $(x_i, \widetilde{y}_{i})_{i = 1}^n$ generated according to the nonlinear regression model
\begin{align}\label{eq:phase-retrieval-model}
	\widetilde{y}_{i} = \big|  \langle x_{i}, \beta_{\star} \rangle \big| + \sigma \cdot \varepsilon_{i}, \quad i = 1, \ldots, n.
\end{align}
The additive noise is independent of everything else, but now Gaussian, with $\{ \varepsilon_{i} \}_{i=1}^{n} \overset{\mathsf{i.i.d.}}{\sim} \NORMAL(0,1)$.

\begin{corollary}\label{coro-mix-phase-retrieval}
	Fix any value of the signs $\{R_{i}\}_{i=1}^{n} \in \{-1, 1\}$ and parameter $\beta_{\star} \in \real^d$. Also suppose that $\max_{i = 1, \ldots, n} \lvert \langle x_{i}, \beta_{\star} \rangle \rvert \leq 1/2$. Conditioned on $\{ x_i \}_{i = 1}^n$, suppose observations $(x_i, y_i)_{i = 1}^n$ are drawn i.i.d. from the model~\eqref{eq:mix-regression-model} and $(x_i, \widetilde{y}_{i})_{i=1}^{n}$ are drawn i.i.d. from model~\eqref{eq:phase-retrieval-model}.  There exists a universal, positive constant $C$ such that for each $\delta \in (0,1)$ and $\sigma = C\sqrt{\log(20n / \delta)}$, there is an algorithm $h: (\real^d \times \real)^n \to (\real^d \times \real)^n$ running in time polynomial in $n$ and $\log(1/\delta)$ such that 
	\[
		\mathsf{d_{TV}} \Big( h\left( (x_i, y_{i})_{i=1}^{n} \right), (x_i, \widetilde{y}_{i} )_{i=1}^{n} \Big) \leq \delta.
	\] 
\end{corollary}
We provide the proof of Corollary~\ref{coro-mix-phase-retrieval} in Section~\ref{sec:pf-mix-phase}. Taking $\delta$ small but constant, it shows that any mixture of experts model is essentially equivalent to a phase retrieval model with a variance that is inflated logarithmically in the sample size $n$. A consequence is that any estimator for the phase retrieval model, of which there are many (see, e.g.,~\citet{candes2015phase,candes2015phase2,fienup1982phase,tan2019phase,goldstein2018phasemax}) can be used alongside our reduction as an estimator for mixtures of experts (see Eq.~\eqref{eq:pointwise-risk-transfer} and Appendix~\ref{sec:loss-unbounded}). To be concrete, suppose the conditional distribution of $R_i$  given $x_i$ is known. Then a natural method to estimate $\beta_{\star}$ from observations of the model~\eqref{eq:mix-regression-model} is the expectation maximization algorithm~\citep{jordan1994hierarchical,makkuva2019breaking}. However, depending on the complexity of the given conditional distribution, each iteration of this algorithm may be complicated to run, and furthermore, it is difficult to obtain guarantees on the estimation error for such a general method. On the other hand, if we ignore the conditional distribution of the sign entirely and transform the observations into those from the phase retrieval model~\eqref{eq:phase-retrieval-model}, then we can choose any of the aforementioned algorithms to obtain provable and near-optimal estimation of $\beta_{\star}$. 

To produce a quantitative consequence of this flavor, it is useful to consider the source model to be mixtures of linear regressions, where the random variables $R_i$ are Rademacher and independent of everything else. This model has been extensively studied, and many algorithms and lower bounds are known in the literature~\citep[see, e.g.,][]{balakrishnan2017statistical,chen2017convex,brennan2020reducibility,arpino2023statistical}. 
In particular, certain algorithms like the popular alternating minimization heuristic are known to have distinct behavior for mixtures of linear regressions and phase retrieval---for instance, this algorithm is not consistent for parameter estimation in mixtures of linear regressions for any nonzero noise level but always consistent for phase retrieval~\citep{chandrasekher2023sharp}. Corollary~\ref{coro-mix-phase-retrieval} shows that  our reduction can be applied to observations from the mixtures of linear regressions model and then the alternating minimization algorithm can be run; this meta-algorithm then inherits the guarantees of alternating minimization for phase retrieval---in particular, consistent estimation of $\beta_\star$ with high probability---with a logarithmic in $n$ inflation in the noise variance.

Two technical aspects of Corollary~\ref{coro-mix-phase-retrieval} are worth emphasizing. First, note that the phase retrieval model is an instance of a single-index model~\citep{friedman1974projection,li1991sliced} in which the nonlinear link function is taken to be the absolute value function. A similar result to Corollary~\ref{coro-mix-phase-retrieval} can be proved by identical means for other link functions, by applying Theorem~\ref{thm:Unif-Gauss} and repeating the steps in Section~\ref{sec:pf-mix-phase}. Second, Corollary~\ref{coro-mix-phase-retrieval} requires bounded signal to noise ratio with $|\langle x_i, \theta \rangle| \leq 1/2$. This should not be thought of as strict; indeed, in the limiting case of infinite signal-to-noise ratio (in which the models~\eqref{eq:mix-regression-model} and~\eqref{eq:phase-retrieval-model} are both noiseless), there is a simple reduction that takes the absolute values of the responses in model~\eqref{eq:mix-regression-model} to obtain responses in the model~\eqref{eq:phase-retrieval-model}. The regime of finite signal-to-noise ratio is more delicate and the aforementioned simple reduction will not work, since the absolute values of responses in model~\eqref{eq:mix-regression-model} will not correspond to responses from the phase retrieval model but to one with dithering~\citep{thrampoulidis2020generalized}. 

\subsection{Signal denoising with log-concave noise and partial observations}

Let $\Theta \subseteq \mathbb{R}^{\Dim}$ be an arbitrary constraint set in finite-dimensional space, a setting that covers vector problems as well as their isomorphisms. For 
example, using the standard isomorphism between $\real^{d_{1} d_{2}}$ and $\real^{d_{1} \times d_{2}}$, the set $\Theta$ can represent the set of rank-$r$ and $d_{1} \times d_{2}$ matrices. Other examples could involve tensor-valued parameters or other, possibly nonparametric, constraints~\citep[see, e.g.,][]{oymak2016sharp,chatterjee2014new,chatterjee2023denoising}.

Now consider the denoising problem with missing data, which is the problem of estimating an unknown signal $\theta^{\star} \in \Theta$ from noisy observations of a subset of entries 
\begin{align}\label{eq:low-rank-observ-exp}
	y_{i} = \begin{cases} 
	\theta^{\star}_{i} + \xi_{i}, & i \in \Omega, \\
	\star, & i \notin \Omega.
	\end{cases}
\end{align}
Here $\Omega \subseteq \{1,\dots,\Dim\}$ denotes the set of indices for observed entries, $\xi_{i}$ is additive, zero-mean noise at entry $i$, and $\star$ denotes that that entry is unobserved. We assume that the noise variables $\{\xi_{i}\}_{i \in \Omega}$ are i.i.d., and suppose that each $\xi_i$ is exponential with probability density function
\[
	u(x) = \begin{cases} \exp\{-x-1\}, & \quad \text{for } x\geq -1,
	\\0,& \quad \text{for } x< -1.  \end{cases}
\]
We will show that the model with exponential noise~\eqref{eq:low-rank-observ-exp} can be reduced to a model with general log-concave noise, given by
\begin{align}\label{eq:low-rank-observ-log-concave}
	\widetilde{y}_{i} = 
	\begin{cases} \theta^{\star}_{i} + \varepsilon_{i}, &i \in \Omega, \\
	\star, & i \notin \Omega.
	\end{cases}
\end{align}
As before, the noise variables $\{ \varepsilon_{i} \}_{i \in \Omega}$ are i.i.d. and each $\varepsilon_{i}$ is log-concave with probability density function $v(z) = \frac{1}{\sigma}\exp\left\{-\psi\Big(\frac{z}{\sigma}\Big)\right\}$, where $\psi:\real \rightarrow \real$ is convex, differentiable, and satisfies Eq.~\eqref{eq:psi-density-condition}.

Let $T_{samp}$ be the time to sample from the base measure $\mathcal{P}(\cdot|x)$~\eqref{eq:base-measure-log-concave} and $T_{eval}$ be the time to evaluate the kernel $\mathcal{S}^{*}(y|x)$~\eqref{eq:signed-kernel-Exp} at a single point. Recall the definition of functionals $\tau(\sigma)$ in Eq.~\eqref{eq:functionals}. With this notation, we are now ready to state the theorem.
\begin{corollary}\label{coro-low-rank-matrix}
	Suppose $\{y_{i}\}_{i=1}^{\Dim}$ satisfy Eq.~\eqref{eq:low-rank-observ-exp} and $\{\widetilde{y}_{i}\}_{i=1}^{\Dim}$ satisfy Eq.~\eqref{eq:low-rank-observ-log-concave} with the same ground truth $\theta^{\star} \in \mathbb{R}^{\Dim}$. Suppose $\sigma>0$ satisfies $\tau(\sigma)<1$, $M$ satisfies Eq.~\eqref{ineq:M-lower-bound-exp}, and $N = M \log(2/\epsilon) /(1-\tau(\sigma))$ for $\epsilon \in (0,1)$. Then there is an algorithm $h:\mathbb{R}^{\Dim} \rightarrow \mathbb{R}^{\Dim}$ running in time $O\big(|\Omega| N (T_{samp} + T_{eval})\big)$ for which
	\begin{align}\label{ineq:TV-bound-low-rank}
		\mathsf{d_{TV}} \Big( h\big( \{y_{i}\}_{i=1}^{\Dim} \big), \{\widetilde{y}_{i}\}_{i=1}^{\Dim}\Big) \leq |\Omega| \cdot \big(\epsilon +  \tau(\sigma) \big).
	\end{align}
\end{corollary}

We provide the proof of Corollary~\ref{coro-low-rank-matrix} in Section~\ref{sec:pf-low-rank}. The reduction guaranteed by Corollary~\ref{coro-low-rank-matrix} for denoising problems has two salient features. First, it is \emph{structure preserving}, in that the signal $\theta^*$ (and hence its structure) is preserved in both source and target models. Second, the reduction is robust to missing data or partial observations, since we only need to transform the observed entries from the model and leave the unobserved entries as they are. As done in Examples~\ref{example-exp-gaussian}, \ref{example-exp-logistic}, and~\ref{example-exp-laplace}, Corollary~\ref{coro-low-rank-matrix} can be specialized to particular log-concave targets such as the Gaussian, logistic, and Laplace location models. As detailed in those examples, in the latter two cases an exact reduction exists provided $\sigma$ is greater than a universal constant (without dependence on dimension or the set $\Omega$); in the former case, the variance of the target model will need to be inflated logarithmically in the number of observed entries $|\Omega|$.

Let us give two further consequences of Corollary~\ref{coro-low-rank-matrix} for characterizing the risk of estimation. First consider the Gaussian target model with full observations, in which denoising problems and their relatives have been extensively studied in the literature, especially when $\Theta$ is a convex set~\citep{amelunxen2014living,oymak2016sharp,chatterjee2014new,neykov2022minimax}. For such problems, the proximal mapping of observations onto $\Theta$ (which corresponds to a type of regularized least squares estimator) has been carefully studied, and explicit error bounds in $\ell_2$ loss---often with sharp constants---have been proved depending on the geometry of the set around the true signal $\theta^*$. It is also known that the error rate of the least squares estimator in the Gaussian setting can be characterized in terms of the so-called ``critical radius'' of the problem, given by the solution of a certain deterministic variational formula~\citep[Theorem 1.1]{chatterjee2014new}. 

Specifically, let $\widehat{\theta}_{\mathsf{LS}}(w) = \arg\min_{\theta \in \Theta} \| w - \theta \|_2$ denote the least squares estimator that projects any vector $w \in \real^\Dim$ onto a closed convex set $\Theta \subseteq \real^{\Dim}$. In the setting where 
\begin{align} \label{eq:Gaussian-convex}
\widetilde{y} = \theta^{\star} + \sigma Z 
\end{align}
for some $\theta^{\star} \in \Theta$ and $Z \sim \NORMAL(0, I_d)$, \citet{chatterjee2014new} showed that there is an explicit positive scalar $t(\theta^*, \Theta; \sigma)$ expressible in terms of the local Gaussian width of the tangent cone of $\Theta$ centered at $\theta^*$ and the noise standard deviation $\sigma$, such that
\begin{align}\label{eq:risk-bound-gauss-noise}
|\mathbb{E}[ \| \widehat{\theta}_{\mathsf{LS}}(\widetilde{y}) - \theta^* \|_2 ] - t(\theta^*, \Theta; \sigma)| = \mathcal{O} ( \sqrt{t(\theta^*, \Theta; \sigma)} \vee 1).
\end{align}
It is also known that the random loss $\| \widehat{\theta}_{\mathsf{LS}}(\widetilde{y}) - \theta^* \|_2$ concentrates around its expectation~\citep{chatterjee2014new,van2017concentration}. Typically, one should expect that $t(\theta^*, \Theta; \sigma)$ is larger than one---or grows with $n$---so that the fluctuation term $\sqrt{t(\theta^*, \Theta; \sigma)}$ is of a much smaller order than the error prediction $t(\theta^*, \Theta; \sigma)$. In particular, since the bound is in absolute value, this gives a sharp (two-sided) guarantee on performance.

Now imagine that data $y \in \real^\Dim$ are instead drawn from a model with exponential noise~\eqref{eq:low-rank-observ-exp}, which, in particular, has heavier tails than its Gaussian counterpart. We will use the reduction in Corollary~\ref{coro-low-rank-matrix} with $\psi(z) = z^{2}/2+\log(2\pi)/2$ and the setting $\epsilon = \exp(-\sigma^{2}/2)$. Call the resulting output $\overline{y} = h(y)$, noting that $h$ runs in time $O\big(\Dim \sigma^{2} (T_{samp} + T_{eval})\big)$.
Applying Corollary~\ref{coro-low-rank-matrix} in conjunction with the inequality $\tau(\sigma) \leq 2\exp(-\sigma^2/2)$ (see Example~\ref{example-exp-gaussian}), we obtain for $\sigma \geq 2$ and $\widetilde{y}$ as in Eq.~\eqref{eq:Gaussian-convex}, we have
\begin{align}\label{ineq:TV-exp-gauss-signal-denoise}
	\mathsf{d_{TV}}( h(y), \widetilde{y}) \leq 3\Dim\exp\left( -\sigma^2/2 \right).
\end{align}
Note that such a total variation guarantee does not immediately imply a guarantee in $\ell_2$ risk.
Nevertheless, the following proposition shows that the least squares estimator applied to $\overline{y}$ inherits the sharp guarantee proved for Gaussian data.
\begin{proposition} \label{prop:Exp-reduction-sharp}
 Suppose $y  = \theta^{\star} + \xi$ as described in the exponential location model~\eqref{eq:low-rank-observ-exp}, and that $n \geq 2$. Generate $\overline{y} = h(y)$ with parameter $\epsilon = \exp(-\sigma^{2}/2)$ and with input $y_0 = y$ to the rejection kernel. Then for all $\sigma\geq  6\log n$, we have
 \[
 | \mathbb{E}[ \| \widehat{\theta}_{\mathsf{LS}}(\overline{y}) - \theta^* \|_2 ] - t(\theta^*, \Theta; \sigma)| = \mathcal{O} ( \sqrt{t(\theta^*, \Theta; \sigma)} \vee 1) + \mathcal{O}\left(  \frac{ \sigma \log \Dim }{\Dim} \right). 
\]
\end{proposition}
We provide the proof of Proposition~\ref{prop:Exp-reduction-sharp} in Section~\ref{sec:pf-prop-cor2}. 
Note that as previously mentioned, one should expect that $t(\theta^*, \Theta; \sigma)$ is either fixed or grows with $n$, so that the additional fluctuation term $\mathcal{O}\left(  \frac{ \sigma \log \Dim }{\Dim} \right)$ should be thought of as negligible when compared to the analogous bound~\eqref{eq:risk-bound-gauss-noise} under Gaussian noise. Thus, operationally speaking, Proposition~\ref{prop:Exp-reduction-sharp} allows us to sharply characterize the risk of an estimator (in a two-sided fashion) under the exponential location model using the same geometric quantities as those that appear in the Gaussian case. 

A second application is in proving minimax lower bounds. For instance, consider the case of structured matrix estimation with partial observations, a problem that has been considered under various signal and noise models~\citep[see, e.g.,][]{davenport20141,soni2016noisy,mcrae2021low}. By Corollary~\ref{coro-low-rank-matrix}, any minimax lower bound proved under exponential noise transfers to its log-concave counterparts without having to recompute any information-theoretic quantities. To give a concrete example, minimax lower bounds for estimating sparse factor models with missing data are known when the observations are corrupted by Laplace noise~\citep{sambasivan2018minimax}. Via an analog of Corollary~\ref{coro-low-rank-matrix} for Laplace sources\footnote{A reduction from Laplace source to log-concave target can be shown by combining the techniques in Proposition~\ref{prop:Laplace} and Theorem~\ref{thm:exp-log-concave}.}, one can transfer this specific lower bound to any log-concave distribution that is covered by our theory.

\subsection{Privacy: Transforming a Laplace mechanism into a Gaussian mechanism}

Over the last few decades, differential privacy has emerged as a core paradigm for privacy-preserving data analysis, whereby we wish to answer ``aggregate" queries about databases while preserving the privacy of individuals within the database. There are several mechanisms that preserve differential privacy. Many of these are classical, \emph{noise-addition} mechanisms for vector-valued numeric queries about the database~\citep{dwork2006differential}. 

To make things concrete, consider a database of $n$ individuals given by $(X_1, \ldots, X_n) \in \mathcal{X}^n$ and a function $f: \mathcal{X}^n \to \mathbb{R}$ that we wish to compute in a differentially private fashion via a (possibly randomized) mechanism $g: \mathcal{X}^n \to \mathbb{R}$. Databases $\boldsymbol{X} = (X_1, \ldots, X_n)$ and $\boldsymbol{Y} = (Y_1, \ldots, Y_n)$ are said to be neighbors if their Hamming distance is at most $1$. 
The mechanism $g$ is said to be $(\epsilon, \delta)$ differentially private if for all Borel sets $B \subseteq \mathbb{R}$ and neighboring databases $(\boldsymbol{X}, \boldsymbol{Y})$, we have
\[
\mathbb{P}( g(\boldsymbol{X}) \in B ) \leq \exp(\epsilon) \cdot \mathbb{P}( g(\boldsymbol{Y}) \in B ) + \delta.
\]

An important consideration is how $f$ itself behaves on neighboring databases. The \emph{sensitivity} of $f$ is given by
\[
\phi(f) = \max_{\boldsymbol{X}, \boldsymbol{Y} \text{ neighbors} } | f(\boldsymbol{X}) - f(\boldsymbol{Y}) |.
\]
While any query can be made private with the addition of enough noise, one also desires that the mechanism be useful, in that the output of the mechanism is close to the desired query. Accordingly, the $\ell_2$ \emph{accuracy} 
of $g$ is given by $\max_{\boldsymbol{X}} \sqrt{\mathbb{E}[\| g(\boldsymbol{X}) - f(\boldsymbol{X})\|_2^2]}$.

Two common mechanisms for differentially private data release are given by the Laplace mechanism and the Gaussian mechanism. For example, the following result for the Laplace mechanism is classical.

\begin{proposition} [see, e.g., \cite{dwork2014algorithmic}] \label{prop:Laplace-DP}
Let $\xi \sim \mathsf{Lap}(0, b)$ denote a Laplace random variable drawn independently of the database. The Laplace mechanism $g(\boldsymbol{X}) = f(\boldsymbol{X}) + \xi$ is $\left( \frac{\phi(f)}{b}, 0 \right)$ differentially private for $f$ with accuracy $b\sqrt{2}$.
\end{proposition}

Other mechanisms come equipped with their own privacy and accuracy guarantees as well as composition properties, and different scenarios may call for our privacy-preserving mechanism to have distinct properties~\citep{dwork2014algorithmic,wasserman2010statistical,wood2018differential,bun2016concentrated}.  In particular, given that the Gaussian mechanism has a distinct set of operational advantages from its Laplace counterpart, a natural question is whether the output of a Laplace mechanism for $f$ can be transformed into a Gaussian mechanism for $f$ \emph{without having access to} $f$ or the database. For a concrete example, in the problem of differentially private empirical risk minimization (ERM)~\citep{chaudhuri2011differentially,rubinstein2012learning}, a common approach is to corrupt the parameter output of an ERM procedure with Laplace noise. If we wanted to now produce a Gaussian mechanism (in the classical sense), then we would need to recover the original ERM solution---which in turn requires re-solving the ERM problem and de-privatization---and corrupt it with Gaussian noise. In contrast, one might be interested in a direct approach that uses the output of the already executed Laplace mechanism. Does such a transformation exist? An application of Theorem~\ref{thm:Lap-Gaussian} answers this question in the affirmative.

\begin{corollary} \label{cor:privacy}
Suppose $g$ is given by the Laplace mechanism in Proposition~\ref{prop:Laplace-DP}, and let $Z \sim \NORMAL(0, 2b^2 \log(12/\delta))$ denote a Gaussian random variable drawn independently of everything else. For each $\boldsymbol{X} \in \mathcal{X}^n$, there is a mechanism $h$ constructible in time polynomial in $1/\delta$ purely from $g(\boldsymbol{X})$ (without access to $f$ or $\boldsymbol{X}$) such that $\mathsf{d_{TV}} ( h, f(\boldsymbol{X}) + Z) \leq \delta$.

In particular, the mechanism $h$ is $\left( \frac{\phi(f)}{b}, 0 \right)$ differentially private for $f$ with accuracy
\[
\sqrt{2b^2 \log(12/\delta) + 2b^2 + \frac{b^2}{4} \cdot \delta \log^{3/2} (12/\delta) }.
\]
\end{corollary}
Corollary~\ref{cor:privacy} is proved in Section~\ref{sec:pf-privacy}. A few comments are in order. First, note that the created mechanism $h$ is a postprocessing of the Laplace mechanism, and hence inherits its differential privacy guarantee (see~\citet[Theorem 3.14]{dwork2014algorithmic}). Second, it achieves nearly the same privacy-accuracy tradeoff that is achieved by the classical Gaussian mechanism, which has access to $f(\boldsymbol{X})$ and corrupts this query output with Gaussian noise having the same variance as $Z$ in Corollary~\ref{cor:privacy} (see~\citet[Appendix A]{dwork2014algorithmic}). The similarity of the two privacy-accuracy tradeoffs is particularly clear when $\delta \downarrow 0$, where the accuracy guarantee of $h$ tends to the limit $\sqrt{2b^2 \log(12/\delta)}$, which is precisely the standard deviation of the Gaussian $Z$.  In other words, adding Laplace noise to privatize data gives, to a degree, a universal mechanism, in that the result of this mechanism can be transformed, \emph{downstream}, to resemble an ``oracle'' Gaussian mechanism that is created with access to $f$. 
Third, note that Corollary~\ref{cor:privacy} provides an alternative proof that the classical Gaussian mechanism $f(\bm{X}) + Z$ is $\left( \frac{\phi(f)}{b}, \delta \right)$ differentially private. Such a proof is illuminating because it explicitly exhibits a $\left( \frac{\phi(f)}{b}, 0 \right)$ mechanism---one that can be realized via post-processing of the Laplace mechanism---that is $\delta$-close in total variation to the Gaussian mechanism. This provides a construction of an intermediate mechanism that witnesses the conditions in~\citet[Lemma 3.17]{dwork2014algorithmic}, which state (roughly speaking) that for any $(\epsilon, \delta)$ mechanism $\mathcal{M}$, there must exist an $(\epsilon, 0)$ mechanism that is $\delta$ close to $\mathcal{M}$ in total variation.
Finally, note that there is a direct analog of the above corollary for vector-valued queries, in which case Proposition~\ref{prop:Laplace} can be directly applied to produce the reduction instead of going through Theorem~\ref{thm:Lap-Gaussian}.

\section{Discussion} \label{sec:discussion}

Motivated by the problem of constructing computationally efficient reductions between high-dimensional models, we revisited the problem of bounding the total variation deficiency between some canonical statistical models, showcasing our framework on source models such as the Laplace, Erlang, and uniform location models. Note that our main lemma (Lemma~\ref{lem:rej-sampling}) is not limited to these reductions, and can in principle be applied much more broadly. 
As concrete consequences, we covered problems spanning mixture modeling, denoising, missing data, and privacy. As alluded to in Sections~\ref{sec:setup} and~\ref{sec:consequences}, we expect our results to have several consequences not covered here, specifically to proving (computationally constrained) minimax lower bounds.

Our work leaves open several questions; let us list a few here. First, note that as articulated in Remarks~\ref{rem:uniform-target} and~\ref{rem:nonsmooth-target}, the differentiability properties of source and target appear crucial to producing good reductions via the techniques introduced in this paper. In particular, all the source distributions that we have considered have nondifferentiable densities, and the corresponding target densities are at least as smooth as the respective source distributions. In particular, a reduction from a very smooth source density (such as a Gaussian) appears to be out of the reach of our techniques, unless the target is also Gaussian. It is an interesting open problem to prove a nontrivial reduction from a Gaussian source model, since it would be particular useful in transferring lower bounds from the canonical Gaussian location model to other noise models. 

The second set of questions pertains to our construction of a Markov kernel. Indeed, Algorithm~\textsc{rk} was designed to sample from the Markov kernel $\widehat{T}$~\eqref{eq:MK-thresh}, which was, in turn, a simple thresholded version of our oracle signed kernel. Is such a construction optimal given access to a signed kernel $\mathcal{S}^*$ satisfying $\SignedDef(\mathcal{U}, \mathcal{V}; \mathcal{S}^*) = 0$? Can we directly pursue the optimal Markov kernel by approximately solving the infinite-dimensional linear program~\eqref{eq:one-way-def} for some interesting class of source and target families?

A third direction would be to extend approximate reductions beyond total variation deficiency. Indeed, it is reasonable to define notions of deficiency---i.e., ask question (Q2)---in any ``distance" between probability measures, with natural choices being the family of $f$-divergences, R\'enyi divergences, and integral probability metrics. What forms do corresponding reductions take, and are they significantly different from reductions in total variation? How does the choice of divergence influence the particular data processing inequality that must be satisfied in order to map source to target?

As a final direction, we reiterate that an important application of reductions between experiments is to transfer risk bounds and algorithmic lower bounds, and we pose a few concrete questions in this vein. The paradigm of smoothed analysis~\citep{spielman2004smoothed} has been extensively applied to explain the practical efficiency of algorithms that are inefficient in the worst-case. More generally, average-case analysis of algorithms provides a framework to reason about the properties of algorithms when run on random instances of problems. These definitions rely intrinsically on the perturbations of an instance by noise that is chosen from a particular distribution, and this distribution is in turn chosen in some ad-hoc fashion. Reductions between problem instances corrupted by different noise distributions would show that similar insights hold for several other perturbation distributions. A distinct application is to prove lower bounds on particular classes of algorithms (or models of computation) for high-dimensional statistics problems, and not all computationally efficient algorithms as alluded to in Section~\ref{sec:setup}. Examples include algorithms that utilize low-degree polynomials of data~\citep{hopkins2018statistical,kunisky2019notes} or the sum-of-squares hierarchy~\citep{barak2019nearly}. If the reduction between source and target can be performed by an algorithm that is itself implementable by a certain model of computation, then it would show that lower bounds on this model of computation transfer from source to target instances. This would obviate the need for tailor-made lower bounds for each individual noise distribution and each class of algorithms.

\section{Proofs} \label{sec:proofs}

In this section, we provide proofs of all our main results, postponing the proofs of the more technical lemmas to Appendix~\ref{sec:proof-technical}. Throughout, we draw on the general notation and convention introduced in Section~\ref{sec:setup} of the main paper.
Additionally, we use $\mathcal{\mathcal{F}} \{ g(\cdot) \} (\omega) = \int_{\real^d} e^{i \langle \omega,  z \rangle } g(z) \mathrm{d} z$ to denote the Fourier transform of a function $g: \real^d \to \real$ evaluated at $\omega \in \mathbb{C}^d$, where $i = \sqrt{-1}$ and $\langle \cdot, \cdot \rangle$ is the standard inner product in $\mathbb{C}^d$.

\subsection{Proof of Proposition~\ref{prop:plugin}} \label{sec:pf-prop-plugin}

The proof is based on lower bounding the total variation between two distributions by the supremum of the difference between their characteristic functions. Letting $Z$ denote a standard normal variate and letting $G_W$ denote the characteristic function of $W$, we have
\begin{align*}
\mathsf{d_{TV}}(\mathsf{K}_{\mathsf{plugin}}(X_{\theta}), Y_{\theta}) &= \mathsf{d_{TV}}(\theta + W + \sigma Z, \theta + \sigma Z) \\
&= \mathsf{d_{TV}}(W + \sigma Z, \sigma Z) \\
&\overset{\1}{\geq} \sup_{t \in \mathbb{R}} \; \exp\{ -t^2 \sigma^2 / 2\} \cdot | G_W(t) - 1 |,
\end{align*}
where step $\1$ follows from~\cite{krishnamurthy2020algebraic}. Now observe that
\begin{align*}
G_W(t) = 
\begin{cases}
(1 + t^2)^{-1}  \quad &\text{ in case (i)} \\
(1 - it)^{-1} &\text{ in case (ii)}  \\
\frac{2}{t} \sin(t/2) &\text{ in case (iii).}
 \end{cases}
\end{align*}
In cases (i) and (iii), we have $| G_W(t) - 1 | \gtrsim t^2$ for all $|t| \leq 1$, and in case (ii), we have $| G_W(t) - 1 | \gtrsim t$ for all $|t| \leq 1$. Setting $t = 1/\sigma$ and noting that $\sigma \geq 1$ completes the proof. 
\qed


\subsection{Proof of Lemma~\ref{lem:rej-sampling}} \label{sec:pf-lemma1}

In this proof, we will assume that the source distribution $u(\cdot; \theta)$ as well as $\overline{\mathcal{S}}(\cdot|x)$ and $\mathcal{P}(\cdot | x)$ are absolutely continuous with respect to Lebesgue measure---this simplifies our notation and allows us to write densities $(u(x; \theta), \overline{\mathcal{S}}(y|x), \mathcal{P}(y | x))$. Note however that the proof extends straightforwardly to the general case.

The algorithm runs for at most $N$ iterations, and at each iteration, there is one sampling step and one evaluation step. The runtime is thus $N(T_{samp} + T_{eval})$.
To prove the guarantee on total variation, fix any $\theta \in \Theta$. By the correspondence between total variation and $\ell_1$ norm, we have
\begin{align*}
&\| \mathcal{L} \left[ \textsc{rk}(X_{\theta}, N, M, y_0) \right] - v( \cdot ; \theta) \|_{\mathsf{TV}} \\
&\quad= \frac{1}{2} \| \mathcal{L} \left[ \textsc{rk}(X_{\theta}, N, M, y_0) \right] - v( \cdot ; \theta) \|_{1} \\
&\quad \overset{\1}{\leq} \frac{1}{2} \| \mathcal{L} \left[ \textsc{rk}(X_{\theta}, N, M, y_0) \right] - (\widehat{\mathcal{T}} \circ \mathcal{U}) (\cdot | \theta) \|_{1} + \frac{1}{2} \| (\widehat{\mathcal{T}} \circ \mathcal{U}) (\cdot | \theta) - (\mathcal{S}^* \circ \mathcal{U}) (\cdot | \theta) \|_{1} + \frac{1}{2} \| (\mathcal{S}^* \circ \mathcal{U}) (\cdot | \theta) - \mathcal{V} (\cdot | \theta) \|_{1} \\
&\quad \overset{\2}{\leq} \frac{1}{2} \| \mathcal{L} \left[ \textsc{rk}(X_{\theta}, N, M, y_0) \right] - (\widehat{\mathcal{T}} \circ \mathcal{U}) (\cdot | \theta) \|_{1} + \frac{1}{2} \| (\widehat{\mathcal{T}} \circ \mathcal{U}) (\cdot | \theta) - (\mathcal{S}^* \circ \mathcal{U}) (\cdot | \theta) \|_{1} + \SignedDef(\mathcal{U}, \mathcal{V}; \mathcal{S}^*),
\end{align*}
where step~$\1$ follows by triangle inequality, and step $\2$ from the definition of $\SignedDef(\mathcal{U}, \mathcal{V}; \mathcal{S}^*)$~\eqref{eq:one-way-sign-fixed}. To complete the proof, it suffices to show that for each $\theta \in \Theta$, we have
\begin{subequations}
\begin{align}
\| \mathcal{L} \left[ \textsc{rk}(X_{\theta}, N, M, y_0) \right] - (\widehat{\mathcal{T}} \circ \mathcal{U}) (\cdot | \theta) \|_{1} &\leq 4\exp\left\{ -\frac{N}{M} \cdot \inf_{x \in \mathbb{X}}p(x) \right\},\quad \text{and} \label{clm:sim} \\
\| (\widehat{\mathcal{T}} \circ \mathcal{U}) (\cdot | \theta) - (\mathcal{S}^* \circ \mathcal{U}) (\cdot | \theta) \|_{1} &\leq \int_{\mathbb{X}} \left( |p(x) - 1| + q(x) \right) \cdot u(x;\theta) \;\mathrm{d}x. \label{clm:trunc}
\end{align}
\end{subequations}
We prove each statement in turn, taking as given the following technical lemma for the rejection sampling step. The proof of this lemma is provided in Appendix~\ref{sec:rej-samp-tech-lemma}. 

\begin{lemma} \label{lem:tech-prob}
Let $g(x) =  \Big( 1- \frac{p(x)}{M}\Big)^{N}$. For all $C \in \mathcal{B}(\mathbb{Y})$, we have
\begin{align}\label{eq:probability-Y-E}
	\Pr\big\{ Y \in C \big\} =  \frac{\int_{y \in C} \overline{\mathcal{S}}(y | x) \mathrm{d}y }{p(x)} \cdot (1-g(x)) + \mathbbm{1}_{y_{0} \in C} \cdot g(x).
\end{align}
\end{lemma}

\paragraph*{Proof of claim~\eqref{clm:sim}.}
Suppose $f(\cdot|x)$ is the probability density function of the random variable $Y = \textsc{rk}(x, N, M, y_0)$. By definition, we have
\begin{align}\label{ineq:upper-bound-l1-distance-master-lemma}
	\| \mathcal{L} \left[ \textsc{rk}(X_{\theta}, N, M, y_0) \right] - (\widehat{\mathcal{T}} \circ \mathcal{U}) (\cdot | \theta) \|_{1} &= \int_{y \in \mathbb{Y}} \bigg| \int_{x \in \mathbb{X}}  f(y | x) \cdot u(x;\theta) \mathrm{d}x - \int_{x\in \mathbb{X}} \widehat{\mathcal{T}} (y|x) \cdot u(x;\theta) \mathrm{d}x \bigg| \mathrm{d}y \nonumber
	\\ & \overset{\1}{\leq} \int_{x \in \mathbb{X}}  \int_{y\in \mathbb{Y}} \Big| f(y|x) - \widehat{\mathcal{T}} (y|x) \Big|   \mathrm{d}y \cdot u(x;\theta) \mathrm{d}x \nonumber
	\\ & =  \int_{x \in \mathbb{X}}   \big\| f(\cdot|x) - \widehat{\mathcal{T}} (\cdot|x) \big\|_{1}  \cdot u(x;\theta) \mathrm{d}x,
\end{align}
where step $\1$ follows from Jensen's inequality and Fubini's theorem. Continuing, we obtain
\begin{align*}
	\big\| f(\cdot|x) - \widehat{\mathcal{T}} (\cdot|x) \big\|_{1} &= 2 \big\| f(\cdot|x) - \widehat{\mathcal{T}} (\cdot|x) \big\|_{\mathsf{TV}} 
	\\& = 2 \sup_{C \in \mathcal{B}(\mathbb{Y}) } \bigg| \int_{y \in C} f(y|x) \mathrm{d}x  -  \int_{y \in C} \widehat{\mathcal{T}}(y|x) \mathrm{d}x \bigg|
	\\ &\overset{\1}{=} 2 \sup_{C \in \mathcal{B}(\mathbb{Y}) } \bigg| \frac{\int_{y \in C} \overline{\mathcal{S}}(y | x) \mathrm{d}y }{p(x)} \cdot (1-g(x)) + \mathbbm{1}_{y_{0} \in C} \cdot g(x) -  \frac{\int_{y \in C} \overline{\mathcal{S}}(y | x) \mathrm{d}y }{p(x)} \bigg|
	\\ &= 2 \sup_{C \in \mathcal{B}(\mathbb{Y}) } \bigg| -\frac{\int_{y \in C} \overline{\mathcal{S}}(y | x) \mathrm{d}y }{p(x)} \cdot g(x)  + \mathbbm{1}_{y_{0} \in C} \cdot g(x)  \bigg|
	\overset{\2}{\leq} 4g(x).
\end{align*}
In step $\1$, we use the definition $\int_{y \in C} f(y|x) \mathrm{d}x = \Pr\{Y \in C\}$ and apply Lemma~\ref{lem:tech-prob}. We also use the definition $\widehat{\mathcal{T}}(y|x) = \overline{\mathcal{S}}(y | x)/p(x)$. In step $\2$, we apply the triangle inequality and use the fact that $\frac{\int_{y \in C} \overline{\mathcal{S}}(y | x) \mathrm{d}y }{p(x)} \leq 1$. Putting together the pieces, we have 
\begin{align*}
	\| \mathcal{L} \left[ \textsc{rk}(X_{\theta}, N, M, y_0) \right] - (\widehat{\mathcal{T}} \circ \mathcal{U}) (\cdot | \theta) \|_{1}  \leq  \int_{x \in \mathbb{X}} 4g(x) \cdot u(x;\theta) \mathrm{d}x &\leq \sup_{x \in \mathbb{X}} 4 g(x) 
	 \overset{\1}{\leq} \sup_{x \in \mathbb{X}}\; 4\exp\left\{ -\frac{p(x)}{M} \cdot N \right\},
\end{align*}
where in step $\1$ we use $g(x) = (1-p(x)/M)^{N} \leq e^{-p(x)N/M}$. This concludes the proof of claim~\eqref{clm:sim}.

\paragraph*{Proof of claim~\eqref{clm:trunc}.} We have
\begin{align*}
\| (\widehat{\mathcal{T}} \circ \mathcal{U}) (\cdot | \theta) - (\mathcal{S}^* \circ \mathcal{U}) (\cdot | \theta) \|_{1}  &= \int_{\mathbb{Y}} \left| \int_{\mathbb{X}} (\mathcal{S}^*(y| x) - \widehat{\mathcal{T}}(y|x)) u(x; \theta) \mathrm{d}x \right| \mathrm{d}y \\
&\overset{\1}{\leq}    \int_{\mathbb{X}} \int_{\mathbb{Y}} \left| (\mathcal{S}^*(y| x) - \widehat{\mathcal{T}}(y|x)) u(x; \theta) \right| \mathrm{d}y \mathrm{d}x    \\
&\overset{\2}{=} \int_{\mathbb{X}} \int_{\mathbb{Y}} \left| \left( \mathcal{S}^*(y| x) - \frac{\mathcal{S}^*(y|x) \vee 0}{p(x)} \right) u(x; \theta) \right| \mathrm{d}y \mathrm{d}x, 
\end{align*}
where step $\1$ follows as before from Jensen's inequality and Fubini's theorem, and step $\2$ from the definition of $\widehat{\mathcal{T}}(y|x)$ and $p(x)$. Splitting the integral over $\mathbb{Y}$ into two parts corresponding to whether or not $\mathcal{S}^*(y|x)$ is positive, we have
\begin{align*}
&\int_{\mathbb{Y}} \left| \left( \mathcal{S}^*(y| x) - \frac{\mathcal{S}^*(y|x) \vee 0}{p(x)} \right) u(x; \theta) \right| \mathrm{d}y \\
&= \int_{\mathbb{Y}} \frac{|p(x) - 1|}{p(x)} \mathcal{S}^*(y| x) u(x; \theta) \cdot \ind{\mathcal{S}^*(y|x) \geq 0} \mathrm{d}y + \int_{\mathbb{Y}} [- \mathcal{S}^*(y| x)] u(x; \theta) \cdot \ind{\mathcal{S}^*(y|x) < 0} \mathrm{d}y \\
&= u(x; \theta) \cdot \frac{|p(x) - 1|}{p(x)}  \int_{\mathbb{Y}}  [\mathcal{S}^*(y| x) \vee 0 ] \mathrm{d}y + u(x; \theta) \cdot \int_{\mathbb{Y}} [- \mathcal{S}^*(y| x) \wedge 0]  \mathrm{d}y \\
&=  u(x; \theta) \cdot |p(x) - 1| + u(x | \theta) \cdot q(x).
\end{align*}
Integrating over $x \in \mathbb{X}$ completes the proof of claim~\eqref{clm:trunc}.

Putting together the two claims completes the proof. 
\qed


\subsection{Proof of Proposition~\ref{prop:Laplace}} \label{sec:pf-propLap}
We provide two versions of the proof. The first is a heuristic proof sketch based on the Fourier inversion theorem, which shows a constructive procedure for the claimed reduction---however, making it rigorous would require stronger assumptions than those stated. Our second technique is rigorous and based on integration by parts, and proved under the weak assumptions~\eqref{assump-target-laplace} stated in the proposition. Recall $u(x;\theta)$ in Eq.~\eqref{eq:Laplace-location}. Since the theorem is claimed for all $\Theta \subseteq \real^d$, it suffices to verify the equation
\begin{align} \label{eq:Fourier1}
	\int_{\real^{d}} u(x;\theta) \mathcal{S}^{*}(y|x) \mathrm{d}x = v(y;\theta), \quad \text{for all} \quad  y\in \mathbb{Y}, \theta\in \real^d.
\end{align}
\subsubsection{Proof sketch based on Fourier Inversion}
Define the shorthand $h(y;\theta) = \int_{\real^d} u(x;\theta) \cdot \mathcal{S}^*(y| x) \mathrm{d}x$ for convenience, and let
\begin{align*}
 g_{1}(z) = \prod_{i=1}^{d} \frac{e^{-|z_i|/b_i}}{2b_{i}}, \quad g_{2}(z) = \mathcal{S}^{*}(y|z), \quad \text{for all} \quad z \in \real^d.
\end{align*}
By definition, 
$
	h(y;\theta) = \int_{\real^d} g_{1}(x-\theta) g_{2}(x) \mathrm{d}x = \big( g_{1} * g_{2} \big) (\theta).
$
Taking Fourier transforms and applying the convolution theorem, we obtain
\begin{align}\label{step:convolution-lap}
	\mathcal{F} \{ h(y;\cdot)\} (\omega) \overset{\1}{=} \mathcal{F} \{ g_{1} \} (\omega) \cdot \mathcal{F} \{ g_{2} \} (\omega).
\end{align}
By definition of the two functions, we have
\begin{align*}
	&\mathcal{F} \{ g_{1} \} (\omega) = \int_{\real^d} e^{i \langle \omega, z \rangle} g_{1}(z) \mathrm{d}z = \prod_{i = 1}^d \frac{1}{1 + b_{i}^{2}\omega_i^2} \quad  \text{and}\\
	&\mathcal{F} \{ g_{2} \} (\omega) = \int_{\real^d} e^{i \langle \omega, z \rangle} \mathcal{S}^{*}(y|z) \mathrm{d}z = \prod_{i = 1}^d (1 + b_{i}^{2}\omega_i^2) \cdot \mathcal{F}\{v(y;\cdot)\}.
\end{align*}
Multiplying the above two displays and using Eq.~\eqref{step:convolution-lap} yields $\mathcal{F} \{ h(y;\cdot)\} (\omega) = \mathcal{F}\{v(y;\cdot)\}$. Taking inverse Fourier transforms, we then obtain $h(y;\theta) \overset{\2}{=} v(y;\theta)$ for all $\theta \in \real^{d}$, as claimed. 

Here step $\1$ relies on the assumption that the convolution theorem holds for $g_1$ and $g_2$, and step $\2$ relies on the assumption that the Fourier inversion theorem holds. Standard sufficient conditions for these steps are stronger than those stated in the theorem---we next provide a rigorous proof.

\subsubsection{Proof based on integration by parts}

With the intuition gained from the previous sketch, we can now directly verify that Eq.~\eqref{eq:Fourier1} holds, and under weaker conditions. Our proof proceeds via integration by parts and requires the following technical lemma. We provide proof of Lemma~\ref{lemma:aux-pf-prop-laplace} in Section~\ref{sec:pf-lemma-prop-laplace-integral}.
 
\begin{lemma}\label{lemma:aux-pf-prop-laplace}
Consider the setting of Proposition~\ref{prop:Laplace} and suppose assumption~\eqref{assump-target-laplace} holds. For each $x\in \real^{d}$, $y \in \mathbb{Y}$, $\theta \in \Theta$, and $1\leq k \leq d$, we have
\begin{align*}
	\int_{\real} \frac{1}{2b_{k}}\exp\big( - |x_{k} - \theta_{k}| / b_{k} \big) \cdot \bigg( v_{k}(y;x) - b_{k}^{2} \frac{\partial^{2} v_{k}(y;x) }{\partial x_{k}^{2}} \bigg) \mathrm{d} x_{k} = v_{k}(y;x) \mid_{x_{k} = \theta_{k}}.
\end{align*} 
\end{lemma}

We take this lemma as given for the moment and prove the proposition.
By definition of $\mathcal{S}^{*}(y|x)$~\eqref{S-star-Laplace} and $v_{k}(y;x)$~\eqref{eq:v-funcs-laplace}, we obtain 
\begin{align}
	\mathcal{S}^{*}(y|x) = v_{1}(y;x) - b_{1}^{2}\frac{\partial^{2} v_{1}(y;x)}{ \partial x_{1}^2}.
\end{align}
Thus
\begin{align}\label{eq1:pf-lap-kernel-integral}
	&\int_{\real^d} u(x;\theta) \cdot \mathcal{S}^*(y| x) \mathrm{d}x  \notag \\
	&\qquad = \int_{\real^{d-1}} \prod_{i=2}^{d} \frac{\exp(-|x_i - \theta_i| /b_i ) }{2b_i} \int_{\real} \frac{\exp(-|x_1 - \theta_1|/b_1) }{2b_1}  \bigg( v_{1}(y;x) - b_{1}^{2}\frac{\partial^{2} v_{1}(y;x)}{ \partial x_{1}^2} \bigg) \mathrm{d}x_{1}  \prod_{i=2}^{d} \mathrm{d}x_{i} \nonumber \\
	&\qquad = \int_{\real^{d-1}} \prod_{i=2}^{d} \frac{\exp(-|x_i - \theta_i| /b_i ) }{2b_i} v_{1}(y;x)\mid_{x_{1} = \theta_{1}}   \prod_{i=2}^{d} \mathrm{d}x_{i}
\end{align}
where in the last step we apply Lemma~\ref{lemma:aux-pf-prop-laplace} for $k=1$. Continuing, we have by definition that
\[
	v_{1}(y;x) \mid_{x_1 = \theta_1} = v_{2}(y;x) \mid_{x_1 = \theta_1} - b_{2}^{2} \frac{\partial^{2} v_{2}(y;x) }{ \partial x_{2}^{2}} \Big\vert_{x_1 = \theta_1}.
\]
Substituting the equation in the above display into Eq.~\eqref{eq1:pf-lap-kernel-integral} yields
\begin{align*}
	&\int_{\real^d} u(x;\theta) \cdot \mathcal{S}^*(y| x) \mathrm{d}x \\
	&= \int_{\real^{d-2}} \prod_{i=3}^{d} \frac{\exp(-|x_i - \theta_i| /b_i ) }{2b_i}   \int_{\real} \frac{\exp(-|x_2 - \theta_2|/b_2) }{2b_2} \Big( v_{2}(y;x) \mid_{x_1 = \theta_1} - b_{2}^{2} \frac{\partial^{2} v_{2}(y;x) }{ \partial x_{2}^{2}} \Big\vert_{x_1 = \theta_1}  \Big)  \mathrm{d}x_{2}    \prod_{i=3}^{d} \mathrm{d}x_{i} \\
	& = \int_{\real^{d-2}} \prod_{i=3}^{d} \frac{\exp(-|x_i - \theta_i| /b_i ) }{2b_i} \cdot   v_{2}(y;x)\mid_{x_1 = \theta_1, x_2 = \theta_2} \prod_{i=3}^{d} \mathrm{d}x_{i},
\end{align*} 
where in the last step we apply Lemma~\ref{lemma:aux-pf-prop-laplace} for $k=2$. Repeating the same process for $k=3,\dots,d-1$, we obtain
\begin{align*}
	&\int_{\real^d} u(x;\theta) \cdot \mathcal{S}^*(y| x) \mathrm{d}x \\
	&\quad = \int_{\real} \frac{\exp( - | x_{d} - \theta_d |/ b_{d} ) }{2b_{d}} \cdot \bigg( v_{d}(y;x)\mid_{x = [\theta_1\vert \dots \vert \theta_{d-1} \vert x_d]^{\top}} - b_{d}^{2} \frac{\partial^{2} v_{d}(y;x)  }{ \partial x_{d}^2} \Big \vert_{x = [\theta_1\vert \dots \vert \theta_{d-1} \vert x_d]^{\top}}  \bigg) \mathrm{d}x_{d} \\ 
	&\quad \overset{\1}{=} v_{d}(y;x) \mid_{x = \theta} = v(y;\theta)
\end{align*}
where in step $\1$ we apply Lemma~\ref{lemma:aux-pf-prop-laplace} for $k=d$, and the last step follows by the definition~\eqref{eq:v-funcs-laplace} of $v_{d}(y;x)$. This completes the proof. 
\qed


\subsection{Proof of Theorem~\ref{thm:Lap-Gaussian}} \label{sec:pf-thmLapGauss}

Let $\phi_{\sigma}(\cdot)$ denote the zero mean Gaussian pdf with standard deviation $\sigma$, and note that $v(y; \theta) = \phi_{\sigma}(y - \theta)$. Clearly, target $v(y; \theta) = \phi_{\sigma}(y - \theta)$ satisfies assumption~\eqref{assump-target-laplace}. Thus, from Eq.~\eqref{S-star-Laplace}, the signed kernel
\begin{align}\label{S-star-Laplace-general}
	\mathcal{S}^*(y | x ) = \phi_{\sigma}(y-x) \Big( 1 + b^{2}\sigma^{-2} - b^{2} \sigma^{-4}(y-x)^{2} \Big).
\end{align}
satisfies $\SignedDef(\mathcal{U}, \mathcal{V};\mathcal{S}^*) = 0$.

\paragraph*{Bounding the Radon--Nikodym derivative:} The truncated kernel $\mathcal{S}^* \vee 0$ satisfies, for $\sigma \geq b$, the relation
\begin{align}\label{ineq:S-P-ratio-Lap-b}
	\frac{ \mathcal{S}^{*}(y|x) \vee 0}{\mathcal{P}(y|x)} = (1 + b^{2}\sigma^{-2} - b^{2} \sigma^{-4}(y-x)^{2} ) \vee 0 \leq 2,  \quad \text{for all} \quad x,y \in \real.
\end{align}
Thus, we may set $M = 2$ in the algorithm.

\paragraph*{Establishing claim~\eqref{eq:Lap-Gaussian-TV}:} We use the shorthand $\Delta = \sigma\sqrt{\sigma^2 + b^2}/b$. It suffices to evaluate the quantities $p(x)$ and $q(x)$ from Eq.~\eqref{eq:pq}. For each $x \in \real$, we have
\begin{align}\label{eq:px-qx-laplace-gauss-b}
p(x) &= \int_{- \Delta}^{ \Delta} \phi_{\sigma}(t) \left\{ 1 + b^2\sigma^{-2} -  b^2\sigma^{-4} t^2 \right\} \mathrm{d}t  \nonumber \\
&\overset{\1}{=} \int_{- \infty}^{\infty} \phi_{\sigma}(t) \left\{ 1 + b^2\sigma^{-2} - b^2\sigma^{-4} t^2 \right\} + 2 \int_{\Delta}^{\infty} \phi_{\sigma}(t) \left\{ b^2  \sigma^{-4} t^2 - 1 - b^2\sigma^{-2} \right\} \mathrm{d}t  \nonumber \\
&= 1 + 2 \int_{\Delta}^{\infty} \phi_{\sigma}(t) \left\{ b^2\sigma^{-4}t^2 - 1 - b^2 \sigma^{-2} \right\} \mathrm{d}t,
\end{align}
where step $\1$ follows due to symmetry. We also have
\begin{align*}
q(x) &= 2 \int_{\Delta}^{\infty} \phi_{\sigma}(t) \left\{ b^2 \sigma^{-4} t^2  - 1 - b^2 \sigma^{-2} \right\} \mathrm{d}t.
\end{align*}
Putting together the pieces and noting that $ b^{2} \sigma^{-4}t^{2} - 1 - b^{2}\sigma^{-2} \geq 0$ when $t \geq \Delta$, we have
\begin{align}\label{px-qx-bound-lap-gauss-general}
	|p(x)-1| + q(x) \leq 4\int_{\Delta}^{+\infty} \phi_{\sigma}(t) \Big( b^{2} \sigma^{-4}t^{2} - 1 - b^{2}\sigma^{-2} \Big) \mathrm{d}t &\leq \frac{4b^2}{\sigma^4} \int_{\Delta}^{+\infty} \phi_{\sigma}(t) t^{2} \mathrm{d}t \nonumber \\
	& \overset{\1}{=} \frac{4b^2}{\sigma^4} \Big( \frac{\Delta \sigma}{\sqrt{2\pi} } e^{-\frac{\Delta^2}{2\sigma^2}} + \sigma^2 \int_{\Delta}^{+\infty} \phi_{\sigma}(t) \mathrm{d}t \Big) \nonumber \\
	& \overset{\2}{\leq} \frac{2(b\sqrt{\sigma^2+b^2} + b^2)}{\sigma^2} \cdot e^{-\frac{\sigma^2}{2b^2}} \nonumber \\
	& \leq 6 e^{-\frac{\sigma^2}{2b^2}},
\end{align}
where in step $\1$ we use integration by parts, in step $\2$ we use $\int_{\Delta}^{+\infty} \phi_{\sigma}(t) \mathrm{d}t \leq \frac{1}{2} e^{-\frac{\Delta^2}{2\sigma^2}}$, and in the last step we use $\sigma\geq b$. Note that $p(x) \geq 1$ from inequality~\eqref{eq:px-qx-laplace-gauss-b}. With the bound in the above display and Eq.~\eqref{ineq:S-P-ratio-Lap-b}, applying Lemma~\ref{lem:rej-sampling} yields that for $X_{\theta} \sim \mathsf{Lap}(\theta,b)$, $M=2$, $N\geq 1$, any $y_0 \in \real$ and $\sigma\geq b$, we have
\begin{align}
	\sup_{\theta \in \Theta} \; \| \mathcal{L} \left[ \textsc{rk}(X_{\theta}, N, M, y_0) \right] - v( \cdot ; \theta) \|_{\mathsf{TV}}  \leq 2\exp\left\{-\frac{N}{2} \right\} + 6\exp\left\{ -\frac{\sigma^2}{2b^2} \right\}.
\end{align}
Substituting the value of $N$ completes the proof.
\qed


\subsection{Proof of Proposition~\ref{prop:Erlang}} \label{sec:pf-Erlang}
As before, we provide a heuristic proof sketch based on the Fourier inversion theorem to demonstrate a constructive procedure for the reduction. We then provide a rigorous proof based on integration by parts under the weak assumptions~\eqref{eq:assump-prop-Erlang} stated in the proposition. We must verify for all $\theta \in \Theta$ and $y \in \mathbb{Y}$ the equality
\begin{align}\label{eq:erlang-kernel-source-int}
\int_{\theta}^{\infty} \frac{\lambda^k}{(k - 1)!} \cdot (x - \theta)^{k - 1} e^{-\lambda (x - \theta)} \cdot \mathcal{S}^*(y| x) \mathrm{d}x = v(y; \theta).
\end{align}
In order to prove this for all $\Theta \subseteq \real$, it suffices to do so for $\Theta = \real$.

\subsubsection{Proof sketch based on Fourier Inversion}
To bring it into a more algebraically convenient for, introduce the variable $\widetilde{x} := x - \theta$, kernel $\mathcal{K}^*(y | x) := S^*(y| -x)$, and $\widetilde{v}(y; \theta) := v(y; -\theta)$. Then Eq.~\eqref{eq:erlang-kernel-source-int} is equivalent to
\begin{align*}
\int_{0}^{\infty} \frac{\lambda^k}{(k - 1)!} \cdot \widetilde{x}^{k - 1} e^{-\lambda \widetilde{x}} \cdot \mathcal{K}^*(y| \theta - \widetilde{x} ) \mathrm{d}\widetilde{x} = \widetilde{v}(y; \theta).
\end{align*}
Define the shorthand $h(y;\theta) = \int_{0}^{\infty} \frac{\lambda^k}{(k - 1)!} \cdot \widetilde{x}^{k - 1} e^{-\lambda \widetilde{x}} \cdot \mathcal{K}^*(y| \theta - \widetilde{x} ) \mathrm{d}\widetilde{x}$ for convenience, and define for all $z \in \real$, 
\[
	f_{1}(z) = \frac{\lambda^k}{(k - 1)!} \cdot z^{k - 1} e^{-\lambda z} \ind{z \geq 0}, \quad f_{2}(z) = \mathcal{K}^*(y| z ).
\]
By definition, we have $h(y;\theta) = \big(f_{1} * f_{2} \big)(\theta)$. Taking the Fourier transform and applying the convolution theorem yields 
$
\mathcal{F}\{ h(y;\cdot) \}(\omega) \overset{\1}{=}  \mathcal{F}\{ f_{1} \}(\omega) \cdot \mathcal{F}\{ f_{2} \}(\omega).
$
But by the definition of the functions $f_1$ and $f_2$, we obtain
\begin{align*}
&\mathcal{F} \left\{ f_{1} \right\}(\omega) = \int_{\real} e^{iwz} \cdot f_{1}(z) \mathrm{d}z =\left(1 - \frac{i\omega}{\lambda} \right)^{-k}, \\ 
&\mathcal{F}\{ f_{2} \}(\omega) = - \int_{\real} e^{-iwz} \cdot  \mathcal{S}^{*}(y|z) \mathrm{d}z  = \left(1 - \frac{i\omega}{\lambda} \right)^{k} \cdot \mathcal{F}\{ \widetilde{v}(y;\cdot) \}(\omega),
\end{align*}
where in the last step we use the definition of $\mathcal{S}^{*}(y|\cdot)~\eqref{eq:signed-kernel-Erlang}$. 
Thus, $\mathcal{F}\{ h(y;\cdot) \}(\omega) = \mathcal{F}\{ \widetilde{v}(y;\cdot) \}(\omega)$. Taking the inverse Fourier transform yields
$h(y;\theta) \overset{\2}{=} \widetilde{v}(y;\theta)$.

Here step $\1$ relies on the assumption that the convolution theorem holds for $f_1$ and $f_2$, and step $\2$ relies on the assumption that the Fourier inversion theorem holds.

\subsubsection{Proof based on integration by parts} 
With the intuition gained from the previous sketch, we now rigorously verify Eq.~\eqref{eq:erlang-kernel-source-int} via a careful integration by parts argument.
Note that Eq.~\eqref{eq:signed-kernel-Erlang} is equivalent to
\begin{align*}
	\mathcal{S}^{*}(y|x) = \sum_{j=0}^{k} \binom{k}{j} (-1)^{j}\lambda^{-j} \cdot \nabla^{(j)}_t v(y; t) \Big|_{t = x}. 
\end{align*}
To reduce the notational burden, we define functions $g,g_{1},g_{2}: \real \rightarrow \real$ as
\[
	g_{1}(x) = \exp\big(-\lambda(x - \theta)\big), \quad g_{2}(x) = (x-\theta)^{k-1}, \quad g(x) = g_{1}(x) \cdot g_{2}(x) \quad \text{for all} \quad x \in \real.
\]
We must verify Eq.~\eqref{eq:erlang-kernel-source-int}. Using the notation above, we obtain
\begin{align}\label{eq:T1-T2}
	&\int_{\theta}^{+\infty} (x-\theta)^{k-1}e^{-\lambda(x-\theta)} \mathcal{S}^{*}(y|x) \mathrm{d}x =  \frac{\lambda^{k}}{(k-1)!} \cdot \big( T_{1} + T_{2} \big), \quad \text{where}  \nonumber \\ 
	&T_{1} = \int_{\theta}^{+\infty} g(x) \cdot v(y|x) \mathrm{d}x \quad
	\text{and} \quad  T_{2} = \sum_{j=1}^{k} \binom{k}{j} (-1)^{j}\lambda^{-j}  \int_{\theta}^{+\infty} g(x) \cdot \nabla^{(j)}_t v(y; t) \Big|_{t = x} \mathrm{d}x.
\end{align}
Note that the following claim suffices to verify Eq.~\eqref{eq:erlang-kernel-source-int}:
\begin{subequations}\label{claim:T1-T2-Erlang}
\begin{align}
	\label{claim:T1-Erlang}
	T_{1} &= \frac{(k-1)!}{\lambda^k} \cdot  v(y;\theta) + \sum_{i=0}^{k-1}\frac{1}{\lambda^{i+1}} \int_{\theta}^{+\infty} g_{1}(x) g_{2}^{(i)}(x) \nabla_{t}v(y;t)\Big \vert_{t=x} \mathrm{d}x, \\
	T_{2} &= - \sum_{i=0}^{k-1}\frac{1}{\lambda^{(i+1)}}  \int_{\theta}^{+\infty} g_{1}(x) \cdot g_{2}^{(i)}(x) \cdot \nabla_{t} v(y;t) \mid_{t=x} \mathrm{d}x. \label{claim:T2-Erlang}
\end{align}
\end{subequations}

Thus, we dedicate the rest of the proof to establishing Eq.~\eqref{claim:T1-T2-Erlang}. We prove each part in turn. Let $g'$ be the first-order derivative of $g$, and use $g^{(i)}$ to denote $i$-th order derivative of $g$ for $i\geq 0$. We use analogous notation for the functions $g_{1}$ and $g_{2}$.

\noindent \underline{Proof of Eq.~\eqref{claim:T1-Erlang}:} By definition of $T_{1}$, we obtain
\begin{align}\label{ineq1:T1-int-parts}
	T_{1} & = \int_{\theta}^{+\infty} (-\lambda)^{-1} g_{1}'(x) g_{2}(x) \cdot v(y;x) \mathrm{d} x \nonumber \\
	& \overset{\1}{=} (-\lambda)^{-1} g_{1}(x)g_{2}(x) v(y;x) \Big\vert_{\theta}^{+\infty} - (-\lambda)^{-1} \int_{\theta}^{+\infty} g_{1}(x) \Big( g_{2}'(x) v(y;x) + g_{2}(x) \cdot \nabla_{t} v(y;t)\Big \vert_{t=x} \Big) \mathrm{d}x \nonumber \\
	& \overset{\2}{=} \frac{1}{\lambda} \int_{\theta}^{+\infty} g_{1}(x) g_{2}'(x) v(y;x) \mathrm{d}x + \frac{1}{\lambda} \int_{\theta}^{+\infty} g_{1}(x)g_{2}(x) \nabla_{t} v(y;t)\Big \vert_{t=x} \mathrm{d}x,
\end{align}
where in step $\1$ we apply integration by parts and in step $\2$ we use assumptions~\eqref{eq1:assump-prop-Erlang} and~\eqref{eq2:assump-prop-Erlang}.
Continuing by applying the same steps to the first term of the RHS of inequality~\eqref{ineq1:T1-int-parts}, we obtain
\begin{align}\label{ineq2:T1-int-parts}
	 \int_{\theta}^{+\infty} g_{1}(x) &g_{2}'(x) v(y;x) \mathrm{dx} = (-\lambda)^{-1} \int_{\theta}^{+\infty} g_{1}'(x) g_{2}'(x) v(y;x) \mathrm{d}x \nonumber \\
	& \overset{\1}{=} \lambda^{-1} g_{1}(x)g_{2}'(x) v(y;x) \Big \vert_{x = \theta}^{+\infty} + \lambda^{-1} \int_{\theta}^{+\infty} g_{1}(x)\Big( g_{2}^{(2)}(x)v(y;x) + g_{2}'(x) \nabla_{t}v(y;t)\Big \vert_{t=x} \Big) \mathrm{d}x \nonumber \\
	& \overset{\2}{=} \frac{1}{\lambda} \int_{\theta}^{+\infty}g_{1}(x) g_{2}^{(2)}(x) v(y;x) \mathrm{d}x + \frac{1}{\lambda} \int_{\theta}^{+\infty} g_{1}(x) g_{2}'(x) \nabla_{t}v(y;t)\Big \vert_{t=x} \mathrm{d}x,
\end{align}
where in step $\1$ we apply integration by parts and in step $\2$ we use assumptions~\eqref{eq1:assump-prop-Erlang} and ~\eqref{eq2:assump-prop-Erlang}.
Putting inequalities~\eqref{ineq1:T1-int-parts} and~\eqref{ineq2:T1-int-parts} together yields
\begin{align*}
 T_{1} = \frac{1}{\lambda^2} \int_{\theta}^{+\infty}g_{1}(x) g_{2}^{(2)}(x) v(y;x) \mathrm{d}x + \sum_{i=0}^{1} \frac{1}{\lambda^{i+1}} \int_{\theta}^{+\infty} g_{1}(x) g_{2}^{(i)}(x) \nabla_{t}v(y;t)\Big \vert_{t=x} \mathrm{d}x.
\end{align*}
Repeating the same process $k-1$ times, we obtain
\begin{align*}
	T_{1} &= \frac{1}{\lambda^{k-1}} \int_{\theta}^{+\infty}g_{1}(x) g_{2}^{(k-1)}(x) v(y;x) \mathrm{d}x + \sum_{i=0}^{k-2}\frac{1}{\lambda^{i+1}} \int_{\theta}^{+\infty} g_{1}(x) g_{2}^{(i)}(x) \nabla_{t}v(y;t)\Big \vert_{t=x} \mathrm{d}x \\
	& =  \frac{(k-1)!}{\lambda^k} \cdot  v(y;\theta) + \sum_{i=0}^{k-1}\frac{1}{\lambda^{i+1}} \int_{\theta}^{+\infty} g_{1}(x) g_{2}^{(i)}(x) \nabla_{t}v(y;t)\Big \vert_{t=x} \mathrm{d}x,
\end{align*}
where in the last step we apply integration by parts once more so that
\begin{align*}
	\int_{\theta}^{+\infty}g_{1}(x) g_{2}^{(k-1)}(x) v(y;x) \mathrm{d}x &= (k-1)! \cdot (-\lambda)^{-1}\int_{\theta}^{+\infty}g_{1}'(x) v(y;x) \mathrm{d}x \\
	&=(k-1)! \cdot \lambda^{-1} v(y;\theta) + \lambda^{-1} \int_{\theta}^{+\infty}  g_{1}(x) g_{2}^{(k-1)}(x) \nabla_{t} v(y;t) \Big \vert_{t=x} \mathrm{d}x.
\end{align*} 
Eq.~\eqref{claim:T1-Erlang} is thus proved.
Towards proving Eq.~\eqref{claim:T2-Erlang}, we require the following combinatorial result, whose proof we defer to Section~\ref{sec:pf-lemma-combinatorial}.
\begin{lemma}\label{lemma:combinatorial}
For all $k\geq 1$ and $0\leq \ell \leq k-1$, we have
\begin{align}\label{eq:T2-Erlang-comb}
	\sum_{j= \ell}^{k-1} \binom{k}{j+1}  (-1)^{j-\ell} \binom{j}{j-\ell} = 1.
\end{align}
\end{lemma}

\noindent \underline{Proof of Eq.~\eqref{claim:T2-Erlang}:} For $2\leq j \leq k$, we obtain 
\begin{align*}
	\int_{\theta}^{+\infty} g(x) \nabla_{t}^{(j)} v(y;t) \mid_{t=x} \mathrm{d}x &= g(x) \nabla_{t}^{(j-1)} v(y;t) \mid_{t=x} \Big \vert_{\theta}^{+\infty} - \int_{\theta}^{+\infty} g'(x) \nabla_{t}^{(j-1)} v(y;t) \mid_{t=x} \mathrm{d}x \\
	& = - \int_{\theta}^{+\infty} g'(x) \nabla_{t}^{(j-1)} v(y;t) \mid_{t=x} \mathrm{d}x,
\end{align*}
where in the last step we use assumptions~\eqref{eq1:assump-prop-Erlang} and~\eqref{eq3:assump-prop-Erlang}. Repeating the same process $j-1$ times yields
\begin{align}\label{eq1:aux-step-pf-prop-Erlang}
	\int_{\theta}^{+\infty} g(x) \cdot & \nabla_{t}^{(j)} v(y;t) \mid_{t=x} \mathrm{d}x = (-1)^{j-1} \int_{\theta}^{+\infty} g^{(j-1)}(x) \cdot  \nabla_{t} v(y;t) \mid_{t=x} \mathrm{d}x  \nonumber\\
	& = (-1)^{j-1} \sum_{i=0}^{j-1} (-\lambda)^{i} \binom{j-1}{i} \int_{\theta}^{+\infty} g_{1}(x) \cdot g_{2}^{(j-i-1)}(x) \cdot \nabla_{t} v(y;t) \mid_{t=x} \mathrm{d}x,
\end{align}
where in the last step we use $g(x) = g_{1}(x) g_{2}(x)$ so that
\[
	g^{(j-1)}(x) = \sum_{i=0}^{j-1} \binom{j-1}{i} g_{1}^{(i)}(x) \cdot g_{2}^{(j-i-1)}(x) = \sum_{i=0}^{j-1} (-\lambda)^{i} \binom{j-1}{i} g_{1}(x) \cdot g_{2}^{(j-i-1)}(x). 
\]
Substituting Eq.~\eqref{eq1:aux-step-pf-prop-Erlang} into Eq.~\eqref{eq:T1-T2}, we obtain
\begin{align*}
	T_{2} & = \sum_{j=1}^{k} \binom{k}{j} (-1)^{j}\lambda^{-j}  (-1)^{j-1} \sum_{i=0}^{j-1} (-\lambda)^{i} \binom{j-1}{i} \int_{\theta}^{+\infty} g_{1}(x) \cdot g_{2}^{(j-i-1)}(x) \cdot \nabla_{t} v(y;t) \mid_{t=x} \mathrm{d}x \\
	& = - \sum_{j=0}^{k-1} \binom{k}{j+1} \lambda^{-(j+1)}  \sum_{i=0}^{j} (-\lambda)^{i} \binom{j}{i} \int_{\theta}^{+\infty} g_{1}(x) \cdot g_{2}^{(j-i)}(x) \cdot \nabla_{t} v(y;t) \mid_{t=x} \mathrm{d}x,
\end{align*}
where in the last step we change the index by letting $j \gets j+1$. Note that in the above summation, for each $0\leq \ell \leq k-1$, the coefficient of the integral $\int_{\theta}^{+\infty} g_{1}(x) \cdot g_{2}^{(\ell)}(x) \cdot \nabla_{t} v(y;t) \mid_{t=x} \mathrm{d}x$ is
\[
	 - \sum_{j= \ell}^{k-1} \binom{k}{j+1} \lambda^{-(j+1)} (-\lambda)^{j-\ell} \binom{j}{j-\ell} = - \lambda^{-(\ell+1)} \sum_{j= \ell}^{k-1} \binom{k}{j+1}  (-1)^{j-\ell} \binom{j}{j-\ell} = - \lambda^{-(\ell+1)},
\]
where in the last step we use Eq.~\eqref{eq:T2-Erlang-comb}. Consequently, we obtain 
\[
	T_{2} = - \sum_{\ell=0}^{k-1}\frac{1}{\lambda^{(\ell+1)}} \int_{\theta}^{+\infty} g_{1}(x) \cdot g_{2}^{(\ell)}(x) \cdot \nabla_{t} v(y;t) \mid_{t=x} \mathrm{d}x,
\]
as claimed.
\qed


\subsection{Proof of Theorem~\ref{thm:exp-log-concave}} \label{sec:pf-ExpLog}

Note that the exponential distribution $u(x;\theta)$~\eqref{eq:exponential-location} is an Erlang distribution with $\lambda = k = 1$ and mean $\theta$. Note that the target satisfies assumption~\eqref{eq:assump-prop-Erlang} since for each $w \in \real$, we have
\[
	\lim_{x \rightarrow +\infty} e^{-x} v(y;\theta) \big\vert_{\theta = x+w} = \lim_{x \rightarrow +\infty}  \frac{e^{-x}}{\sigma} \exp\Big( -\psi\Big( \frac{y-x-w}{\sigma} \Big) \Big) = 0,
\]
where the last step holds since $\lim_{z \rightarrow +\infty} \sigma^{-1}e^{-\psi(z/\sigma)} = 0$. Consequently, applying Proposition~\ref{prop:Erlang} yields that the signed kernel 
\begin{align*}
	\mathcal{S}^{*}(y|x) &= v(y;x+1) - \nabla_{\theta}v(y;\theta) \mid_{\theta = x+1}
	\\ & = \frac{1}{\sigma} \exp\left\{-\psi\Big(\frac{y-x-1}{\sigma} \Big) \right\} \cdot \bigg( 1 - \frac{1}{\sigma}\psi'\Big( \frac{y-x-1}{\sigma}\Big) \bigg)
\end{align*}
satisfies $\SignedDef(\mathcal{U},\mathcal{V};\mathcal{S}^*) = 0$.   

Recall the definition of $\kappa(\sigma)$ from Eq.~\eqref{eq:kappa-tau}, and note that the function $\psi'(z)$ is nondecreasing since $\psi(z)$ is convex. Thus,
\begin{align*}
	\sgn(\mathcal{S}^{*}(y|x)) =  
	\begin{cases}
	1 \quad &\text{ if } y\leq \sigma \cdot \kappa(\sigma) + x+1 \\
	-1 \quad &\text{ otherwise}.
	\end{cases}
\end{align*}
With this characterization of the sign of $\mathcal{S}^{*}$, we first verify Assumption~\ref{assptn:RND} with $\mathcal{P}(y|x) = \frac{1}{2\sigma} \exp\left\{ - \psi\Big( \frac{y-x-1}{2\sigma} \Big) \right\}$ and then prove claim~\eqref{ineq:exp-log-concave}.

\paragraph*{Bounding Radon--Nikodym derivative:} The truncated kernel $\mathcal{S}^{*} \vee 0$ satisfies
\begin{align}\label{ineq:kernel-base-ratio-exp-log-concave}
	\frac{\mathcal{S}^{*}(y|x) \vee 0}{\mathcal{P}(y|x)} &= 2\exp\left\{ -\psi\Big( \frac{y-x-1}{\sigma}\Big) + \psi\Big( \frac{y-x-1}{2\sigma} \Big) \right\} \cdot \Big( 1-\frac{1}{\sigma} \psi'\Big( \frac{y-x-1}{\sigma} \Big) \Big) \vee 0 \nonumber\\ 
	&\overset{\1}{\leq} 2\exp\left\{ \frac{\psi(0)}{2}\right\} \exp\left\{- \frac{1}{2}\psi\Big( \frac{y-x-1}{\sigma} \Big) \right\} \cdot \Big( 1-\frac{1}{\sigma} \psi'\Big( \frac{y-x-1}{\sigma} \Big) \Big) \vee 0 \nonumber\\ 
	& \overset{\2}{\leq} 2 \exp\left\{ \frac{\psi(0)}{2}\right\} \cdot \sup_{z \leq \kappa(\sigma)} \exp\left\{-\frac{\psi(z)}{2} \right\} \big(1-\psi'(z) /\sigma \big) \leq M,
\end{align}
where in step $\1$ we use the convexity of $\psi$ so that 
\[
	\psi\Big( \frac{y-x-1}{2\sigma} \Big) = \psi\Big( \frac{1}{2} \cdot \frac{y-x-1}{\sigma} + \frac{1}{2} \cdot 0 \Big) \leq \frac{1}{2} \psi\Big( \frac{y-x-1}{\sigma} \Big) + \frac{1}{2} \psi(0),
\]
and in step $\2$ we use the fact that $\mathcal{S}(y|x) \geq 0$ if and only if $y\leq \sigma \cdot \kappa(\sigma) + x+1$. 

\paragraph*{Proving inequality~\eqref{ineq:exp-log-concave}:}
By definition of $p(x)$ and $q(x)$ in Eq.~\eqref{eq:pq}, we obtain
\begin{align*}
	p(x) &= \int_{-\infty}^{\sigma \cdot \kappa(\sigma) + x + 1} \frac{1}{\sigma} \exp\left\{-\psi\Big(\frac{y-x-1}{\sigma} \Big) \right\} \cdot \bigg( 1 - \frac{1}{\sigma}\psi'\Big( \frac{y-x-1}{\sigma}\Big) \bigg) \mathrm{d}y
	\\ &\overset{\1}{=} \int_{-\infty}^{\kappa(\sigma)} e^{-\psi(z)} \cdot \big(1-\psi'(z) / \sigma \big) \mathrm{d}z
	\\& = 1 - \int_{\kappa(\sigma)}^{+\infty}  e^{-\psi(z)} \mathrm{d}z + \frac{1}{\sigma}e^{-\psi(\kappa(\sigma))},
\end{align*}
where in step $\1$ we change variables by letting $z = (y-x-1)/\sigma$. Proceeding similarly to above, we have
\[
	q(x) = \int_{\kappa(\sigma)}^{+\infty} e^{-\psi(z)} \cdot \big(\psi'(z) / \sigma  - 1 \big) \mathrm{d}z
	= \frac{1}{\sigma} e^{-\psi(\kappa(\sigma))} - \int_{\kappa(\sigma)}^{+\infty}  e^{-\psi(z)} \mathrm{d}z.
\]
Applying the triangle inequality yields that for all $x\in \real$, we have
\begin{align}\label{ineq:px-qx-bound-exp-log-concave}
	\big| p(x) - 1 \big| + q(x) \leq 2\int_{\kappa(\sigma)}^{+\infty}  e^{-\psi(z)} \mathrm{d}z + 
	\frac{2}{\sigma} e^{-\psi(\kappa(\sigma))} : = \tau(\sigma) \quad \text{and} \quad p(x) \geq 1 - \tau(\sigma).
\end{align}
Integrating it with respect to the probability measure $u(\cdot;\theta)$ and applying Lemma~\ref{lem:rej-sampling}, we obtain that for each $\theta \in \real$,
\begin{align*}
	\| \mathcal{L} \left[ \textsc{rk}(X_{\theta}, N) \right] - v(\cdot\; \theta) \|_{\mathsf{TV}} &\leq 2e^{-\frac{N}{M}\big(1-\tau(\sigma) \big)} + \int_{\mathbb{X}} \big( | p(x)-1 | + q(x) \big) \cdot u(x ; \theta) \mathrm{d}x 
	\\ &\leq  2e^{-\frac{N}{M}\big(1-\tau(\sigma) \big)} + \tau(\sigma).
\end{align*}
This concludes the proof.
\qed

\subsubsection{Calculations for Example~\ref{example-exp-gaussian}} \label{sec:pf-example1}

We evaluate the functionals $\kappa(\sigma)$ and $\tau(\sigma)$ for the Gaussian target $v(\cdot;\theta) = \NORMAL(\theta, \sigma^2)$. Since $\psi'(z) = z$, we have $\kappa(\sigma) = \inf\{ z: \psi'(z) = \sigma \} = \sigma$, and
\begin{align*}
	\tau(\sigma) = 2\int_{\kappa(\sigma)}^{+\infty} e^{-\psi(z)} \mathrm{d}z + \frac{2}{\sigma}e^{-\psi(\kappa(\sigma))} &= 2\int_{\sigma}^{+\infty} \frac{1}{\sqrt{2\pi}} e^{-\frac{z^2}{2}} \mathrm{d}z + \frac{2}{\sigma} \frac{1}{\sqrt{2\pi}}e^{-\frac{\sigma^2}{2}}
	\\& \leq 2\int_{0}^{+\infty} \frac{1}{\sqrt{2\pi}} e^{-\frac{z^2}{2}} \mathrm{d}z \cdot e^{-\frac{\sigma^2}{2}}  + e^{-\frac{\sigma^2}{2}} \leq 2e^{-\frac{\sigma^2}{2}}.
\end{align*}
For $\sigma \geq 1$, we obtain
\begin{align*}
	2 \exp\left\{ \frac{\psi(0)}{2}\right\} \cdot \sup_{z \leq \kappa(\sigma)} \exp\left\{-\frac{\psi(z)}{2} \right\} \big(1-\psi'(t) /\sigma \big) &= 2 \cdot \sup_{z\leq \sigma}\;  \exp\{-z^{2}/4\} \big(1-z/\sigma \big)
	\\& \leq 2 + 2\max_{z\geq 0}\; \exp\{-z^{2}/4\} \cdot  z/\sigma  = 2 + \frac{2\sqrt{2}}{\sqrt{e}\sigma} \leq 4,
\end{align*}
as claimed.
\qed

\subsubsection{Calculations for Example~\ref{example-exp-logistic}} \label{sec:pf-example2}

For the logistic target, we have
\[
	\psi'(z) = \frac{\pi}{\sqrt{3}} \left( \frac{1-e^{-\pi z/\sqrt{3}}}{1+e^{-\pi z/\sqrt{3}}}\right) \leq \frac{\pi}{\sqrt{3}},
\]
where the equality is achieved only as $z \to \infty$.
Consequently, if $\sigma \geq \pi/\sqrt{3}$, then $\kappa_{\sigma} = \infty$. Thus
$
	\tau_{\sigma} = 2\int_{\kappa_{\sigma}}^{+\infty} e^{-\psi(z)} \mathrm{d}z + \frac{2}{\sigma}e^{-\psi(\kappa_\sigma)} = 0.
$
Furthermore, the derivative $\psi'(z) = 0$ if and only if $z = 0$. Thus $\psi(z) \geq \psi(0)$ for all $z\in \real$. Thus, we obtain for $\sigma \geq \pi/\sqrt{3}$ that
\begin{align*}
	2 \exp\left\{ \frac{\psi(0)}{2}\right\} \cdot \sup_{z \leq \kappa(\sigma)} \exp\left\{-\frac{\psi(z)}{2} \right\} \big(1-\psi'(z) /\sigma \big) \overset{\1}{=} 2 \cdot (1 + \frac{\pi}{\sqrt{3} \sigma} ) \leq 4,
\end{align*}
where in step $\1$ we use $\exp\{-\psi(z)/2\} \leq \exp\{-\psi(0)/2\}$ and $|\psi'(z)| \leq \pi / \sqrt{3}$ for all $z\in \real$.
\qed

\subsubsection{Calculations for Example~\ref{example-exp-laplace}} \label{sec:pf-example3}
Recall our notation $\psi(z) = |z|$ and $\psi_{\eta}$~\eqref{eq:psi-epsilon}. We make the following three-part claim for each $\eta \in (0, 1)$, deferring its proof to the end of this section:
\begin{subequations} \label{claim:example-exp-laplace}
\begin{align}
	\frac{\exp\left\{-\psi_{\eta}(z)\right\}}{\frac{1}{2}\exp\left\{-\psi(z)\right\}} &\leq \exp\left\{\frac{\eta^{2}}{2} \right\}, \label{eq:ex-Lap1} \\
	 \exp\{\psi_{\eta}(0)\} &\leq 10, \text{ and } \label{eq:ex-Lap2} \\
	\big|\psi'_{\eta}(z)\big| &\leq 1 \quad \text{for all} \quad z \in \real \text{ with equality attained in the limit } z \to \infty. \label{eq:ex-Lap3}
\end{align}
\end{subequations}

Applying the last claim, we obtain $\kappa(\sigma) = + \infty$ for $\sigma \geq 1$. Consequently, using the definition of $\tau(\sigma)$, we have 
\[
	\tau(\sigma) = \int_{\kappa(\sigma)}^{+\infty} e^{-\psi_{\eta}(z)} \mathrm{d}z + \frac{2}{\sigma} e^{-\psi_{\eta}( \kappa(\sigma))} = 0,
\]
where in the last step, we use the fact that $\psi_{\eta}(z) \rightarrow + \infty$ as $z \rightarrow +\infty$, which holds true since $e^{-\psi_{\eta}(z)}$ is a probability density function. Next, we verify that $\frac{\mathcal{S}^{*}(y|x) \vee 0}{\mathcal{P}(y|x)} \leq 35$. Note that
\[
	\frac{\mathcal{S}^{*}(y|x) \vee 0}{\mathcal{P}(y|x)}  = \frac{\mathcal{S}^{*}(y|x) \vee 0}{\mathcal{P}_{\eta}(y|x)} \cdot \frac{\mathcal{P}_{\eta}(y|x)}{\mathcal{P}(y|x)}, \quad \text{where} \quad \mathcal{P}_{\eta}(y|x) = \frac{1}{2\sigma} \exp\left\{ - \psi_{\eta}\Big( \frac{y-x-1}{2\sigma} \Big) \right\}.
\]
Using inequality~\eqref{ineq:kernel-base-ratio-exp-log-concave}, we obtain 
\begin{align}
\hspace{-0.5cm}
	\frac{\mathcal{S}^{*}(y|x) \vee 0}{\mathcal{P}_{\eta}(y|x)} \leq  2 \exp\left\{ \frac{\psi_{\eta}(0)}{2}\right\} \cdot \sup_{z \leq \kappa(\sigma)} \exp\left\{-\frac{\psi_{\eta}(z)}{2} \right\} \big(1-\psi'_{\eta}(z) /\sigma \big) \leq 2\sqrt{10} \cdot e^{1/4} \cdot 2 \leq 20,
\end{align}
where in the last step we use inequality~\eqref{claim:example-exp-laplace} so that
\begin{align*}
	e^{\psi_{\eta}(0)} \leq 10,\quad e^{-\psi_{\eta}(z)} \leq e^{\eta^2/2} e^{-|z|}/2 \leq \sqrt{e}, \quad |\psi'_{\eta}(z)| \leq 1.
\end{align*}
Continuing, note that
\[	
	\frac{\mathcal{P}_{\eta}(y|x)}{\mathcal{P}(y|x)} \leq \frac{\exp\left\{ - \psi_{\eta}\Big( \frac{y-x-1}{2\sigma} \Big) \right\}}{\frac{1}{2} \exp\left\{-\psi\Big(\frac{y-x-1}{2\sigma} \Big) \right\}} \leq \sup_{z\in\real} \frac{\exp\left\{-\psi_{\eta}(z)\right\}}{\frac{1}{2}\exp\left\{-\psi(z)\right\}} \leq \exp\left\{\frac{\eta^{2}}{2} \right\} \leq \sqrt{e}.
\]
Putting the two pieces together yields
\[
	\frac{\mathcal{S}^{*}(y|x) \vee 0}{\mathcal{P}_{\eta}(y|x)} \leq 20 \sqrt{e} \leq 35.
\]
Thus, applying Theorem~\ref{thm:exp-log-concave} with $M = 35$, $N = 35\log(4/\epsilon)$ yields
\[
	\left\| \mathcal{L} \left[ \textsc{RK}(X_{\theta}, N, M, y_0) \right] - v_{\eta}(\cdot; \theta) \right\|_{\mathsf{TV}} \leq \frac{\epsilon}{2}.
\] 

Finally, we turn to verifying inequality~\eqref{ineq:approx-error}. Applying Pinsker's inequality yields
\begin{align*}
	\big( \| v(\cdot; \theta) - v_{\eta}(\cdot; \theta)  \|_{\mathsf{TV}} \big)^{2} \leq \frac{1}{2}\int_{y \in \real} v_{\eta}(y;\theta) \log \Big( \frac{v_{\eta}(y;\theta)}{v(y;\theta)} \Big) \mathrm{d}y \overset{\1}{\leq} \frac{1}{2} \int_{y \in \real} v_{\eta}(y;\theta) \log \Big( e^{\eta^{2}/2} \Big) \mathrm{d}y \leq \frac{\eta^2}{4},
\end{align*}
where in step $\1$ we use inequality~\eqref{claim:example-exp-laplace} so that
\[
	\frac{v_{\eta}(y;\theta)}{v(y;\theta)} = \frac{\exp\left\{ -\psi_{\eta}\Big( \frac{y-\theta}{\sigma} \Big) \right\}}{\frac{1}{2}\exp\left\{ -\psi\Big( \frac{y-\theta}{\sigma} \Big) \right\}} \leq \sup_{z\in\real} \frac{\exp\left\{-\psi_{\eta}(z)\right\}}{\frac{1}{2}\exp\left\{-\psi(z)\right\}} \leq \exp\left\{\frac{\eta^{2}}{2} \right\}.
\]
The desired inequality~\eqref{ineq:approx-error} immediately follows. It remains to prove claim~\eqref{claim:example-exp-laplace}, and we verify each subclaim in turn.

\paragraph*{Proof of claim~\eqref{eq:ex-Lap1}.}
By definition of $\psi_{\eta}$, we have
\begin{align}\label{eq:pdf-eta-laplace}
\hspace{-0.5cm}
	\exp\left\{ -\psi_{\eta}(z) \right\} &= \frac{1}{2} \int_{y\in \real} \phi_{\eta}(y-z) e^{-|y|} \mathrm{d}y = \frac{1}{2} \int_{y \leq 0} \phi_{\eta}(y-z) e^{y} \mathrm{d}y + \frac{1}{2} \int_{y > 0} \phi_{\eta}(y-z) e^{-y} \mathrm{d}y.
\end{align}
The first term on the RHS can be bounded as
\[
\int_{y \leq 0} \phi_{\eta}(y-z) e^{y} \mathrm{d}y = e^{\frac{|z+\eta^2|^2 - z^2}{2\eta^2}} \cdot  \int_{y\leq 0} \frac{1}{\sqrt{2\pi}\eta} e^{-\frac{|y-(z+\eta^2)|^{2}}{2\eta^2}} \mathrm{d}y = e^{z + \frac{\eta^2}{2}} \cdot \int_{z+\eta^2}^{+\infty} \phi_{\eta}(y) \mathrm{d}y,
\]
and the second term as
\[
\int_{y > 0} \phi_{\eta}(y-z) e^{-y} \mathrm{d}y = e^{\frac{|z-\eta^2|^2 - z^2}{2\eta^2}} \cdot \int_{y > 0} \frac{1}{\sqrt{2\pi}\eta} e^{-\frac{|y-(z-\eta^2)|^{2}}{2\eta^2}} \mathrm{d}y = e^{-z + \frac{\eta^2}{2}}  \cdot \int_{-z+\eta^2}^{+\infty} \phi_{\eta}(y) \mathrm{d}y.
\]
Putting the pieces together yields that for all $z\in\real$, we have
\begin{align} \label{eq:mengqi-key}
	 \exp\left\{ -\psi_{\eta}(z) \right\} = \frac{e^{z + \frac{\eta^2}{2}}}{2} \cdot \int_{z+\eta^2}^{+\infty} \phi_{\eta}(y) \mathrm{d}y + \frac{e^{-z + \frac{\eta^2}{2}} }{2} \cdot \int_{-z+\eta^2}^{+\infty} \phi_{\eta}(y) \mathrm{d}y.
\end{align}

We split the rest of the proof into two cases. 

\noindent \underline{Case $z \geq 0$.} Here, we have
\begin{align*}
	\frac{\exp\left\{ -\psi_{\eta}(z) \right\}}{\frac{1}{2}\exp\left\{-\psi(z)\right\}} = e^{\eta^2/2}e^{2z} \int_{z+\eta^2}^{+\infty} \phi_{\eta}(y) \mathrm{d}y + e^{\eta^2/2} \int_{-z+\eta^2}^{+\infty} \phi_{\eta}(y) \mathrm{d}y \leq e^{\eta^{2}/2},
\end{align*}
where in the last step we use
\begin{align}\label{ineq1:pf-claim-ex-laplace}
	e^{2z} \int_{z+\eta^2}^{+\infty} \phi_{\eta}(y) \mathrm{d}y + \int_{-z+\eta^2}^{+\infty} \phi_{\eta}(y) \mathrm{d}y &= \int_{0}^{+\infty} \frac{1}{\sqrt{2\pi} \eta} e^{-\frac{|y+z-\eta^2|^{2}}{2\eta^2}}e^{-2y} \mathrm{d}y + \int_{-\infty}^{z-\eta^2} \phi_{\eta}(y) \mathrm{d}y \nonumber \\
	&\leq \int_{0}^{+\infty} \frac{1}{\sqrt{2\pi} \eta} e^{-\frac{|y+z-\eta^2|^{2}}{2\eta^2}} \mathrm{d}y + \int_{-\infty}^{z-\eta^2} \phi_{\eta}(y) \mathrm{d}y \nonumber \\
	&= \int_{z-\eta^2}^{+\infty} \phi_{\eta}(y) \mathrm{d}y + \int_{-\infty}^{z-\eta^2} \phi_{\eta}(y) \mathrm{d}y  = 1.
\end{align}

\noindent \underline{Case $z < 0$.} In this case,
\begin{align*}
	\frac{\exp\left\{ -\psi_{\eta}(z) \right\}}{\frac{1}{2}\exp\left\{-\psi(z)\right\}} = e^{\eta^{2}/2} \cdot \int_{z+\eta^2}^{+\infty} \phi_{\eta}(y) \mathrm{d}y + e^{\eta^{2}/2} e^{-2z} \cdot \int_{-z+\eta^2}^{+\infty} \phi_{\eta}(y) \mathrm{d}y \leq e^{\eta^2/2},
\end{align*}
where in the last step we use
\begin{align*}
	e^{-2z} \int_{-z+\eta^2}^{+\infty} \phi_{\eta}(y) \mathrm{d}y + \int_{z+\eta^2}^{+\infty} \phi_{\eta}(y) \mathrm{d}y \overset{\1}{=} 
	e^{2z'} \int_{z'+\eta^2}^{+\infty} \phi_{\eta}(y) \mathrm{d}y + \int_{-z'+\eta^2}^{+\infty} \phi_{\eta}(y) \mathrm{d}y \overset{\2}{\leq} 1.
\end{align*}
Above, in step $\1$, we let $z' = -z$ so that $z' > 0$, and step $\2$ follows from inequality~\eqref{ineq1:pf-claim-ex-laplace}. 

\medskip

\noindent Combining the two cases establishes claim~\eqref{eq:ex-Lap1}.

\paragraph*{Proof of claim~\eqref{eq:ex-Lap2}.} Substituting $z = 0$ into Eq.~\eqref{eq:mengqi-key}, we have
\begin{align*}
	\exp\left\{ -\psi_{\eta}(0) \right\} = e^{ \frac{\eta^2}{2}} \int_{\eta^2}^{+\infty} \phi_{\eta}(y) \mathrm{d}y = e^{ \frac{\eta^2}{2}} \int_{\eta}^{+\infty} \phi(y) \mathrm{d}y \geq \int_{1}^{+\infty} \phi(y) \mathrm{d}y \geq 0.1,
\end{align*}
where $\phi$ is the pdf of standard normal random variable, and we have used the fact that $\eta \in (0,1)$. We thus obtain $\exp\left\{ \psi_{\eta}(0) \right\} \leq 10$, as claimed.

\paragraph*{Proof of claim~\eqref{eq:ex-Lap3}.}
Next we turn to bound $|\psi'_{\eta}(z)|$. By definition, we have
\begin{align}\label{eq:psi-eta-deriative}
	\psi'_{\eta}(z) = \frac{\int_{y \in \real} \phi_{\eta}(y-z) \frac{z-y}{\eta^2} \frac{1}{2}e^{-|y|} \mathrm{d}y }{\int_{y\in \real} \phi_{\eta}(y-z) \frac{1}{2}e^{-|y|} \mathrm{d} y } = \frac{ T_{1} + T_{2} }{\int_{y\in \real} \phi_{\eta}(y-z) \frac{1}{2}e^{-|y|} \mathrm{d} y }, 
\end{align}
where 
\[
T_{1} = \int_{y \leq 0} \phi_{\eta}(y-z) \frac{z-y}{\eta^2} \frac{1}{2}e^{y} \mathrm{d}y \quad \text{and} \quad T_{2} = \int_{y \geq 0} \phi_{\eta}(y-z) \frac{z-y}{\eta^2} \frac{1}{2}e^{-y} \mathrm{d}y.
\]
Let us calculate $T_{1}$ and $T_{2}$ individually. By integrating by parts, we obtain
\begin{align*}
	T_{1} &= \int_{-\infty}^{0} \phi_{\eta}(y-z) \frac{-(y-z)}{\eta^2} \frac{1}{2}e^{y} \mathrm{d} y \\
	& = \frac{1}{2}e^{y} \phi_{\eta}(y-z) \Big \vert_{-\infty}^{0} - \int_{-\infty}^{0} \phi_{\eta}(y-z) \frac{1}{2}e^{y} \mathrm{d} y \\
	& = \frac{1}{2} \phi_{\eta}(z) - \int_{-\infty}^{0} \phi_{\eta}(y-z) \frac{1}{2}e^{y} \mathrm{d} y.
\end{align*}
Proceeding similarly as in the display above yields
\begin{align*}
	T_{2} = \int_{y=0}^{+\infty} \frac{1}{2}e^{-y} \mathrm{d} \phi_{\eta}(y-z) = - \frac{1}{2}\phi_{\eta}(z) + \int_{y=0}^{+\infty} \phi_{\eta}(y-z) \frac{1}{2}e^{-y} \mathrm{d} y.
\end{align*}
Putting together the pieces yields
\begin{align*}
	\big| T_{1} + T_{2} \big| = \left| - \int_{-\infty}^{0} \phi_{\eta}(y-z) \frac{1}{2}e^{y} \mathrm{d} y + \int_{y=0}^{+\infty} \phi_{\eta}(y-z) \frac{1}{2}e^{-y} \mathrm{d} y \right| \leq \int_{y \in \real} \phi_{\eta}(y-z) \frac{1}{2}e^{-|y|} \mathrm{d} y,
\end{align*}
and substituting into Eq.~\eqref{eq:psi-eta-deriative} yields the claimed result.
\qed


\subsection{Proof of Proposition~\ref{prop:uniform-S-star}} \label{sec:pf-Unif}

Recall that
\[
u(x; \theta) =
\begin{cases}
1 \quad &\text{ if } \theta - 1/2 \leq x \leq \theta + 1/2 \\
0 \quad &\text{ otherwise,}  
\end{cases} 
\]
and that we have set 
\begin{align*}
S^*(y|x) = 
\begin{cases}
g^-(y) \quad &\text{ for all } x \leq \theta_{0}-1/2, \\
g^+(y) \quad &\text{ for all } x \geq \theta_{0}+1/2
\end{cases}
\end{align*}
where $\theta_0$ is a point at which $v(y; \cdot)$ is possibly non-differentiable. 
It suffices to verify that the following system of equations holds for all $\theta \in [-1/2, 1/2]$ and $ y \in \mathbb{Y}$:
\begin{align} \label{eq:equiv-uniform}
\int_{\theta - 1/2 \;\wedge\; \theta_{0}-1/2}^{\theta_{0}-1/2} g^- (y) \mathrm{d}x + \int_{\theta - 1/2 \; \vee \; \theta_{0} - 1/2}^{\theta + 1/2 \; \wedge \; \theta_{0}+1/2} \mathcal{S}^*(y| x) \mathrm{d}y + \int_{\theta_{0}+1/2}^{\theta + 1/2 \; \vee \; \theta_{0} + 1/2} g^+(y) \mathrm{d}x = v(y; \theta).
\end{align}

Simplifying Eq.~\eqref{eq:equiv-uniform}, we obtain the equivalent systems of equations
\begin{subequations}\label{eq2:equiv-uniform}
\begin{align}\label{subeq1:equiv-uniform}
\int_{\theta - 1/2}^{\theta_{0}+ 1/2} \mathcal{S}^*(y|x) \mathrm{d}x + \int_{\theta_{0}+1/2}^{\theta+1/2} g^+(y) \mathrm{d}x &= v(y; \theta) \quad \text{ for all } \theta \in [\theta_{0}, 1/2], \; y \in \mathbb{Y}, \text{ and } \\
\label{subeq2:equiv-uniform}
\int_{\theta - 1/2}^{\theta_{0}-1/2} g^- (y) \mathrm{d}x + \int_{\theta_{0} - 1/2}^{\theta+1/2} \mathcal{S}^*(y| x) \mathrm{d}x &= v(y; \theta) \quad \text{ for all } \theta \in [-1/2, \theta_{0}), \; y \in \mathbb{Y}.
\end{align}
\end{subequations}
We now verify that the kernel $\mathcal{S}^{*}(y|x)$ defined in Eq.~\eqref{eq:g-plus-g-minus-y} and Eq.~\eqref{S-star-uniform} satisfies Eq.~\eqref{eq2:equiv-uniform}. 

\paragraph*{Case $\theta_{0} \leq \theta \leq 1/2$:} Expanding the LHS of Eq.~\eqref{subeq1:equiv-uniform} by substituting for $\mathcal{S}^*$~\eqref{S-star-uniform} yields
\begin{align*}
	&\int_{\theta - 1/2}^{\theta_{0}+ 1/2} \mathcal{S}^*(y|x) \mathrm{d}x + \int_{\theta_{0}+1/2}^{\theta+1/2} g^+(y) \mathrm{d}x \\&= \int_{\theta - 1/2} ^{0} g^+(y) - \nabla_{\theta'} v(y; \theta')\mid_{\theta' = x + \frac{1}{2}} \mathrm{d}x + \int_{0}^{\theta_{0}+1/2} \nabla_{\theta'} v(y; \theta')\mid_{\theta' = x - \frac{1}{2}} + g^-(y) \mathrm{d}x + (\theta - \theta_{0}) \cdot g^{+}(y)
	\\ &= g^{+}(y) \cdot (1/2 - \theta_{0}) + g^{-}(y) \cdot (\theta_{0}+1/2) - v(y;1/2) + v(y;\theta) + v(y; \theta_{0}) - v(y;-1/2)  = v(y;\theta).
\end{align*}
where the last equality holds due to condition~\eqref{eq:g-plus-g-minus-y} on $g^+$ and $g^-$.

\paragraph*{Case $-1/2 \leq \theta < \theta_{0}$:} Expanding the LHS of Eq.~\eqref{subeq2:equiv-uniform} by substituting for $\mathcal{S}^*$~\eqref{S-star-uniform} yields
\begin{align*}
	&\int_{\theta - 1/2}^{\theta_{0}-1/2} g^- (y) \mathrm{d}x + \int_{\theta_{0} - 1/2}^{\theta+1/2} \mathcal{S}^*(y| x) \mathrm{d}x
	\\ &= (\theta_{0} - \theta) \cdot g^{-}(y) + \int_{\theta_{0}-1/2}^{0} g^+(y) - \nabla_{\theta'} v(y; \theta')\mid_{\theta' = x + \frac{1}{2}} \mathrm{d}x + \int_{0}^{\theta+1/2} \nabla_{\theta'} v(y; \theta')\mid_{\theta' = x - \frac{1}{2}} + g^-(y)  \mathrm{d}x
	\\ &= g^{-}(y) \cdot (\theta_{0}+1/2) + g^{+}(y) \cdot (1/2 - \theta_{0}) - v(y;1/2) + v(y;\theta_{0}) + v(y;\theta) - v(y;-1/2) = v(y;\theta),
\end{align*}
where again the equality step holds due to condition~\eqref{eq:g-plus-g-minus-y} on $g^+$ and $g^-$.

Putting together the two cases concludes the proof.
\qed

\begin{remark}
While we proved Proposition~\ref{prop:uniform-S-star} by verifying that our choice of $\mathcal{S}^*, g^+, g^-$ satisfies the system\eqref{eq:equiv-uniform}, a constructive way to arrive at the solution is to differentiate and equate both sides of Eqs.~\eqref{subeq1:equiv-uniform}, since they must hold for all $\theta \in \Theta$. The expressions for $\mathcal{S}^*$ as well as $g^+, g^-$ can then be obtained by solving the resulting equations.
\end{remark}


\subsection{Proof of Theorem~\ref{thm:Unif-Gauss}} \label{sec:pf-UnifGauss}

Recall that the input sample space is given by $\mathbb{X} = [-1, 1]$, the output sample space by $\mathbb{Y} = \real$ and the parameter space by $\Theta = [-1/2, 1/2]$.
Also recall the shorthand
$\alpha_{0} = \max_{-1/2 \leq t \leq 1/2} \big|f(t)\big|$ and $\alpha_1 = \sup_{t \in [-1/2,\theta_{0}) \cup (\theta_{0},1/2]} \big|f'(t) \big|$.
We additionally define some quantities for convenience. Let $\Delta = f(0.5) + f(-0.5) - 2f(\theta_{0})$, and 
\begin{align*}
y_0 := \frac{-2\log(2) \sigma^{2} }{\Delta} 
		+ \frac{ f(0.5)^{2} + f(-0.5)^{2} - 2f(\theta_{0})^{2}  }{2\Delta} \quad &\text{ if } \Delta \neq 0.
\end{align*}
Now on the domain $(\theta_{0} - 1/2, \theta_{0} + 1/2)$, define the functions $\ell$ and $L$ as
\begin{subequations} \label{eq:lx-ux}
	\begin{align}
		\ell(x) &= f(x+1/2) - \frac{\sigma^2}{5\alpha_{1}},\quad  L(x) = f(x+1/2) + \frac{\sigma^2}{5\alpha_{1}}, \quad \text{if} \quad \theta_{0} - 1/2 < x \leq 0,\\
		\ell(x) &= f(x-1/2) - \frac{\sigma^2}{5\alpha_{1}},\quad  L(x) = f(x-1/2) + \frac{\sigma^2}{5\alpha_{1}}, \quad \text{if} \quad 0< x < \theta_{0} +1/2.
	\end{align}
	\end{subequations}

To prove the theorem, we need to bound the appropriate Radon--Nikodym derivative in this case, and establish Ineq.~\eqref{eq:claim-unif-Gauss}. We use two technical lemmas to do so.
The first lemma is stated below, and its proof can be found in Appendix~\ref{sec:proof-lemma-aux-unif-gauss}.

\begin{lemma}\label{lemma:auxiliary-proof-Unif-Gauss}
Let $\mathcal{S}^{*}$ be defined in Eq.~\eqref{S-star-uniform}. 
The following holds for  $\sigma \geq 5 \max\{\alpha_{0}, \alpha_{1}\}$.
\begin{itemize}
	\item[(a)] Suppose $x \geq \theta_{0} + 1/2$ or $x \leq \theta_{0}-1/2$. 
	\begin{itemize}
		\item[(i)] If $\Delta = 0$, then $\mathcal{S}(y|x) \geq 0$ for all $y \in \real$.
		\item[(ii)] If $\Delta > 0$, then 
	\begin{subequations}\label{ineq:part-1-lemma-Unif-Gauss}
	\begin{align}\label{ineq:postivity-S-star-part-1}
		\mathcal{S}^{*}(y|x) &\geq 0 \quad \text{ for all }  \; y \geq y_0,
		\\ \label{ineq1:S-star-bound-part-1}
		1 - \exp\left\{-\frac{\log^{2}(2) \sigma^{2}}{2 \Delta^{2} } \right\} &\leq \int_{y_{0}}^{+\infty} \mathcal{S}^{*}(y|x) \mathrm{d}y \leq 1 + \exp\left\{-\frac{\log^{2}(2) \sigma^{2}}{2 \Delta^{2} } \right\}, 
		\\ \label{ineq2:S-star-bound-part-1}
		\text{and} \quad \int_{-\infty}^{y_{0}} \big| \mathcal{S}^{*}(y|x) \big| \mathrm{d}y &\leq \frac{3}{2} \exp\left\{-\frac{\log^{2}(2) \sigma^{2}}{2 \Delta^{2} } \right\}.
	\end{align}
	\end{subequations}
	\item[(iii)] If $\Delta < 0$, then
	\begin{subequations}\label{ineq:part3-lemma-Unif-Gauss}
	\begin{align}\label{ineq:postivity-S-star-part-3}
		\mathcal{S}^{*}(y|x) &\geq 0 \quad \text{ for all }  \; y \leq y_0,
		\\ \label{ineq1:S-star-bound-part-3}
		1 - \exp\left\{-\frac{\log^{2}(2) \sigma^{2}}{2 \Delta^{2} } \right\} &\leq \int_{-\infty}^{y_{0}} \mathcal{S}^{*}(y|x) \mathrm{d}y \leq 1 + \exp\left\{-\frac{\log^{2}(2) \sigma^{2}}{2 \Delta^{2} } \right\}, 
		\\ \label{ineq2:S-star-bound-part-3}
		\text{and} \quad \int_{y_{0}}^{+\infty} \big| \mathcal{S}^{*}(y|x) \big| \mathrm{d}y &\leq \frac{3}{2} \exp\left\{-\frac{\log^{2}(2) \sigma^{2}}{2 \Delta^{2} } \right\}.
	\end{align}
	\end{subequations}
\end{itemize}
	\item[(b)] Suppose $\theta_{0} - 1/2 < x < \theta_{0} + 1/2$ and recall the functions $\ell$ and $L$ from Eq.~\eqref{eq:lx-ux}.
 We have $\mathcal{S}^{*}(y|x) \geq 0$ for all $\ell(x) \leq y \leq L(x)$, and
	\begin{subequations}\label{ineq:part-2-lemma-Unif-Gauss}
	\begin{align}\label{ineq1:S-star-bound-part-2}
		1 - 2 \exp\left\{-\frac{\sigma^{2}}{200 \alpha_{1}^{2}} \right\}  \leq &\int_{\ell(x)}^{L(x)} \mathcal{S}^{*}(y|x) \mathrm{d}y \leq 1 + 2 \exp\left\{-\frac{\sigma^{2}}{200 \alpha_{1}^{2}} \right\}, 
		\\ \label{ineq2:S-star-bound-part-2}
		\int_{-\infty}^{\ell(x)} \big| \mathcal{S}^{*}(y|x) \big| \mathrm{d}y \; \vee \; &\int_{L(x)}^{+\infty} \big| \mathcal{S}^{*}(y|x) \big| \mathrm{d}y \leq  2 \exp\left\{ -\frac{\sigma^{2}}{200 \alpha_{1}^{2}} \right\}.
	\end{align}
	\end{subequations}
\end{itemize}
\end{lemma}
For each $x$, Lemma~\ref{lemma:auxiliary-proof-Unif-Gauss} characterizes the interval on which the kernel $\mathcal{S}^{*}(\cdot|x)$ is positive, and it gives bounds on the integral of $\mathcal{S}^{*}(\cdot|x)$ over that interval. We will use Lemma~\ref{lemma:auxiliary-proof-Unif-Gauss} to evaluate $p(x)$ and $q(x)$ from Eq.~\eqref{eq:pq}. We state another auxiliary lemma to bound $\frac{\mathcal{S}^{*}(y|x) \; \vee \;0}{\mathcal{P}(y|x)}$.

\begin{lemma}\label{lemma:M-30}
	Under the setting of Theorem~\ref{thm:Unif-Gauss}, we have 
	\[
		\frac{\mathcal{S}^{*}(y|x) \; \vee \; 0}{\mathcal{P}(y|x)} \leq 30, \quad \text{ for all } \; x \in \mathbb{X} \quad \text{and} \quad y\in \real.
	\]
\end{lemma}
We provide the proof of Lemma~\ref{lemma:M-30} in Appendix~\ref{sec:proof-lemma-M-30}. Taking these two lemmas as given, we now proceed to a proof of the theorem.

\paragraph*{Bounding the Radon--Nikodym derivative:} Applying Lemma~\ref{lemma:M-30}, we obtain that it suffices to set $M = 30$ in Algorithm~\textsc{RK}.

\paragraph*{Establishing inequality~\eqref{eq:claim-unif-Gauss}:} We split the proof into two cases depending on the value of $x$. 

\noindent \underline{Case 1: $x\leq \theta_{0}-1/2$ or $x\geq \theta_{0}+1/2$.} In this case, we obtain from Eq.~\eqref{S-star-uniform} that
\begin{align}\label{eq:closed-form-S-star-case1}
\hspace{-1cm}
	\mathcal{S}^{*}(y|x) &= \frac{1}{\sqrt{2\pi} \sigma} \bigg( \exp\left\{-\frac{|y-f(0.5)|^{2}}{2\sigma^2}\right\} + \exp\left\{-\frac{|y-f(-0.5)|^{2}}{2\sigma^2} \right\} - \exp\left\{-\frac{|y-f(\theta_{0})|^{2}}{2\sigma^2} \right\}  \bigg).
\end{align}
Recall that $\Delta = f(0.5)+f(-0.5) - 2f(\theta_{0})$. If $\Delta = 0$, applying part (i) of Lemma~\ref{lemma:auxiliary-proof-Unif-Gauss} and using the expression of $\mathcal{S}^{*}(y|x)$ in Eq.~\eqref{eq:closed-form-S-star-case1} yields
\begin{align}\label{ineq:px-qx-bound-case0}
	p(x) = \int_{y\in \real} \mathcal{S}^{*}(y|x) \mathrm{d}y = 1 \quad \text{and} \quad q(x) = 0.
\end{align}
Now suppose $\Delta>0$, and let $y_{0}$ be defined as in Eq.~\eqref{ineq:postivity-S-star-part-1}. Applying inequalities~\eqref{ineq:postivity-S-star-part-1} and~\eqref{ineq1:S-star-bound-part-1} yields
\begin{align*}
	p(x) = \int_{y \in \real} \mathcal{S}^{*}(y|x) \vee 0 \;\mathrm{d}y \geq \int_{y_{0}}^{+\infty} \mathcal{S}^{*}(y|x) \mathrm{d}y \geq 1 - \exp\left\{-\frac{\log^{2}(2) \sigma^{2}}{2 \Delta^{2} } \right\}.
\end{align*}
Proceeding similarly and applying inequality~\eqref{ineq:part-1-lemma-Unif-Gauss}, we have
\begin{align*}
	p(x) &\leq \int_{y_{0}}^{+\infty} \mathcal{S}^{*}(y|x) \mathrm{d}y + \int_{-\infty}^{y_{0}} \big| \mathcal{S}^{*}(y|x) \big| \mathrm{d}y \leq 1 + \frac{5}{2} \exp\left\{-\frac{\log^{2}(2) \sigma^{2}}{2 \Delta^{2} } \right\} , \\
	q(x) &= \int_{y \in \real} -\big( \mathcal{S}^{*}(y|x) \wedge 0 \big) \;\mathrm{d}y \leq \int_{-\infty}^{y_{0}} \big| \mathcal{S}^{*}(y|x)\big| \mathrm{d}y \leq \frac{3}{2} \exp\left\{-\frac{\log^{2}(2) \sigma^{2}}{2 \Delta^{2} } \right\}.
\end{align*}
Putting the pieces together yields
\begin{align}\label{ineq:px-qx-bound-case1}
	\big| p(x) - 1 \big| + q(x) \leq 4 \cdot \exp\left\{-\frac{\log^{2}(2) \sigma^{2}}{2 \Delta^{2} } \right\}.
\end{align}
Finally, if $\Delta<0$ then we proceed exactly as in the case $\Delta>0$. In particular, using inequality~\eqref{ineq:part3-lemma-Unif-Gauss} yields the same bound~\eqref{ineq:px-qx-bound-case1}; we omit the calculations for brevity.

\noindent \underline{Case 2: $\theta_{0} - 1/2 < x < \theta_{0} + 1/2$.} Let $\ell(x)$ and $L(x)$ be defined as in~\eqref{eq:lx-ux}. Applying part (ii) of Lemma~\ref{lemma:auxiliary-proof-Unif-Gauss}, we obtain $\mathcal{S}^{*}(y|x) \geq 0$ for all $\ell(x) \leq y \leq L(x)$. Consequently, we have 
\begin{align*}
	p(x) = \int_{y \in \real} \mathcal{S}^{*}(y|x) \vee 0 \; \mathrm{d}y \geq \int_{\ell(x)}^{L(x)} \mathcal{S}^{*}(y|x) \mathrm{d}y \geq 1 - 2 \exp\left\{-\frac{\sigma^{2}}{200 \alpha_{1}^{2}} \right\} 
\end{align*}
where in the last step we use inequality~\eqref{ineq1:S-star-bound-part-2}. Proceeding similarly as in the display above, we obtain
\begin{align*}
	p(x) &\leq \int_{\ell(x)}^{L(x)} \mathcal{S}^{*}(y|x) \mathrm{d}y + \int_{-\infty}^{\ell(x)} \big| \mathcal{S}^{*}(y|x) \big| \mathrm{d}y  + \int_{L(x)}^{+\infty} \big| \mathcal{S}^{*}(y|x) \big| \mathrm{d}y \overset{\1}{\leq} 1 + 6\exp\left\{-\frac{\sigma^{2}}{200 \alpha_{1}^{2}} \right\},\\
	q(x) &= \int_{y\in \real} - \big( \mathcal{S}^{*}(y|x) \wedge 0 \big) \mathrm{d}y \leq \int_{-\infty}^{\ell(x)} \big| \mathcal{S}^{*}(y|x) \big| \mathrm{d}y  + \int_{L(x)}^{+\infty} \big| \mathcal{S}^{*}(y|x) \big| \mathrm{d}y \overset{\2}{\leq} 4\exp\left\{-\frac{\sigma^{2}}{200 \alpha_{1}^{2}} \right\},
\end{align*}
where in steps $\1$ and $\2$ we use inequality~\eqref{ineq:part-2-lemma-Unif-Gauss}. Putting the pieces together yields
\begin{align}\label{ineq:px-qx-bound-case2}
	\big| p(x) - 1 \big| + q(x) \leq 10 \exp\left\{-\frac{\sigma^{2}}{200 \alpha_{1}^{2}} \right\}.
\end{align}

\noindent \underline{Putting the pieces together.} Combining equation~\eqref{ineq:px-qx-bound-case0}, inequalities~\eqref{ineq:px-qx-bound-case1} and~\eqref{ineq:px-qx-bound-case2}, we have that for all $x \in \mathbb{X}$,
\begin{align*}
	\big|p(x) - 1 \big| + q(x) \leq 10 \bigg( \exp\left\{-\frac{\log^{2}(2) \sigma^{2}}{2\Delta^{2}}\right\} \vee  \exp\left\{-\frac{\sigma^{2}}{200 \alpha_{1}^{2}} \right\}\bigg) \leq 10 \exp\left\{-\frac{\sigma^{2}}{200 \alpha_{1}^{2}} \right\},
\end{align*}
where in the last step we have used 
\begin{align*}
|\Delta| = |f(0.5) + f(-0.5) - 2f(\theta_{0})| &\leq |f(0.5) - f(\theta_{0})| + |f(-0.5) - f(\theta_0)| 
\\ & \leq \alpha_{1} \cdot \big(|0.5 - \theta_{0}| + |-0.5 - \theta_{0}| \big) \leq 2 \alpha_{1}.
\end{align*}
Consequently, for $\sigma \geq 40\alpha_{1}$ and for all $x\in \mathbb{X}$, we have
\[
	p(x) \geq 1 - 10e^{-8} \geq 0.5.
\]
Note that by Proposition~\ref{prop:uniform-S-star}, we have $\SignedDef(\mathcal{U}, \mathcal{V};\mathcal{S}^*) = 0$. Consequently, applying Lemma~\ref{lem:rej-sampling}, we obtain that for each $-1/2 \leq \theta \leq 1/2$ and $M = 30$, $y \in \real$,
\begin{align*}
	\| \mathcal{L} \left[ \textsc{rk}(X_{\theta}, N, M, y_0) \right] - \NORMAL(\theta, \sigma^2) \|_{\mathsf{TV}} &\leq 2e^{-\frac{N}{M} \cdot 0.5 } + \int_{\mathbb{X}} 10 \exp\bigg\{ - \frac{\sigma^2}{200\alpha_{1}^{2}}\bigg\} \cdot u(x ; \theta) \mathrm{d}x 
	\\ &= \frac{\epsilon}{2} + 10 \exp\bigg\{-\frac{\sigma^2}{200\alpha_{1}^{2}}\bigg\},
\end{align*}
where in the last step we let $N = 2M \log(4/\epsilon)$.
\qed


\subsection{Proof of Corollary~\ref{coro-mix-phase-retrieval}}\label{sec:pf-mix-phase}
Condition on the tuple $\{ x_i\}_{i = 1}^n$. For each $1 \leq i \leq n$, we have $y_{i} \sim \mathsf{Unif}( [\theta_{i} - 1/2, \theta_{i} + 1/2])$ and $\widetilde{y}_{i} \sim \NORMAL(f(\theta_{i}),\sigma^{2})$, where $\theta_{i} = R_{i} \cdot \langle x_{i},\beta_{\star} \rangle$ and $f(\cdot) = |\cdot|$. Thus, our source distribution is uniform of unit width around its mean $\theta_{i}$ and our target distribution is normal with mean $|\theta_{i}|$ and variance $\sigma^{2}$. We therefore make use of Theorem~\ref{thm:Unif-Gauss} to design the algorithm $h$. 

\paragraph*{Algorithm.} Using Eq.~\eqref{S-star-uniform} to calculate the signed kernel for this target model, we have 
\begin{align}\label{S-star-uniform-gaussian-abs}
\begin{split}
	\mathcal{S}^*(z| w) = 
	\begin{cases}
	2\phi_{\sigma}(z-1/2) - \phi_{\sigma}(z)  &\text{ for all } w \leq -\frac{1}{2} \text{ or } w \geq \frac{1}{2}, \\
	2\phi_{\sigma}(z-1/2) - \phi_{\sigma}(z) - \phi_{\sigma}(z-w-1/2) \cdot \frac{z-w-1/2}{\sigma^2}  &\text{ for all } -1/2 < w \leq 0, \\
	2\phi_{\sigma}(z-1/2) - \phi_{\sigma}(z) - \phi_{\sigma}(z+w-1/2) \cdot \frac{z+w-1/2}{\sigma^2}  &\text{ for all } 0 < w < 1/2,
	\end{cases}
\end{split}
\end{align} 
where $\phi_{\sigma}(\cdot)$ is the pdf of Gaussian random variable with zero mean and variance $\sigma^2$, and $z, w \in \real$ are placeholder variables. Now for $\epsilon \in (0,1)$, let $M = 30$ and $N = 60 \log(4/\epsilon)$, design the algorithm $h((x_{i},y_{i})_{i=1}^{n})$ via the following entry-wise reduction. 
\begin{enumerate}
	\item For each $i \in [n]$, let the sensing vector $x_{i}$ be same, i.e., $x_{i} \gets x_{i}$ and
	\begin{align}\label{eq:step1-mix-phase}
		 \overline{y}_{i} \gets \textsc{RK}(y_{i},N,M, y_i), 
	\end{align}
	with kernel $\mathcal{S}^*$ as in Eq.~\eqref{S-star-uniform-gaussian-abs} and base measure $\mathcal{P}(\cdot|w) = \NORMAL(0,2\sigma^{2})$ for all $w \in \real$.
	\item Output $h((x_{i},y_{i})_{i=1}^{n}) = (x_{i},\overline{y}_{i})_{i=1}^{n}$.
\end{enumerate}  

\paragraph*{Analysis.} In the notation of Theorem~\ref{thm:Unif-Gauss}, we have $f(\cdot) = | \cdot |$ and thus $\theta_0 = 0$. Applying Theorem~\ref{thm:Unif-Gauss} yields that for $\sigma \geq C\sqrt{\log(20/\epsilon)}$,
\[
	\mathsf{d_{TV}}\big((x_{i},\overline{y}_{i}), (x_{i},\widetilde{y}_{i}) \big) \leq \epsilon, \quad \text{ for all } \; 1 \leq i \leq n.
\]
Since $(x_{i},\widetilde{y}_{i})_{i=1}^{n}$ are i.i.d. and $(x_{i},\overline{y}_{i})_{i=1}^{n}$ are i.i.d., we obtain
\begin{align*}
	\mathsf{d_{TV}}\big(h\big((x_{i},y_{i})_{i=1}^{n} \big), (x_{i},\widetilde{y}_{i})_{i=1}^{n} \big) &= \mathsf{d_{TV}}\big((x_{i},\overline{y}_{i})_{i=1}^{n}, (x_{i},\widetilde{y}_{i})_{i=1}^{n} \big) 
	\\ &\leq \sum_{i=1}^{n} \mathsf{d_{TV}}\big((x_{i},\overline{y}_{i}), (x_{i},\widetilde{y}_{i}) \big) \leq n \epsilon.
\end{align*}
Choosing $\epsilon = \delta/n$ yields $\mathsf{d_{TV}}\big(h\big((x_{i},y_{i})_{i=1}^{n} \big), (x_{i},\widetilde{y}_{i})_{i=1}^{n} \big) \leq \delta$.
Note that step~\eqref{eq:step1-mix-phase} runs in time $N(T_{samp}+T_{eval}) = O\big( \log(4n/\delta) (T_{samp}+T_{eval})  \big)$ and we need to run this step $n$ times. Thus, the total running time is $O\big( n \log(4n/\delta) (T_{samp}+T_{eval})  \big)$. This concludes the proof.
\qed


\subsection{Proof of Corollary~\ref{coro-low-rank-matrix}}\label{sec:pf-low-rank}

Note that for each $i \in \Omega$, the random variable $y_{i}$ is a shifted exponential with mean $\theta^{\star}_{i}$ and unit variance, and $\widetilde{y}_{i}$ is log-concave with mean $\theta^{\star}_{i}$ and scale $\sigma$. 
We therefore make use of Theorem~\ref{thm:exp-log-concave} to design the algorithm $h$. 

\paragraph*{Algorithm.} We output $\{\overline{y}_{i}\}_{i=1}^{\Dim} = h(\{y_{i}\}_{i=1}^{\Dim})$ via
\begin{enumerate}
	\item For each $i\notin \Omega$, let $\overline{y}_{i} = \star$. For each $i \in \Omega$, let
	\begin{align}\label{eq:step1-low-rank}
		\overline{y}_{i} \gets \textsc{RK}(y_{i},N,M, y_i),
	\end{align}
	with the kernel $\mathcal{S}^{*}$ defined in Eq.~\eqref{eq:signed-kernel-Exp}, base measure $\mathcal{P}$ defined in Eq.~\eqref{eq:base-measure-log-concave}, $N \geq 1$, and $M$ satisfying Eq.~\eqref{ineq:M-lower-bound-exp}.
	\item Output $h\big(\{y_{i}\}_{i=1}^{\Dim} \big) = \{\overline{y}_{i} \}_{i=1}^{\Dim}$.
\end{enumerate}

\paragraph*{Analysis.} Applying Theorem~\ref{thm:exp-log-concave} yields that for all $\sigma>0$,
\[
	\mathsf{d_{TV}}\big( \overline{y}_{i}, \widetilde{y}_{i} \big) \leq  2e^{-\frac{N}{M} \cdot \big(1-\tau(\sigma)\big)} + \tau(\sigma), \quad \text{for all} \; i \in \Omega.
\]
Moreover, for each $i \notin \Omega$, we have $\mathsf{d_{TV}}\big( \overline{y}_{i}, \widetilde{y}_{i} \big) = 0$ since $\overline{y}_{i} = y_{i} = \star$. Since $\{\overline{y}_{i}\}_{i\in \Omega}$ are i.i.d. and $\{\widetilde{y}_{i}\}_{i\in \Omega}$ are i.i.d., we obtain
\begin{align*}
	\mathsf{d_{TV}}\big( h\big( \{y_{i}\}_{i=1}^{\Dim} \big), \{ \widetilde{y}_{i} \}_{i=1}^{\Dim} \big) &= \mathsf{d_{TV}}\big( \{ \overline{y}_{i} \}_{i=1}^{\Dim}, \{ \widetilde{y}_{i} \}_{i=1}^{\Dim} \big) 
	\\ &\leq \sum_{i \in \Omega}  \mathsf{d_{TV}}\big( \overline{y}_{i}, \widetilde{y}_{i} \big) + \sum_{i \notin \Omega} \mathsf{d_{TV}}\big( \overline{y}_{i}, \widetilde{y}_{i} \big)
	\\ &\leq |\Omega| \cdot \Big( 2e^{-\frac{N}{M} \cdot \big(1-\tau(\sigma)\big)} + \tau(\sigma) \Big).
\end{align*}
Since $1-\tau(\sigma)>0$ by assumption, letting $N = M\log(2/\epsilon)/ (1-\tau(\sigma))$ for $\epsilon \in (0,1)$ yields
\[
	\mathsf{d_{TV}}\big( h\big( \{y_{i}\}_{i=1}^{\Dim} \big), \{ \widetilde{y}_{i} \}_{i=1}^{\Dim} \big) \leq |\Omega|\cdot\big(\epsilon + \tau(\sigma)\big).
\]
Note that each step~\eqref{eq:step1-low-rank} runs in time $N(T_{eval}+T_{samp})$ and we need to this step $|\Omega|$ times. Thus, the total running time of $h$ is $O\big(|\Omega| N(T_{eval} + T_{samp})\big)$.
\qed


\subsection{Proof of Proposition~\ref{prop:Exp-reduction-sharp}}\label{sec:pf-prop-cor2}
Recall that $\overline{y} = h(y)$. Applying Lemma~\ref{lemma:risk-bound} with the $\ell_{2}$ loss and the setting $p=q=1/2$, we obtain 
\begin{align*}
	\Big| \EE \big[ \| \widehat{\theta}_{\mathsf{LS}}(\overline{y}) - \theta^{\star} \|_2 \big] - \EE \big[ \| \widehat{\theta}_{\mathsf{LS}}(\widetilde{y}) - \theta^{\star} \|_2 \big] \Big| &\leq \overline{L} \cdot \mathsf{d_{TV}}\big( \overline{y}, \widetilde{y} \big)  \\
	 &+  \EE \big[ \| \widehat{\theta}_{\mathsf{LS}}(\overline{y}) - \theta^{\star} \|_2^2 \big]^{1/2} \cdot \Pr\big\{ \| \widehat{\theta}_{\mathsf{LS}}(\overline{y}) - \theta^{\star} \|_2 \geq \overline{L} \big\}^{1/2} \\
	 &+ \EE \big[ \| \widehat{\theta}_{\mathsf{LS}}(\widetilde{y}) - \theta^{\star} \|_2^2 \big]^{1/2} \cdot \Pr\big\{ \| \widehat{\theta}_{\mathsf{LS}}(\widetilde{y}) - \theta^{\star} \|_2 \geq \overline{L} \big\}^{1/2}
\end{align*}
Since $\Theta$ is a convex set, we obtain $\| \widehat{\theta}_{\mathsf{LS}}(w) - \theta^{\star} \|_{2} \leq \| w - \theta^{\star}\|_{2}$ for all $w\in \real^{\Dim}$. Consequently, using the bound in the display yields
\begin{align*}
	\Big| \EE \big[ \| \widehat{\theta}_{\mathsf{LS}}(\overline{y}) - \theta^{\star} \|_2 \big] - \EE \big[ \| \widehat{\theta}_{\mathsf{LS}}(\widetilde{y}) - \theta^{\star} \|_2 \big] \Big| 
	&\leq \overline{L} \cdot \mathsf{d_{TV}}\big( \overline{y}, \widetilde{y} \big) +  \EE \big[ \| \overline{y} - \theta^{\star} \|_2^2 \big]^{1/2} \cdot \Pr\big\{ \| \overline{y} - \theta^{\star} \|_2 \geq \overline{L} \big\}^{1/2} \\
	 &+ \EE \big[ \| \widetilde{y} - \theta^{\star} \|_2^2 \big]^{1/2} \cdot \Pr\big\{ \| \widetilde{y} - \theta^{\star} \|_2 \geq \overline{L} \big\}^{1/2}
\end{align*}
We then turn to bound each term in the RHS of the display above. Recall that $\overline{y}_{i} \gets \textsc{RK}(y_{i},N,M, y_i)$ in the design of algorithm $h$~\eqref{eq:step1-low-rank}. \\
\underline{Controlling the expectations.} 
We obtain
$
	\EE\big[ \| \overline{y} - \theta^{\star} \|_2^2  \big] = \sum_{i=1}^{\Dim} \EE\left[(\overline{y}_{i} - \theta^{\star}_{i} )^{2} \right].
$
Applying Lemma~\ref{lemma:prob-bound-RK-sample-exp} then yields
\begin{align*}
	\EE\left[(\overline{y}_{i} - \theta^{\star}_{i} )^{2} \right] = \int_{0}^{\infty} \Pr\big\{ (\overline{y}_{i} - \theta^{\star}_{i} )^{2} \geq t \big\} \mathrm{d}t &= \int_{0}^{\infty} \Pr\big\{ |\overline{y}_{i} - \theta^{\star}_{i} | \geq t^{1/2} \big\} \mathrm{d}t \\
	&\leq 3 \int_{0}^{\infty} \exp\big(-t / (8\sigma^2)\big) \mathrm{d}t + 5 \int_{0}^{\infty} \exp(-\sqrt{t}/2) \mathrm{d}t \\
	&= 24\sigma^2 + 40.
\end{align*}
Putting the two pieces together yields
\[
	\EE\big[ \| \overline{y} - \theta^{\star} \|_2^2  \big]^{1/2} \leq \sqrt{\Dim} \cdot \sqrt{24\sigma^2 + 40} 
\]
Note that $\widetilde{y}_{i} - \theta^{\star}_{i} \sim \NORMAL(0,\sigma^2)$. Following identical steps yields
$
	\EE\big[ \| \widetilde{y} - \theta^{\star} \|_2^2  \big]^{1/2} = \sqrt{n\sigma^2}.
$

\noindent \underline{Controlling the tails.} Let $\alpha \geq 8$ and $\overline{L} = \sqrt{\Dim} \alpha \sigma  \log(\Dim)$. We obtain 
\begin{align*}
	\Pr\big\{ \| \overline{y} - \theta^{\star} \|_2 \geq \overline{L} \big\} = 
	\Pr\big\{ \| \overline{y} - \theta^{\star} \|_2^2 \geq \overline{L}^{2} \big\} &\overset{\1}{\leq}  \sum_{i=1}^{\Dim}  \Pr\big\{ |\overline{y}_{i} - \theta^{\star}_{i}| \geq \alpha \sigma \log(\Dim) \big\} \\
	&\overset{\2}{\leq} \Dim\Big( 3e^{-\alpha^{2} \log(\Dim)/8} + 5e^{-\alpha \sigma \log(\Dim)/2} \Big)
	\leq 10  \Dim^{-\alpha+1},
\end{align*}
where in step $\1$ we apply union bound, and in step $\2$ we apply Lemma~\ref{lemma:prob-bound-RK-sample-exp}. Similarly, we obtain 
\[
	\Pr\big\{ \| \widetilde{y} - \theta^{\star} \|_2 \geq \overline{L} \big\} \leq \Dim \Pr\big\{ |\widetilde{y}_{i} - \theta^{\star}_{i}| \geq \alpha \sigma  \log(\Dim) \big\} \leq \Dim^{-\alpha/2 + 1}.
\]

\noindent \underline{Putting the pieces together.} From inequality~\eqref{ineq:TV-exp-gauss-signal-denoise}, we have $\mathsf{d_{TV}}\big( \overline{y}, \widetilde{y} \big) \leq 3\Dim\exp(-\sigma^2/2)$. Combining all the pieces yields for all $\alpha \geq 8$,
\begin{align*}
	\Big| \EE \big[ \| \widehat{\theta}_{\mathsf{LS}}(\overline{y}) - \theta^{\star} \|_2 \big] - \EE \big[ \| \widehat{\theta}_{\mathsf{LS}}(\widetilde{y}) - \theta^{\star} \|_2 \big] \Big| &\lesssim
	\sqrt{\Dim} \alpha \sigma  \log(\Dim) \cdot \Dim\exp(-\sigma^2/2) + \sqrt{\Dim\sigma^{2} } \Dim^{-\alpha/4 + 1/2} \\
	&\leq \alpha \log(\Dim) \Dim^{3/2} \cdot \sigma \exp(-\sigma^2/2) + \sigma \Dim^{-\alpha/4 + 1} \\
	& \lesssim \sigma \log(\Dim) / \Dim,
\end{align*}
where in the last step we let $\sigma^{2} \geq 5 \log(n)$ and $\alpha = 8$ so that $ \Dim^{3/2} \cdot \exp(-\sigma^2/2) \leq \Dim^{-1}$. Here we use $a_1 \lesssim a_2$ to mean that the inequality holds up to an absolute constant. The proposition then follows by Eq.~\eqref{eq:risk-bound-gauss-noise}.
\qed


\subsection{Proof of Corollary~\ref{cor:privacy}}\label{sec:pf-privacy}

From Proposition~\ref{prop:Laplace-DP}, we have $g(\bm{X}) \sim \mathsf{Lap}(f(\bm{X}), b)$. Our goal is to transform $g$ into a random variable that is $\delta$-close in total variation to one drawn from the distribution $\NORMAL(f(\bm{X}), 2b^2 \log(12/\delta))$. Our algorithm will thus use the Laplace to Gaussian reduction sketched in Theorem~\ref{thm:Lap-Gaussian}.

\paragraph*{Algorithm.} Let the signed kernel be given by Eq.~\eqref{S-star-Laplace-general}, where we set $\sigma^2 = 2b^2 \log (12/\delta)$.
With this choice of signed kernel, compute and output
	\begin{align}\label{eq:alg-privacy}
		 h \gets \textsc{RK}(g(\bm{X}),N,M, y_0), 
	\end{align}
	with base measure $\mathcal{P}(\;\cdot\;| g(\bm{X})) = \NORMAL(g(\bm{X}), 2b^2 \log (12/\delta))$, $M = 2$, $N = 2 \log (48/\delta)$, and $y_0 = g(\bm{X})$.

\paragraph*{Analysis.} Note that $h$ is realized by composing the Laplace mechanism---which is $\left( \frac{\phi(f)}{b}, 0 \right)$ differentially private, with the above algorithm, which does not access $\bm{X}$. By Theorem 3.14 of~\citet{dwork2014algorithmic}, the mechanism $h$ is thus $\left( \frac{\phi(f)}{b}, 0 \right)$ differentially private.
%

To prove the accuracy bound, apply the second part of Lemma~\ref{lemma:prob-bound-RK-sample-Laplace} (with $\sigma^2 = 2b^2 \log(12/\delta)$) to obtain
\begin{align*}
\mathbb{E}[ (h - f(\bm{X})^2] \leq 2b^2 + 2b^2 \log(12/\delta) + \frac{b^2}{3\sqrt{2}} \cdot \delta \log^{3/2} (12/\delta),
\end{align*}
so that taking square roots and further bounding $3\sqrt{2} > 4$ completes the proof.
\qed

\subsection*{Acknowledgments} 
ML and AP were supported in part by the NSF under grants CCF-2107455 and DMS-2210734 and by research awards/gifts from Adobe, Amazon, and Mathworks. GB gratefully acknowledges support from the NSF Career award CCF-1940205. AP is grateful to Juba Ziani for helpful discussions and pointers to the literature on differential privacy. We are grateful to the Simons Institute program on Computational Complexity of Statistical Inference, during which our discussions were initiated.

\small
\bibliographystyle{abbrvnat}
\bibliography{refs}

\normalsize

\appendix

\section{An example for the non-existence of efficient reductions} \label{sec:examples-inefficient}
In this section, we provide an example of source/target distributions such that their Le Cam distance is small but the reduction procedure cannot be computed in polynomial time assuming validity of the planted clique conjecture. We note that such an example can be produced for \emph{any} problem with a statistical-computational gap. In particular, even if a small-deficiency reduction from a problem having a statistical-computational gap to one without a statistical-computational gap exists in principle, it cannot be accomplished in polynomial time unless we violate conjectures in average-case complexity.

Let us give a concrete such example below. Suppose the unknown parameter $S$ is set-valued, and given by some $k$-sized subset of the set $[d] = \{ 1, 2, \ldots, d \}$. Let $\binom{[d]}{k}$ denote the collection of all such subsets. Let $v^S \in \mathbb{R}^d$ denote a sparse vector such that for all $i \in [d]$,
\[
v^S_i = 
\begin{cases}
\sqrt{k} \quad &\text{ if } i \in S \\
0 &\text{ otherwise.}
\end{cases}
\]
For some unknown $S \in \binom{[d]}{k}$, define the source random variable $X = v^S + z$ for $z \sim \NORMAL(0, I)$. This defines the source statistical model as $S$ varies over all $k$-subsets, and corresponds to a detection version of the classical sparse mean estimation model.

Let $M^S$ denote the sparse matrix with
\[
M^S_{i, j} = 
\begin{cases}
1 &\text{ if } i, j \in S \\
0 &\text{ otherwise.}
\end{cases}
\]
For a random matrix $Z$ having i.i.d. standard normal entries and unknown $S \in \binom{[d]}{k}$, define the output $Y = M^S + Z$. In other words, the target model is given by the sparse submatrix detection model~\citep{ma2015computational} as $S$ varies over all $k$-subsets. 

Suppose that $4 \log d \ll k \ll \sqrt{d}$. We will now show that (a) There exist Markov kernels that can map the source experiment to the target experiment as well as the target experiment to the source up to arbitrarily small total variation distance, thereby showing that the Le Cam distance is small. (b) If a Markov kernel from the target to the source attains small deficiency and is implementable in polynomial time, then the planted clique conjecture is false. Taken together, these two observations constitute an example: the Le Cam distance between these models is small but the transformation of the models cannot be computed in polynomial time in the sample size (assuming validity of the planted clique conjecture).

\paragraph{Reduction from source to target:} Given the random vector $X$, compute the subset 
\begin{align} \label{eq:vector-threshold}
\widehat{S} = \arg\min_{\widetilde{S} \subseteq \binom{[d]}{k}} \| X - v^{\widetilde{S}} \|_2.
\end{align}
Note that the estimator $\widehat{S}$ can be computed efficiently by selecting entries with the $k$ largest values --- this is an immediate application of the rearrangement inequality.

Now generate $\widehat{Y} = M^{\widehat{S}} + Z'$, where $Z'$ is a random matrix with standard Gaussian entries chosen independently of everything else. Letting $Y^S = M^S + Z$ denote the output random variable that we are interested in mapping to, we have
\[
\mathsf{d_{TV}}(\widehat{Y}, Y^S) \leq \Pr\{ \widehat{S} \neq S \} \overset{\1}{=} o(1), 
\] 
where step $\1$ follows since $k \gg 4 \log d$, from a simple union bound calculation, since for $\widehat{S}$ to be exactly equal to $S$ we need $\max_{j \in [d] \setminus S} X_j < \min_{j \in S} X_j$. Since this holds for every choice of $S \in \binom{[d]}{k}$, we have a polynomial-time reduction from source to target attaining $o(1)$ deficiency.

\paragraph{Reduction from target to source:} Now we proceed the other way. Given random matrix $Y$ of the above form, we set 
\begin{align} \label{eq:matrix-threshold}
\widehat{S} = \arg\min_{\widetilde{S} \subseteq \binom{[d]}{k}} \| Y - M^{\widetilde{S}} \|_F.
\end{align}
Next, generate $\widehat{X} = v^{\widehat{S}} + z'$, where $z'$ is a random $d$-dimensional vector with standard Gaussian entries chosen independently of everything else. Letting $X^S = v^S + z$ denote the source random variable that we are interested in mapping to, we have
\[
\mathsf{d_{TV}}(\widehat{X}, X^S) \leq \Pr\{ \widehat{S} \neq S \} \overset{\2}{=} o(1), 
\] 
where step $\2$ follows from results on detection of a sparse submatrix~\citep{butucea2013detection,ma2015computational}. Since this holds for every choice of $S \in \binom{[d]}{k}$, we have a reduction from target to source attaining $o(1)$ deficiency.

Taken together with the previous reduction, we see that our source and target models have small (i.e. $o(1)$) Le Cam distance. We note that the $o(1)$ term can be made explicit in terms of the tuple $(k, d, n)$, but we have chosen not to so as to keep our exposition simple.

\paragraph{Efficient reduction from target to source violates planted clique conjecture:} We know that computationally efficient submatrix detection in the regime $4 \log d \ll k \ll \sqrt{d}$ is impossible given the planted clique conjecture~\citep{ma2015computational}. 
Assume for the sake of contradiction that a computationally efficient Markov kernel exists from submatrix detection to sparse mean detection. But, as sketched on page 2, if a computationally efficient Markov kernel exists from submatrix detection to sparse mean detection, then sparse mean detection must also be computationally hard. However, as argued above in Eq.~\eqref{eq:vector-threshold}, the sparse mean detection problem is solvable in polynomial time, which leads to a contradiction.

Note that another interpretation of our argument is that the optimization problem in Eq.~\eqref{eq:matrix-threshold} cannot be computed in polynomial time unless the planted clique conjecture is false.

\section{Total variation deficiency and risk for unbounded losses} \label{sec:loss-unbounded}

In this section, we state a result relating the total variation deficiency between two statistical models to risk bounds on estimators when the loss function is unbounded.

Suppose $L:\Theta \times \Theta \rightarrow \mathbb{R}_{\geq 0}$ is some loss function, $\widehat{\theta}:\mathbb{Y} \rightarrow \Theta$ is an estimator for the target model, and $\textsc{K}:\mathbb{X} \rightarrow \mathbb{Y}$ is a reduction algorithm from the source model to the target model. Then $\widetilde{\theta}:= \widehat{\theta} \circ \textsc{K}: \mathbb{X} \rightarrow \Theta$ is an estimator for the source model. The following lemma gives an analog of Eq.~\eqref{eq:pointwise-risk-transfer} for such an unbounded loss.

\begin{lemma}\label{lemma:risk-bound}
Let $\theta^* \in \Theta$ be the unknown ground truth parameter. Let $X_{{\theta^*}}$ be a sample from the source model and $Y_{{\theta^*}}$ be a sample from the target model. Then for all scalars $\overline{L}>0$ and $p,q\geq 1$ with $1/p + 1/q = 1$, we have
\begin{align}\label{ineq:risk-bound}
	&\Big| \EE\big\{  L\big(\widetilde{\theta}, \theta^* \big) \big\} - \EE\big\{ L\big( \widehat{\theta},\theta^*\big) \big\} \Big|  \notag \\
	&\leq \overline{L} \cdot \mathsf{d_{TV}}\big( \textsc{K}(X_{\theta^*}), Y_{\theta^*} \big) + \EE\big\{ L\big( \widetilde{\theta}, \theta^* \big)^{p} \big\}^{1/p} \cdot \Pr\big\{ L\big( \widetilde{\theta}, \theta^* \big) \geq \overline{L} \big\}^{1/q}
	 + \EE\big\{ L\big( \widehat{\theta}, \theta^* \big)^{p} \big\}^{1/p} \cdot \Pr\big\{ L\big( \widehat{\theta}, \theta^* \big) \geq \overline{L} \big\}^{1/q}.
\end{align}
\end{lemma}
\begin{proof}
Note that $\widehat{\theta}$ and $\widetilde{\theta}$ are both random variables supported on the parameter space $\Theta$. Let $\mu$ and $\nu$ denote their respective measures. 
Defining the set $\Theta_{\overline{L}} := \{\theta \in \Theta: L(\theta, \theta^*) \leq \overline{L}\}$, we have
\begin{align*}
	&\Big| \EE\big\{  L\big(\widetilde{\theta}, \theta^* \big)  \big\} - \EE\big\{ L\big( \widehat{\theta},\theta^* \big) \big\} \Big| \\
	&\qquad \leq \Big| \int_{\Theta_{\overline{L}}} L\big( \theta, \theta^* \big) \cdot \big( \mu(\mathrm{d}\theta) - \nu(\mathrm{d}\theta) \big) \Big| + \int_{\Theta \setminus \Theta_{\overline{L}}} L\big( \theta, \theta^* \big)  \mu(\mathrm{d}\theta) + \int_{\Theta \setminus \Theta_{\overline{L}}} L\big( \theta, \theta^* \big)  \nu(\mathrm{d}\theta) \\
	&\qquad \overset{\1}{\leq} \overline{L} \cdot \mathsf{d_{TV}}\big(\textsc{K}(X_{\theta^*}),Y_{\theta^*}\big) + \EE\big\{ L\big(\widetilde{\theta}, \theta^*\big) \cdot \mathbbm{1}_{L(\widetilde{\theta}, \theta^*) > \overline{L}} \big\}
	+ \EE\big\{ L\big(\widehat{\theta}, \theta^*\big) \cdot \mathbbm{1}_{L(\widehat{\theta}, \theta^*) > \overline{L}} \big\} \\
	&\qquad \overset{\2}{\leq} \overline{L} \cdot \mathsf{d_{TV}}\big(\textsc{K}(X_{\theta^*}),Y_{\theta^*}\big) + \EE\big\{ L\big(\widetilde{\theta}, \theta^*\big)^{p} \big\}^{1/p} \cdot \Pr\big\{ L(\widetilde{\theta}, \theta^*) > \overline{L} \big\}^{1/q} \\
	& \qquad \quad + \EE\big\{ L\big(\widehat{\theta}, \theta^*\big)^{p} \big\}^{1/p} \cdot \Pr\big\{ L(\widehat{\theta}, \theta^*) > \overline{L} \big\}^{1/q}.
\end{align*}
Here step $\1$ follows since $L(\theta, \theta^*) \leq \overline{L}$ on the set $\Theta_{\overline{L}}$ and by the definition of total variation, and step $\2$ follows by H\"older's inequality for $p,q \geq 1$ and $1/p + 1/q = 1$. 
\end{proof}


\section{Concentration results for approximate target models} \label{sec:concentration-specific}

In this section, we prove tail bounds for a sample generated by the algorithm $\textsc{RK}$ when the source is either the exponential or Laplace distribution, and the target is the Gaussian distribution. In conjunction with Lemma~\ref{lemma:risk-bound}, this is used to prove risk bounds in $\ell_2$ on a meta-procedure for the source model.

\begin{lemma}\label{lemma:prob-bound-RK-sample-exp}
Let $X_{\theta}$ be an exponential random variable with pdf~\eqref{eq:exponential-location} and the target $v(\cdot;\theta) = \NORMAL(\theta,\sigma^{2})$. Run $Y \gets \textsc{RK}(X_{\theta},N,M, y_0)$ with $\mathcal{S}^{*}$ in Eq.~\eqref{eq:signed-kernel-Exp}, base measure $\mathcal{P}$ in Eq.~\eqref{eq:base-measure-log-concave}, $M=4$, $N\geq 1$, and the point $y_{0} = X_{\theta}+1$. Then all $\sigma \geq 2$, we have
	\[
		\Pr\{ |Y - \theta | \geq t\} \leq 3\exp\left\{ -\frac{t^{2}}{8\sigma^2} \right\} + 5\exp\left\{-\frac{t}{2} \right \}, \quad \text{for all} \quad t \geq 0.
	\]
\end{lemma}
\begin{proof}
Note that $X_{\theta}$ is exponential with mean $\theta$ and unit variance. Applying the law of total probability, we obtain
\begin{align}\label{eq:total-prob-exp-gauss}
\Pr\{ |Y - \theta| \geq t \} = \int_{\theta - 1}^{+\infty} \Pr\{ |Y - \theta| \geq t \;|\; X_{\theta} = x \} \cdot e^{-x + \theta-1} \mathrm{d}x
\end{align}
Condition on $X_{\theta} = x$, using $Y \gets \textsc{RK}(x,N,M, y_0)$ and applying Lemma~\ref{lem:tech-prob} with $y_{0} = x+1$, we obtain
\[
	\Pr\{ |Y - \theta| \geq t \;|\; X_{\theta} = x \} = \frac{1-g(x)}{p(x)} \int_{|y-\theta| \geq t} \mathcal{S}^{*}(y|x) \vee 0 \; \mathrm{d}y + \mathbbm{1}_{|x+1 - \theta| \geq t} \cdot g(x),
\]
where $g(x) = (1-p(x)/M)^{N} \leq 1$. Continuing, we have 
\[
	p(x) \overset{\1}{\geq} 1 - \tau(\sigma) \overset{\2}{\geq} 1 - 2e^{-\sigma^2/2} \geq 1/2, \quad \text{for all} \quad x \in \real, \sigma \geq 2,
\]
where in step $\1$ we use inequality~\eqref{ineq:px-qx-bound-exp-log-concave} in the proof of Theorem~\ref{thm:exp-log-concave}, and in step $\2$ we use $\tau(\sigma) \leq 2e^{-\sigma^2/2}$ in Example~\ref{example-exp-gaussian}. Putting the pieces together, we obtain
\begin{align}\label{ineq:condition-tail-exp-gauss}
	\Pr\{ |Y - \theta| \geq t \;|\; X_{\theta} = x \} &\leq 2 \int_{|y-\theta| \geq t} \mathcal{S}^{*}(y|x) \vee 0 \; \mathrm{d}y +  \mathbbm{1}_{|x+1 - \theta| \geq t} \nonumber \\
	& \overset{\1}{=}  2 \int_{| s + x +1 - \theta| \geq t} \phi_{\sigma}(s) (1-s/\sigma^2) \vee 0  \; \mathrm{d}s + \mathbbm{1}_{|x+1 - \theta| \geq t} \nonumber \\
	& = 2T_{1} + 2T_{2} + \mathbbm{1}_{|x+1 - \theta| \geq t},
\end{align}
where in step $\1$ we use from Eq.~\eqref{eq:signed-kernel-Exp}, 
\[
	\mathcal{S}^{*}(y|x) = \frac{1}{\sqrt{2\pi} \sigma} \exp\left\{ -\frac{(y-x-1)^{2}}{2\sigma^2}\right\} \Big( 1 - \frac{y-x-1}{\sigma^2} \Big).
\] 
and let $s = y - x -1$, and in the last step we let
\[
	T_{1} =  \int_{-\infty}^{\theta - x - 1 - t} \phi_{\sigma}(s) (1-s/\sigma^2) \vee 0   \; \mathrm{d}s \quad \text{and} \quad T_{2} = \int_{\theta - x - 1 + t}^{+\infty} \phi_{\sigma}(s) (1-s/\sigma^2) \vee 0\;  \mathrm{d}s.
\]
We first turn to bound $T_{1}$. Note that $x \geq \theta-1$. Thus, $\theta - x - 1 \leq 0$ and for $t\geq 0$,
\begin{align*}
	T_{1} = \int_{-\infty}^{\theta - x - 1 - t} \phi_{\sigma}(s) (1-s/\sigma^2)  \; \mathrm{d}s &\leq \int_{-\infty}^{-t} \phi_{\sigma}(s) (1-s/\sigma^2)  \; \mathrm{d}s \\
	& \leq e^{-t^{2} / (2\sigma^2) } \cdot  \int_{-\infty}^{0} \phi_{\sigma}(s) \mathrm{d}s + \int_{-\infty}^{-t} \phi_{\sigma}(s) \frac{-s}{\sigma^2}  \; \mathrm{d}s \\
	& = \frac{1}{2} e^{-t^{2} / (2\sigma^2) } + \frac{1}{\sqrt{2\pi} \sigma} e^{-t^{2} / (2\sigma^2) } \leq e^{-t^{2} / (2\sigma^2) }.
\end{align*}
Substituting the inequality in the display and inequality~\eqref{ineq:condition-tail-exp-gauss} into Eq.~\eqref{eq:total-prob-exp-gauss} yields
\begin{align} \label{ineq1:tail-exp-gauss}
	\Pr\{ |Y - \theta| \geq t \} &\leq 2 e^{-t^{2} / (2\sigma^2) } + 2 \int_{\theta - 1}^{+\infty} T_{2} \cdot e^{-x + \theta-1} \mathrm{d}x + \int_{\theta-1}^{+\infty} \mathbbm{1}_{|x+1 - \theta| \geq t} \cdot e^{-x + \theta-1} \mathrm{d}x \nonumber \\
	& = 2 e^{-t^{2} / (2\sigma^2) } + 2 \int_{0}^{+\infty} \int_{t - z}^{+\infty} \phi_{\sigma}(s) (1-s/\sigma^2) \vee 0 \; \mathrm{d}s \cdot e^{-z} \mathrm{d} z + e^{-t},
\end{align}
where in the last step we let $z = x - \theta + 1$ and 
\[ 
	\int_{\theta-1}^{+\infty} \mathbbm{1}_{|x+1 - \theta| \geq t} \cdot e^{-x + \theta-1} \mathrm{d}x = \int_{0}^{+\infty} \mathbbm{1}_{|z| \geq t} \cdot e^{-z} \mathrm{d}z = e^{-t}.
\] 
Continuing, for $z \leq t/2$, we have $t-z \geq t/2\geq 0$ so that
\[
	\int_{t - z}^{+\infty} \phi_{\sigma}(s) (1-s/\sigma^2) \vee 0 \; \mathrm{d}s \leq \int_{t/2}^{+\infty} \phi_{\sigma}(s) \mathrm{d}s \leq \frac{1}{2} \exp\left\{ - \frac{t^2}{8\sigma^2} \right\}.
\]
For $z \geq t/2$, we obtain
\begin{align*}
	\int_{t -z}^{+\infty} \phi_{\sigma}(s) (1-s/\sigma^2) \vee 0 \; \mathrm{d}s &\leq \int_{-\infty}^{\sigma^{2}} \phi_{\sigma}(s) (1-s/\sigma^2) \; \mathrm{d}s \\
	& \leq 1 + \frac{1}{\sqrt{2\pi}\sigma} \exp\left\{ -\sigma^2/2 \right\} \leq 2.
\end{align*}
Putting together the two inequalities in the display above yields
\begin{align*}
	\int_{0}^{+\infty} \int_{t - z}^{+\infty} \phi_{\sigma}(s) (1-s/\sigma^2) \vee 0 \; \mathrm{d}s \cdot e^{-z} \mathrm{d} z &\leq \frac{1}{2} \int_{0}^{t/2} e^{-z} e^{-t^{2} / (8\sigma^2) } \mathrm{d}z + \int_{t/2}^{+\infty} 2e^{-z} \mathrm{d} z  \\ 
	& \leq \frac{1}{2} e^{-t^{2} / (8\sigma^2) } + 2e^{-t/2}.
\end{align*}
Substituting the inequality in the display into inequality~\eqref{ineq1:tail-exp-gauss} yields
\[
	\Pr\{ |Y - \theta| \geq t \} \leq 3 e^{-t^{2} / (8\sigma^2) } + 5 e^{-t/2}, \quad \text{for all} \quad t \geq 0. 
\]
\end{proof}

\begin{lemma}\label{lemma:prob-bound-RK-sample-Laplace}
Let $X_{\theta} \sim \mathsf{Lap}(\theta,b)$ and the target $v(\cdot;\theta) = \NORMAL(\theta,\sigma^{2})$. Run $Y \gets \textsc{RK}(X_{\theta},N,M, y_0)$ with $\mathcal{S}^{*}$ in Eq.~\eqref{S-star-Laplace-general}, base measure $\mathcal{P}(\cdot|x) = \NORMAL(x,\sigma^2)$ for all $x\in\real$, $M=2$, $N\geq 1$, and the initial point $y_{0} = X_{\theta}$. Then for $\sigma \geq b >0$, we have
\begin{subequations}
\begin{align}\label{tail-bound-laplace-gauss}
	\Pr\{ |Y - \theta | \geq t\} \leq 2\exp\left\{ -\frac{t^{2}}{8\sigma^2} \right\} + 3\exp\left\{-\frac{t}{2b} \right \}, \quad \text{for all} \quad t \geq 0.
\end{align}
Moreover, we have for all $\sigma,b>0$
\begin{align}\label{var-bound-laplace-gauss}
	\EE\big\{ |Y - \theta |^{2} \big\} \leq 2b^2 + \sigma^{2} \Big( 1+\frac{\sigma}{b} \exp\left\{ - \frac{\sigma^2}{2b^2} \right\} \Big).
\end{align} 
\end{subequations}
\end{lemma}
\begin{proof} We prove each part in turn.\\
\noindent \underline{Proof of inequality~\eqref{tail-bound-laplace-gauss}:}
Applying the law of total probability, we obtain
\begin{align}\label{eq:total-prob-laplace-gauss}
\Pr\{ |Y - \theta| \geq t \} = \int_{\real} \Pr\{ |Y - \theta| \geq t \;|\; X_{\theta} = x \} \cdot \frac{e^{-|x-\theta|/b}}{2b} \mathrm{d}x.
\end{align}
Condition on $X_{\theta} = x$, using $Y \gets \textsc{RK}(x,N,M, y_0)$ and applying Lemma~\ref{lem:tech-prob} with $y_{0} = x$, we obtain
\[
	\Pr\{ |Y - \theta| \geq t \;|\; X_{\theta} = x \} = \frac{1-g(x)}{p(x)} \int_{|y-\theta| \geq t} \mathcal{S}^{*}(y|x) \vee 0 \; \mathrm{d}y + \mathbbm{1}_{|x - \theta| \geq t} \cdot g(x),
\]
where $g(x) = (1-p(x)/M)^{N} \leq 1$. Note that from Eq.~\eqref{eq:px-qx-laplace-gauss-b}, we obtain $p(x) \geq 1$. By Eq.~\eqref{S-star-Laplace-general}, we obtain for $\sigma\geq b$,
\[
	\mathcal{S}^{*}(y|x) = \phi_{\sigma}(y-x) \cdot \big( 1 + b^2\sigma^{-2} - b^2\sigma^{-4}(y-x)^2 \big) \leq 2 \phi_{\sigma}(y-x).
\]
Putting the pieces together yields
\begin{align*}
	\Pr\{ |Y - \theta| \geq t \;|\; X_{\theta} = x \} &\leq \int_{|y-\theta| \geq t} 2 \phi_{\sigma}(y-x)  \mathrm{d}y + \mathbbm{1}_{|x - \theta| \geq t}  \\
	&= \int_{|s+x-\theta|\geq t} 2\phi_{\sigma}(s) \mathrm{d}s + \mathbbm{1}_{|x - \theta| \geq t}.
\end{align*}
Substituting the inequality in the display into inequality~\eqref{eq:total-prob-laplace-gauss} yields
\begin{align*}
	\Pr\{ |Y - \theta| \geq t \} &\leq \int_{\real} \frac{e^{-|x-\theta|/b}}{2b} \cdot \int_{|s+x-\theta|\geq t} 2\phi_{\sigma}(s) \mathrm{d}s \mathrm{d}x + \int_{\real} \frac{e^{-|x-\theta|/b}}{2b} \cdot \mathbbm{1}_{|x - \theta| \geq t} \mathrm{d}x \\ 
	&\leq \int_{\real} \frac{e^{-|z|/b}}{b} \cdot \int_{|s+z|\geq t} \phi_{\sigma}(s) \mathrm{d}s \mathrm{d}z + e^{-t/b}.
\end{align*}
Note that for $|z| \leq t/2$, we have $|s| \geq t/2$ if $|s+z| \geq t$, whence we deduce the bound
\begin{align*}
\int_{\real} \frac{e^{-|z|/b}}{b} \cdot \int_{|s+z|\geq t} \phi_{\sigma}(s) \mathrm{d}s \mathrm{d}z &= \int_{|z|\leq t/2} \frac{e^{-|z|/b}}{b} \int_{|s|\geq t/2}  \phi_{\sigma}(s) \mathrm{d}s \mathrm{d}z + \int_{|z|>t/2} \frac{e^{-|z|/b}}{b} \mathrm{d}z \\
&\leq 2 e^{-t^{2} / (8\sigma^2) } + 2e^{-t/(2b)},
\end{align*}
where in the last step we use $\int_{|s|\geq t/2}  \phi_{\sigma}(s) \mathrm{d}s \leq e^{-t^{2}/(8\sigma^2)}$. Combining the two inequalities in the display above completes the proof.\\

\noindent \underline{Proof of inequality~\eqref{var-bound-laplace-gauss}:} Let $T$ be the total number of iterations in the randomized algorithm $Y \gets \textsc{RK}(X_{\theta},N,M, y_0)$. Applying the law of total expectation yields
\begin{align}\label{eq:cond-expect-total-lap-gauss}
\hspace{-0.5cm}
	\EE\big\{ |Y-\theta|^{2} \big\} = \sum_{k=1}^{N} \EE\big\{ |Y-\theta|^{2} \;\vert\; T = k\big\} \cdot \Pr\{T=k\} + \EE\big\{ |Y-\theta|^{2} \;\vert\; T > N \big\} \cdot \Pr\{T > N\}.
\end{align}
We next calculate the conditional expectation in the display. For each $1\leq k \leq N$, we obtain for all $t\in \real$,
\begin{align*}
	\Pr\big\{Y - \theta \leq t \;\vert\; T = k, X_{\theta} = x \big\} &= \frac{\Pr\big\{ \{Y\leq t\} \cap \{T = k\} \;\vert\; X_{\theta} = x \big\}}{\Pr\{T=k\}} \\
	&= \int_{-\infty}^{t+\theta} \frac{\mathcal{S}^{*}(y|x) \vee 0 }{p(x)} \mathrm{d}y,
\end{align*}
where in the last step we use Eq.~\eqref{eq:distr-Y-RK} and $\Pr\{T=k\} = (1-p(x)/M)^{k-1}p(x)/M$ by the calculations in Section~\ref{sec:rej-samp-tech-lemma}. By the law of total probability, we obtain for all $t\in \real$,
\begin{align*}
	\Pr\big\{Y - \theta \leq t \;\vert\; T = k \big\} = \int_{\real} \frac{e^{-|x-\theta|/b}}{2b} \int_{-\infty}^{t+\theta} \frac{\mathcal{S}^{*}(y|x) \vee 0 }{p(x)} \mathrm{d}y \mathrm{d}x.
\end{align*}
Differentiating with respect to $t$, we obtain the conditional pdf of the random variable $Y-\theta$
\begin{align*}
	f_{Y-\theta}(t \; \vert \; T=k) &= \int_{\real} \frac{e^{-|x-\theta|/b}}{2b}  \frac{\mathcal{S}^{*}(t+\theta|x) \vee 0 }{p(x)} \mathrm{d}x \\
	& \overset{\1}{=}  \int_{\real} \frac{e^{-|x-\theta|/b}}{2b} \cdot \frac{ \phi_{\sigma}(t+\theta-x) \big(1+b^{2}\sigma^{-2} - b^{2}\sigma^{-4}(t+\theta-x)^2 \big) \vee 0 }{1+C_{\sigma,b}} \mathrm{d}x \\
	& = \int_{\real} g_{1}(z) \cdot g_{2}(t-z)  \mathrm{d}z,
\end{align*}
where in step $\1$ we use the closed-form of $\mathcal{S}^*$~\eqref{S-star-Laplace-general}, from Eq.~\eqref{eq:px-qx-laplace-gauss-b}
\[
	p(x) = 1 + C_{\sigma,b}, \quad \text{where} \quad C_{\sigma,b} = 2\int_{\sigma\sqrt{\sigma^2 + b^2}/b}^{+\infty} \phi_{\sigma}(t) \Big( b^{2} \sigma^{-4}t^{2} - 1 - b^{2}\sigma^{-2} \Big) \mathrm{d}t,
\]
and in the last step we let $z = x-\theta$ and 
\[
	g_{1}(z) = \frac{e^{-|z|/b}}{2b}, \quad g_{2}(z) =  \frac{ \phi_{\sigma}(z) \big(1+b^{2}\sigma^{-2} - b^{2}\sigma^{-4}z^2 \big) \vee 0 }{1+C_{\sigma,b}}, \quad \text{for all } z \in \real.
\]
Note that $g_{1}(z),g_{2}(z)$ are both probability density functions, and $f_{Y-\theta}(t\;\vert\; T=k) = (g_1 * g_2)(t)$. Thus, condition on $T=k$, $Y-\theta \overset{(d)}{=} X_{1} + X_{2}$, where $X_1$ and $X_2$ are independent random variables with pdf $g_1$ and $g_2$, respectively. Consequently, using $X_1$ is zero mean, we obtain
\begin{align}\label{ineq:cond-variance-k-lap-gauss}
	\EE\{ |Y-\theta|^{2} \;\vert \; T=k \} = \EE\{ X_{1}^{2} \} + \EE\{X_{2}^{2}\} \overset{\1}{\leq} 2b^2 + \sigma^2 \big( 1 + \sigma b^{-1} e^{-\frac{\sigma^2}{2b^2}}  \big),
\end{align}
Here in step $\1$ we use $\EE\{ X_{1}^{2} \} = 2b^2$, and with $\Delta = \sigma\sqrt{\sigma^2 + b^2}/b$ and $C_{\sigma,b} \geq 0$, we have
\begin{align*}
	\EE\{X_2^2\} = \int_{\real} z^{2} g_{2}(z) \mathrm{d}z &\leq \int_{-\Delta}^{\Delta} z^{2}\phi_{\sigma}(z)(1+b^2\sigma^2 - b^{2}\sigma^{-4}z^2) \mathrm{d}z \\
	& \overset{\1}{=} (1+b^2\sigma^{-2}) \int_{-\Delta}^{\Delta} z^{2}\phi_{\sigma}(z) \mathrm{d}z - b^2\sigma^{-4} \Big( 3\sigma^2 \int_{-\Delta}^{\Delta} z^{2}\phi_{\sigma}(z) \mathrm{d}z - 2\sigma^2 \Delta^3 \phi_{\sigma}(\Delta) \Big) \\
	& = (1-2b^{2} \sigma^{-2}) \int_{-\Delta}^{\Delta} z^{2}\phi_{\sigma}(z) \mathrm{d}z + 2b^2\sigma^{-2}\Delta^3 \phi_{\sigma}(\Delta) \\
	& \overset{\2}{\leq} \sigma^2-2b^{2} +   2b^2\sigma^{-2}\Delta^3 \phi_{\sigma}(\Delta) \leq \sigma^2 \Big( 1 + \sigma b^{-1} \exp\big(-\sigma^2/(2b^2)\big)  \Big),
\end{align*}
where in step $\1$ we apply integration by parts so that
$
	\int_{-\Delta}^{\Delta} \phi_{\sigma}(z)z^4 \mathrm{d}z = 3\sigma^2 \int_{-\Delta}^{\Delta} z^{2}\phi_{\sigma}(z) \mathrm{d}z - 2\sigma^2 \Delta^3 \phi_{\sigma}(\Delta),
$
in step $\2$ we use $\int_{-\Delta}^{\Delta} z^{2}\phi_{\sigma}(z) \mathrm{d}z \leq \sigma^2$, and in the last step we use
\[
	2b^2\sigma^{-2}\Delta^3 \phi_{\sigma}(\Delta) = \frac{2(\sigma^2+b^2)^{3/2}}{b} \frac{\exp\big(-\sigma^2/(2b^2)\big) }{\sqrt{2\pi e}} \leq 2b^2 + \sigma^3 b^{-1} \exp\big(-\sigma^2/(2b^2)\big).
\]
Finally, condition on $T>N$, we obtain $Y-\theta = X_{\theta} - \theta$, and thus
\begin{align*}
	\EE\big\{ |Y-\theta|^2 \; \vert \; T>N \big\} = \int_{\real} \frac{e^{-|x-\theta|/b}}{2b} (x-\theta)^{2} \mathrm{d}x = 2b^{2}.
\end{align*}
Now substituting the bound in the display and inequality~\eqref{ineq:cond-variance-k-lap-gauss} into Eq.~\eqref{eq:cond-expect-total-lap-gauss}, we obtain
\begin{align*}
	\EE\big\{ |Y-\theta|^2 \big\} &\leq \Big( 2b^2 + \sigma^2 \big( 1 + \sigma b^{-1} e^{-\frac{\sigma^2}{2b^2}}  \big) \Big) \cdot \Pr\{T\leq N\} + 2b^2  \cdot \Pr\{T>N\} \\
	&\leq 2b^2 + \sigma^{2} \Big( 1+\frac{\sigma}{b} \exp\left\{ - \frac{\sigma^2}{2b^2} \right\} \Big).
\end{align*}

\end{proof}


\section{Proofs of technical lemmas} \label{sec:proof-technical}

In this appendix, we collect proofs of the technical lemmas stated in the main text.

\subsection{Technical results used in proof of Lemma~\ref{lem:rej-sampling}}

Lemma~\ref{lem:rej-sampling} relied on an auxiliary result about rejection sampling. We prove it below for completeness.

\subsubsection{Proof of Lemma~\ref{lem:tech-prob}} \label{sec:rej-samp-tech-lemma}
Note that by sampling $U_{t} \sim \mathsf{Unif}([0, 1])$ and $Y_{t} \sim \mathcal{P}(\cdot | x)$ mutually independently, we obtain
\begin{align}\label{eq:prob-success-sample}
	\Pr\bigg\{ U_t \leq \frac{\overline{\mathcal{S}}(Y_{t} | x)}{M\cdot \mathcal{P}(Y_{t} | x)} \bigg\} \nonumber &= \EE\left[  \mathbbm{1}_{U_t \leq \frac{\overline{\mathcal{S}}(Y_{t} | x)}{M\cdot \mathcal{P}(Y_{t} | x)} }  \right] 
	\\& =  \EE\left[ \EE\big[ \mathbbm{1}_{U_t \leq \frac{\overline{\mathcal{S}}(Y_{t} | x)}{M\cdot \mathcal{P}(Y_{t} | x)} } \;\vert\; Y_{t} \big]  \right] \nonumber
	\\ & = \EE\left[  \frac{\overline{\mathcal{S}}(Y_{t} | x)}{M\cdot \mathcal{P}(Y_{t} | x)}   \right] \nonumber
	\\ & =  \int_{\mathbb{Y}} \frac{\overline{\mathcal{S}}(y | x)}{M\cdot \mathcal{P}(y | x)} \cdot \mathcal{P}(y | x) \mathrm{d} y = \frac{p(x)}{M},
\end{align}
where in the last step we use the definition $p(x) = \int_{\mathbb{Y}} \overline{\mathcal{S}}(y | x) \mathrm{d} y$ from Eq.~\eqref{eq:pq}. Given $x \in \mathbb{X}$, let $T$ be the total number of iterations of the randomized algorithm $x \mapsto \textsc{rk}(x, N, M, y_0)$ and recall that $Y \gets \textsc{rk}(x, N, M, y_0)$ denotes the output. With Eq.~\eqref{eq:prob-success-sample} in hand, we now turn to evaluate $\Pr\{Y \in C \}$ for any measurable set $C \subseteq \mathcal{B}(\mathbb{Y})$. Applying the law of total probability, we obtain
\begin{align}\label{eq:total-prob-proof-master-lemma}
	\Pr\big\{ Y \in C \big\} = \Pr\big\{ Y \in C \; \vert \; T > N \big\} \cdot \Pr\big\{ T > N \big\} + \sum_{k=1}^{N} \Pr\big\{ \{Y \in C\} \; \cap \; \{ T = k\} \big\}.
\end{align}
On the one hand, when $k \leq N$, we have
\begin{align}\label{eq:distr-Y-RK}
	&\Pr\big\{ \{Y \in C\} \; \cap \; \{T = k\}  \big\} \nonumber \\
	&= \Pr \left\{ \{Y_{k} \in C\} \; \cap \; \Big\{ U_{k} \leq \frac{\overline{\mathcal{S}}(Y_{k} | x)}{M\cdot \mathcal{P}(Y_{k} | x)}  \Big\} \; \cap \; \Big\{ U_{t} > \frac{\overline{\mathcal{S}}(Y_{t} | x)}{M\cdot \mathcal{P}(Y_{t} | x)},\; \text{ for all }\; 1 \leq t \leq k-1  \Big\}  \right\} \nonumber
	\\ 
	&\overset{\1}{=} \prod_{t = 1}^{k-1} \Pr\left\{ U_{t} > \frac{\overline{\mathcal{S}}(Y_{t} | x)}{M\cdot \mathcal{P}(Y_{t} | x)} \right\} \cdot \Pr \left\{ \{Y_{k} \in C\} \; \cap \; \Big\{ U_{k} \leq \frac{\overline{\mathcal{S}}(Y_{k} | x)}{M\cdot \mathcal{P}(Y_{k} | x)}  \Big\} \right\}  \nonumber \\ 
	&\overset{\2}{=} \Big( 1-\frac{p(x)}{M} \Big)^{k-1} \cdot \int_{y \in C} \left( \int_{0}^{\frac{\overline{\mathcal{S}}(y | x)}{M\cdot \mathcal{P}(y | x)}}  1 \; \mathrm{d} u \right) \;  \mathcal{P}(y | x) \mathrm{d} y  \nonumber \\
	&= \Big( 1-\frac{p(x)}{M} \Big)^{k-1} \cdot \int_{y \in C} \frac{\overline{\mathcal{S}}(y | x)}{M} \mathrm{d}y.
\end{align}
In step $\1$, we use the mutual independence between the tuple of random variables $(U_{t},Y_{t})$ and $(U_{t'},Y_{t'})$ for $t\neq t'$, and in step $\2$ we use Eq.~\eqref{eq:prob-success-sample}. Summing the geometric series and letting $g(x) =  \Big( 1- \frac{p(x)}{M}\Big)^{N}$ for convenience, we have
\begin{subequations}
\begin{align}
\sum_{k = 1}^N \Pr\big\{ \{Y \in C\} \; \cap \; \{T = k\}  \big\} &= \frac{M}{p(x)} \cdot \left[ 1 - \left( 1 - \frac{p(x)}{M}\right)^N \right] \cdot \int_{y \in C} \frac{\overline{\mathcal{S}}(y | x)}{M} \mathrm{d}y \notag \\
&= (1 - g(x)) \cdot \int_{y \in C} \frac{\overline{\mathcal{S}}(y | x)}{p(x)} \mathrm{d}y. \label{eq:small-k-bound}
\end{align}

On the other hand, when $T > N$, we have $Y = y_{0}$, and so
\begin{align}\label{eq:condition-prob2-proof-master-lemma}
	\Pr\big\{ Y \in C \; \vert \; T>N \big\} \cdot \Pr\big\{ T > N \big\} = \mathbbm{1}_{y_{0} \in C} \cdot \Big( 1- \frac{p(x)}{M}\Big)^{N} =  \mathbbm{1}_{y_{0} \in C} \cdot g(x).
\end{align}
\end{subequations}
Substituting Eq.~\eqref{eq:small-k-bound} and Eq.~\eqref{eq:condition-prob2-proof-master-lemma} into Eq.~\eqref{eq:total-prob-proof-master-lemma}, we obtain
\begin{align*}
	\Pr\big\{ Y \in C \big\} =  \frac{\int_{y \in C} \overline{\mathcal{S}}(y | x) \mathrm{d}y }{p(x)} \cdot (1-g(x)) + \mathbbm{1}_{y_{0} \in C} \cdot g(x),
\end{align*}
as claimed.
\qed

\subsection{Technical results used in the proof of Proposition~\ref{prop:Laplace}}

We provide the proof of the integration by parts formula in Lemma~\ref{lemma:aux-pf-prop-laplace}, which holds under the assumptions in Eq.~\eqref{assump-target-laplace}.

\subsubsection{Proof of Lemma~\ref{lemma:aux-pf-prop-laplace}}\label{sec:pf-lemma-prop-laplace-integral}
We may write the given integral as
\begin{align}\label{eq1:pf-lemma-prop-laplace}
	\int_{\real} \frac{1}{2b_{k}}\exp( - |x_{k} - \theta_{k}| / b_{k}) \cdot v_{k}(y;x) \mathrm{d}x_{k} &= \frac{\exp( - \theta_{k} / b_{k} ) }{2b_{k}} \int_{-\infty}^{\theta_{k}} \exp( x_{k}/b_{k} ) \cdot v_{k}(y;x) \mathrm{d}x_{k}  \nonumber \\
	 & \quad +  \frac{ \exp(\theta_k / b_k  ) }{2b_{k}} \int_{\theta_{k}}^{+\infty} \exp(-x_k/b_k) \cdot v_{k}(y;x) \mathrm{d}x_k.
\end{align}
We handle the two integrals on the RHS separately. First, we have
\begin{align*}
	\int_{-\infty}^{\theta_{k}}  \frac{\exp(x_{k} / b_{k})}{b_{k}}  v_{k}(y;x) \mathrm{d}x_{k} &= \exp(x_{k} / b_{k}) v_{k}(y;x) \Big \vert_{-\infty}^{\theta_{k}} - \int_{-\infty}^{\theta_{k}} \exp(x_{k} / b_{k}) \cdot \frac{\partial v_{k}(y;x)}{\partial x_{k}} \mathrm{d}x_{k} \\
	& = \exp(x_{k} / b_{k}) \lim_{x_{k} \rightarrow \theta_k }v_{k}(y;x) - \int_{-\infty}^{\theta_{k}} \exp(x_{k} / b_{k})  \cdot \frac{\partial v_{k}(y;x)}{\partial x_{k}} \mathrm{d}x_{k},
\end{align*}
where in the last step we use assumption~\eqref{assump1-target-laplace}. Applying integration by parts to the second term and using assumption~\eqref{assump1-target-laplace} yields
\begin{align*}
	\int_{-\infty}^{\theta_{k}} \exp(x_{k} / b_{k}) \frac{\partial v_{k}(y;x)}{\partial x_{k}} \mathrm{d}x_{k} = b_{k}\exp(x_{k} / b_{k})  \cdot  \frac{\partial v_{k}(y;x)}{\partial x_{k}} \Big \vert_{x_{k} = \theta_k}  - b_{k}\int_{-\infty}^{\theta_{k}} \exp(x_{k} / b_{k}) \cdot \frac{\partial^2 v_{k}(y;x)}{\partial x_{k}^2} \mathrm{d}x_{k}.
\end{align*}
Putting the two pieces together yields
\begin{subequations} \label{eq:two-lou-integrals}
\begin{align}
	\int_{-\infty}^{\theta_{k}}  \frac{\exp(x_{k} / b_{k})}{b_{k}}  v_{k}(y;x) \mathrm{d}x_{k} &= \exp(x_{k} / b_{k}) \lim_{x_{k} \rightarrow \theta_k }v_{k}(y;x)
	- b_{k}\exp(x_{k} / b_{k})  \cdot  \frac{\partial v_{k}(y;x)}{\partial x_{k}} \Big \vert_{x_{k} = \theta_k}  \nonumber \\
	&\quad + b_{k}\int_{-\infty}^{\theta_{k}} \exp(x_{k} / b_{k})  \cdot \frac{\partial^2 v_{k}(y;x)}{\partial x_{k}^2} \mathrm{d}x_{k}.
\end{align}
To bound the second term on the RHS of Eq.~\eqref{eq1:pf-lemma-prop-laplace}, we proceed again via integration by parts
and use assumption~\eqref{assump2-target-laplace}. Carrying out the calculations exactly as above, we obtain
\begin{align}
	\int_{\theta_{k}}^{+\infty} \frac{ \exp(-x_{k} / b_{k}) }{b_{k}} v(y;x) \mathrm{d}x_k &= \exp(-x_{k} / b_{k}) \lim_{x_{k} \rightarrow \theta_k }v_{k}(y;x) +
	b_{k}\exp(-x_{k} / b_{k}) \cdot \frac{\partial v_{k}(y;x)}{\partial x_{k}} \Big \vert_{x_{k} = \theta_k} \nonumber \\
	& \quad +  b_{k} \int_{\theta_{k}}^{+\infty} \exp(-x_{k} / b_{k}) \cdot \frac{\partial^2 v_{k}(y;x)}{\partial x_{k}^2} \mathrm{d}x_{k}. 
\end{align}
\end{subequations}
Substituting Eq.~\eqref{eq:two-lou-integrals} into Eq.~\eqref{eq1:pf-lemma-prop-laplace}, we then have
\[
	\int_{\real} \frac{1}{2b_{k}}\exp( - |x_{k} - \theta_{k}| / b_{k} ) \cdot v_{k}(y;x) \mathrm{d}x_{k} = \lim_{x_{k} \rightarrow \theta_k }v_{k}(y;x) + \int_{\real} \frac{1}{2b_{k}} \exp( - |x_{k} - \theta_{k}| / b_{k} ) \cdot b_{k}^2 \frac{\partial^2 v_{k}(y;x)}{\partial x_{k}^2} \mathrm{d}x_{k},
\]
so that
\begin{align*}
	\int_{\real} \frac{1}{2b_{k}}\exp( - |x_{k} - \theta_{k}| /b_{k} ) \cdot \bigg( v_{k}(y;x) -  b_{k}^{2}\frac{\partial^{2} v_{k}(y;x) }{\partial x_{k}^{2}} \bigg) \mathrm{d} x_{k} = \lim_{x_{k} \rightarrow \theta_k }v_{k}(y;x) = v_{k}(y;x) \mid_{x_{k} = \theta_{k}}
\end{align*}
as claimed.
\qed

\subsection{Technical results used in the proof of Proposition~\ref{prop:Erlang}}
We provide a proof of the combinatorial formula in Lemma~\ref{lemma:combinatorial}.

\subsubsection{Proof of Lemma~\ref{lemma:combinatorial}}\label{sec:pf-lemma-combinatorial}
Multiplying both sides of Eq.~\eqref{eq:T2-Erlang-comb} by $(-1)^{\ell+1}$, we obtain the equivalent claim that for each $0 \leq \ell \leq k-1$, we have
\begin{align}\label{eq1:T2-Erlang-comb}
	\sum_{j= \ell}^{k-1} \binom{k}{j+1}  (-1)^{j+1} \binom{j}{\ell} = (-1)^{\ell+1}.
\end{align}
We now establish Eq.~\eqref{eq1:T2-Erlang-comb} by induction on $\ell$. 

\paragraph*{Base case:} For the base case $\ell=0$, we have
\begin{align*}
	\sum_{j= 0}^{k-1} \binom{k}{j+1}  (-1)^{j+1} \binom{j}{0} = \sum_{t = 1}^{k} \binom{k}{t}  (-1)^{t} &= \sum_{t = 0}^{k} \binom{k}{t}  (-1)^{t} - 1 = (1-1)^{k} - 1 = -1.
\end{align*}

\paragraph*{Induction step:} For our induction hypothesis, suppose that Eq.~\eqref{eq:T2-Erlang-comb} holds for some $\ell = m \geq 0$, i.e.,
\[
	(-1)^{m+1} = \sum_{j= m }^{k-1} \binom{k}{j+1}  (-1)^{j+1} \binom{j}{m} =: T,
\] 
and define
\[
	T' := \sum_{j= m + 1}^{k-1} \binom{k}{j+1}  (-1)^{j+1} \binom{j}{m+1}.
\]
We need to prove that $T' = (-1)^{m+2}$. Adding the terms $T$ and $T'$, we have
\begin{align*}
	T + T' &= \binom{k}{m+1}(-1)^{m+1} + \sum_{j= m + 1}^{k-1} \binom{k}{j+1}  (-1)^{j+1} \bigg[ \binom{j}{m+1} + \binom{j}{m} \bigg]  \\
	& =  \binom{k}{m+1}(-1)^{m+1} +  \binom{k}{m+1} \sum_{j= m + 1}^{k-1}  (-1)^{j+1}  \binom{k-m-1}{j - m},
\end{align*}
where in the last step we use the identity
\[
	\binom{k}{j+1} \bigg[ \binom{j}{m+1} + \binom{j}{m} \bigg] = \binom{k}{j+1} \binom{j+1}{m+1} = \binom{k}{m+1} \binom{k-m-1}{(j+1) - (m+1)}.
\]
Continuing, we have
\begin{align*}
	T + T' = \binom{k}{m+1} \sum_{j= m }^{k-1}  (-1)^{j+1}  \binom{k-m-1}{j - m} &\overset{\1}{=} \binom{k}{m+1} (-1)^{m+1} \sum_{t=0}^{k-1-m}  (-1)^{t}  \binom{k-m-1}{t} \\
	&  = \binom{k}{m+1} (-1)^{m+1} (1-1)^{k-1-m} = 0.
\end{align*}
where in step $\1$ we let $t = j-m$. Thus, we obtain $T' = -T = (-1) \cdot (-1)^{m+1}$, which concludes the induction step.
\qed

\subsection{Technical results used in proof of Theorem~\ref{thm:Unif-Gauss}}

We stated two lemmas in the main text. We provide their proofs below.

\subsubsection{Proof of Lemma~\ref{lemma:auxiliary-proof-Unif-Gauss}} \label{sec:proof-lemma-aux-unif-gauss} 
We prove each part in order. 
\paragraph*{Proof of part (a).} Using expression~\eqref{eq:closed-form-S-star-case1}, we obtain 
\begin{align}\label{ineq:S-star-lower-bd-case1}
	\mathcal{S}^{*}(y|x) &\geq \frac{1}{\sqrt{2\pi} \sigma} \bigg( 2\exp\left\{-\frac{|y-f(0.5)|^{2}}{4\sigma^2} -\frac{|y-f(-0.5)|^{2}}{4\sigma^2} \right\} - \exp\left\{-\frac{|y-f(\theta_{0})|^{2}}{2\sigma^2} \right\}  \bigg) \nonumber
	\\ &= \phi_{\sigma}\big(y-f(\theta_{0}) \big) \cdot \bigg( 2\exp\left\{ \frac{\Delta}{2\sigma^{2}} \cdot y + \frac{2f(\theta_{0})^{2} -f(0.5)^{2} - f(-0.5)^{2}}{4\sigma^2} \right\} - 1 \bigg),
\end{align}
where in the first step we use the numerical inequality $a+b \geq 2\sqrt{ab}$ for $a,b\geq 0$. Consequently, if $\Delta=0$ then
\begin{align*}
	\mathcal{S}^{*}(y|x) &\geq \phi_{\sigma}\big(y-f(\theta_{0}) \big) \cdot \bigg( 2\exp\left\{ \frac{2f(\theta_{0})^{2} -f(0.5)^{2} - f(-0.5)^{2}}{4\sigma^2} \right\} - 1 \bigg) 
	\\ & \overset{\1}{\geq} \phi_{\sigma}\big(y-f(\theta_{0}) \big) \cdot \bigg( 2\exp\left\{ -\frac{1}{50} \right\} - 1 \bigg) \geq 0,
\end{align*}
where in step $\1$ we used $\sigma \geq 5 \big( f(0.5) \vee f(-0.5) \big)$. Thus, if $\Delta = 0$ then $\mathcal{S}^{*}(y|x) \geq 0$ for all $y \in \real$. 

We then turn to consider the case $\Delta > 0$. From inequality~\eqref{ineq:S-star-lower-bd-case1}, we obtain $\mathcal{S}^{*}(y|x) \geq 0$ for all $y \geq y_{0}$. Using expression~\eqref{eq:closed-form-S-star-case1}, we obtain the bound
\begin{subequations}\label{ineq:bound-S-star-proof-case1}
\begin{align}
	\int_{y_{0}}^{+\infty} \mathcal{S}^{*}(y|x) \mathrm{d}y &= \int_{y_0}^{+\infty} \phi_{\sigma}\big(y-f(0.5)\big) + \phi_{\sigma}\big(y-f(-0.5)\big) - \phi_{\sigma}\big(y-f(\theta_0)\big)  \; \mathrm{d}y \nonumber
	\\ &=  1 - \int_{-\infty}^{y_0 - f(0.5)} \phi_{\sigma}(t) \;\mathrm{d}t + \int_{y_{0} - f(-0.5)}^{y_{0} - f(\theta_{0})} \phi_{\sigma}(t) \; \mathrm{d}t,\\
	\int_{-\infty}^{y_{0}} \big| \mathcal{S}^{*}(y|x) \big| \mathrm{d}y &\leq \int_{-\infty}^{y_{0}} \phi_{\sigma}\big(y-f(0.5)\big) +  \phi_{\sigma}\big(y-f(-0.5) \big) + \phi_{\sigma} \big(y-f(\theta_0) \big) \; \mathrm{d}y \nonumber \\ &\leq 3 \int_{-\infty}^{y_{0} + \alpha_{0} } \phi_{\sigma}(t) \; \mathrm{d}t.
\end{align}
\end{subequations}
Note that $|\Delta| \leq 4\alpha_{0}$, and for $\sigma \geq 5 \alpha_{0}$ and $\Delta>0$,
\begin{align*}
	y_{0} + \alpha_{0} &=  
	\frac{-\log(2) \sigma^{2}}{\Delta} + \frac{\alpha_{0} \cdot \Delta + f(0.5)^{2}/2 + f(-0.5)^{2}/2 - f(\theta_{0})^{2} - \log(2)\sigma^{2}  }{\Delta}
	\\ &\leq \frac{-\log(2) \sigma^{2}}{\Delta} + \frac{\alpha_{0} \cdot 4\alpha_{0} + \alpha_{0}^{2} - \log(2)\sigma^{2}  }{\Delta}
	\leq \frac{-\log(2) \sigma^{2}}{\Delta}.
\end{align*}
Using the bound in the display above, we obtain 
\begin{align*}
	\int_{-\infty}^{y_{0} + \alpha_{0}} \phi_{\sigma}(t) \mathrm{d}t \leq \int_{-\infty}^{0} \phi_{\sigma}(t) \mathrm{d}t \cdot \exp\left\{ -\frac{|y_{0}+\alpha_{0}|^{2}}{2\sigma^2} \right\} \leq \frac{1}{2} \exp\left\{-\frac{\log^{2}(2) \sigma^{2}}{2 \Delta^{2} } \right\}.
\end{align*}
Using the inequality in the display above yields
\begin{align*}
	&\int_{-\infty}^{y_0 - f(0.5)} \phi_{\sigma}(t) \;\mathrm{d}t \leq \int_{-\infty}^{y_{0} + \alpha_{0}} \phi_{\sigma}(t) \mathrm{d}t \leq \frac{1}{2} \exp\left\{-\frac{\log^{2}(2) \sigma^{2}}{2 \Delta^{2} } \right\},
	\\&
	\bigg| \int_{y_{0} - f(-0.5)}^{y_{0} - f(\theta_{0})} \phi_{\sigma}(t) \mathrm{d}t \bigg| \leq \int_{-\infty}^{y_{0} + \alpha_{0}} \phi_{\sigma}(t) \mathrm{d}t \leq \frac{1}{2} \exp\left\{-\frac{\log^{2}(2) \sigma^{2}}{2 \Delta^{2} } \right\}.
\end{align*}
Now, substituting the bounds in the display into inequality~\eqref{ineq:bound-S-star-proof-case1} yields the desired inequalities~\eqref{ineq1:S-star-bound-part-1} and~\eqref{ineq2:S-star-bound-part-1}.

We next consider the case $\Delta<0$. Using inequality~\eqref{ineq:S-star-lower-bd-case1}, we obtain $\mathcal{S}^{*}(y|x) \geq 0$ for all $y \leq y_{0}$. Using expression~\eqref{eq:closed-form-S-star-case1}, we obtain the bound 
\begin{subequations}\label{ineq:bound-S-star-proof-Delta-negative}
\begin{align}
	\int_{-\infty}^{y_{0}} \mathcal{S}^{*}(y|x) \mathrm{d}y &= \int_{-\infty}^{y_{0}} \phi_{\sigma}\big(y-f(0.5)\big) + \phi_{\sigma}\big(y-f(-0.5)\big) - \phi_{\sigma}\big(y-f(\theta_0)\big)  \; \mathrm{d}y \nonumber
	\\ &=  1 - \int_{y_{0}-f(0.5)}^{+\infty} \phi_{\sigma}(t) \;\mathrm{d}t + \int_{y_{0}-f(\theta_{0})}^{y_{0} - f(-0.5)} \phi_{\sigma}(t) \; \mathrm{d}t,\\
	\int_{y_{0}}^{+\infty} \big| \mathcal{S}^{*}(y|x) \big| \mathrm{d}y &\leq \int_{y_{0}}^{+\infty} \phi_{\sigma}\big(y-f(0.5)\big) +  \phi_{\sigma}\big(y-f(-0.5) \big) + \phi_{\sigma} \big(y-f(\theta_0) \big) \; \mathrm{d}y \nonumber \\ &\leq 3 \int_{y_{0} - \alpha_{0}}^{+\infty} \phi_{\sigma}(t) \; \mathrm{d}t.
\end{align}
\end{subequations}
Note that $|\Delta| \leq 4\alpha_{0}$, and for $\sigma \geq 5 \alpha_{0}$ and $\Delta<0$,
\begin{align*}
	y_{0} - \alpha_{0} &=  
	\frac{-\log(2) \sigma^{2}}{\Delta} + \frac{-\alpha_{0} \cdot \Delta + f(0.5)^{2}/2 + f(-0.5)^{2}/2 - f(\theta_{0})^{2} - \log(2)\sigma^{2}  }{\Delta}
	\\ &\geq \frac{-\log(2) \sigma^{2}}{\Delta} + \frac{\alpha_{0} \cdot 4\alpha_{0} +  \alpha_{0}^{2} - \log(2)\sigma^{2}  }{\Delta}
	\geq \frac{-\log(2) \sigma^{2}}{\Delta} =  \frac{\log(2) \sigma^{2}}{|\Delta|}.
\end{align*}
Using the bound in the display above, we obtain 
\begin{align*}
	\int_{y_{0}-\alpha_{0}}^{+\infty} \phi_{\sigma}(t) \mathrm{d}t \leq \int_{0}^{+\infty} \phi_{\sigma}(t) \mathrm{d}t \cdot \exp\left\{ -\frac{|y_{0}-\alpha_{0}|^{2}}{2\sigma^2} \right\} \leq \frac{1}{2} \exp\left\{-\frac{\log^{2}(2) \sigma^{2}}{2 \Delta^{2} } \right\}.
\end{align*}
Using the inequality in the display above yields
\begin{align*}
	&\int_{y_0 - f(0.5)}^{+\infty} \phi_{\sigma}(t) \;\mathrm{d}t \leq \int_{y_{0} - \alpha_{0}}^{+\infty} \phi_{\sigma}(t) \mathrm{d}t \leq \frac{1}{2} \exp\left\{-\frac{\log^{2}(2) \sigma^{2}}{2 \Delta^{2} } \right\},
	\\&
	\bigg| \int_{y_{0} - f(-0.5)}^{y_{0} - f(\theta_{0})} \phi_{\sigma}(t) \mathrm{d}t \bigg| \leq \int_{y_{0}-\alpha_{0}}^{+\infty} \phi_{\sigma}(t) \mathrm{d}t \leq \frac{1}{2} \exp\left\{-\frac{\log^{2}(2) \sigma^{2}}{2 \Delta^{2} } \right\}.
\end{align*}
Substituting the inequalities in the display above into inequality~\eqref{ineq:bound-S-star-proof-Delta-negative} proves the desired inequalities~\eqref{ineq1:S-star-bound-part-3} and~\eqref{ineq2:S-star-bound-part-3}.

\paragraph*{Proof of part (b).} We split the argument into two cases.\\
\underline{Case 1: $\theta_{0}-1/2 < x \leq 0$.} From Eq.~\eqref{S-star-uniform}, we obtain
\begin{align}\label{eq:S-star-expression-case2}
	\mathcal{S}^{*}(y|x) &= \phi_{\sigma}\big(y-f(0.5)\big) + \phi_{\sigma}\big( y -f(-0.5) \big) 
	- \phi_{\sigma}\big( y -f(\theta_{0}) \big) \nonumber \\ & \quad - \phi_{\sigma}\big( y-f(x+1/2)\big) \cdot \frac{y - f(x+1/2)}{\sigma^{2}} \cdot f'(x+1/2),
\end{align}
where $\phi_{\sigma}$ is the pdf for Gaussian random variable with zero mean and variance $\sigma$. In this case, we have $|y-f(x+1/2)| \leq \frac{\sigma^{2}}{5\alpha_{1}}$ since $\ell(x) \leq y \leq L(x)$. Using expression~\eqref{eq:S-star-expression-case2}, $|f'(x+1/2)| \leq \alpha_{1}$, and $|y-f(x+1/2)| \leq \frac{\sigma^{2}}{5\alpha_{1}}$, we obtain
\begin{align}\label{ineq0:S_star-positive-proof-case2}
	\mathcal{S}^{*}(y|x) 
	&\geq \phi_{\sigma}\big(y-f(0.5)\big) + \phi_{\sigma}\big( y -f(-0.5) \big) - \phi_{\sigma}\big( y -f(\theta_{0}) \big) - \phi_{\sigma}\big( y-f(x+1/2)\big) / 5 \nonumber
	\\
	&= \phi_{\sigma}\big( y - f(\theta_{0}) \big) \cdot \bigg(\exp\left\{\frac{2[f(0.5)-f(\theta_0)] \cdot y + f(\theta_0)^{2} - f(0.5)^{2}}{2\sigma^2} \right\} \nonumber \\
	& \qquad \qquad \qquad \qquad \quad + \exp\left\{\frac{2[f(-0.5)-f(\theta_0)] \cdot y + f(\theta_0)^{2} - f(-0.5)^{2}}{2\sigma^2} \right\}  \\ 
	& \qquad \qquad \qquad \qquad \quad - 1 - \exp\left\{\frac{2[f(x+1/2)-f(\theta_0)]\cdot y + f(\theta_0)^{2} - f(x+1/2)^{2}}{2\sigma^2} \right\} \cdot \frac{1}{5}\bigg). \nonumber
\end{align}
Note that we obtain for all $-1/2 \leq t\leq 1/2$ and $|y - f(x+1/2)| \leq \sigma^{2}/(5\alpha_{1})$,
\begin{subequations}
\begin{align}\label{ineq1:S_star-positive-proof-case2}
	\frac{2[f(t)-f(\theta_0)] \cdot y + f(\theta_0)^{2} - f(t)^{2}}{2\sigma^2} & \overset{\1}{\geq} \frac{-2|f(t)-f(\theta_0)| \cdot \big( |f(x+1/2)| + \frac{\sigma^2}{5\alpha_{1}} \big) - f(t)^{2}}{2\sigma^2} \nonumber\\
	& \overset{\2}{\geq}  \frac{-2\alpha_{1} \big(\alpha_0 + \frac{\sigma^2}{5\alpha_{1}}  \big) - \alpha_0^2 }{2\sigma^{2}}  \overset{\3}{\geq} -\frac{1}{5} - \frac{3}{50}.
\end{align}
Here, in step $\1$ we use $|y| \leq |f(x+1/2)| + \sigma^2/(5 \alpha_{1})$, in step $\2$, we use $|f(t) - f(\theta_0)| \leq \int_{t}^{\theta_0}|f'(t)| \mathrm{d}t \leq \alpha_{1} \big|t-\theta_0\big| \leq \alpha_{1}$ for $-1/2\leq t,\theta_{0} \leq 1/2$ and $|f(t)|,|f(x+1/2)| \leq \alpha_{0}$, and in step $\3$ we use $\sigma \geq 5 \big( \alpha_{0} \vee \alpha_{1} \big)$. Proceeding similarly as in inequality~\eqref{ineq1:S_star-positive-proof-case2}, we obtain for $-1/2 \leq t\leq 1/2$,
\begin{align}\label{ineq2:S_star-positive-proof-case2}
	\frac{2[f(t)-f(\theta_0)] \cdot y + f(\theta_0)^{2} - f(t)^{2}}{2\sigma^2} &\leq \frac{1}{5} + \frac{3}{50}.
\end{align}
\end{subequations}
Now applying inequality~\eqref{ineq1:S_star-positive-proof-case2} with $t = 0.5$ and $t = -0.5$, inequality~\eqref{ineq2:S_star-positive-proof-case2} with $t = x+1/2$, and substituting them into inequality~\eqref{ineq0:S_star-positive-proof-case2}, we obtain
\begin{align*}
	\mathcal{S}^{*}(y|x) &\geq\phi_{\sigma}\big(y-f(\theta_0)\big) \cdot \Big(2\exp\big( -1/5 - 3/50 \big) -1 - \exp\big( 1/5 + 3/50 \big)/5 \Big) \geq 0.
\end{align*}
We thus conclude that
\begin{align*}
	\mathcal{S}^{*}(y|x) \geq 0 \quad \text{ for all }\; \ell(x) \leq y \leq L(x), \quad \text{where} \quad \ell(x) = f(x+1/2) - \frac{\sigma^2}{5\alpha_{1}},\; L(x) = f(x+1/2) + \frac{\sigma^{2}}{5\alpha_{1}}.
\end{align*}
Next we turn to verify inequalities~\eqref{ineq1:S-star-bound-part-2} and~\eqref{ineq2:S-star-bound-part-2} for $\theta_{0}-1/2 <x \leq 0$. Note that
\[
	\int_{\ell(x)}^{L(x)} \phi_{\sigma}\big(y-f(x+1/2)\big) \cdot \frac{y-f(x+1/2)}{\sigma^2} \mathrm{d}y = \int_{-\frac{\sigma^2}{5\alpha_{1}}}^{\frac{\sigma^2}{5\alpha_{1}}} \phi_{\sigma}(t) \cdot \frac{t}{\sigma^{2}} \mathrm{d}t = 0.
\]
Using the equation in the display and expression~\eqref{eq:S-star-expression-case2}, we obtain
\begin{align*}
	&\int_{\ell(x)}^{L(x)} \mathcal{S}^{*}(y|x) \mathrm{d}y = 
	\int_{\ell(x)}^{L(x)} \phi_{\sigma}\big(y-f(0.5)\big) + \phi_{\sigma}\big(y-f(-0.5)\big) - \phi_{\sigma}\big(y-f(\theta_0) \big)  \mathrm{d}y 
	\\ &= 1 -  \int_{-\infty}^{\ell(x) - f(0.5)} \phi_{\sigma}(t) \mathrm{d}t - \int_{L(x)-f(0.5)}^{+\infty} \phi_{\sigma}(t) \mathrm{d}t + \int_{\ell(x) - f(-0.5)  }^{L(x) - f(-0.5)} \phi_{\sigma}(t) \mathrm{d}t - \int_{\ell(x) - f(\theta_0)}^{L(x) - f(\theta_0)} \phi_{\sigma}(t) \mathrm{d}t
\end{align*}
Using $\sigma \geq 5 \big( \alpha_{0} \vee \alpha_{1} \big)$, we obtain
\begin{align*}
	\ell(x) + \alpha_{0} &= -\frac{\sigma^2}{5\alpha_{1}} + f(x+1/2) + \alpha_{0} \leq -\frac{\sigma^2}{5\alpha_{1}} + 2\alpha_{0} \leq -\frac{\sigma^2}{10\alpha_{1}}, \\
	L(x) - \alpha_{0} &= \frac{\sigma^2}{5\alpha_{1}} + f(x+1/2) - \alpha_{0} \geq \frac{\sigma^2}{5\alpha_{1}} - 2 \alpha_{0} \geq \frac{\sigma^{2}}{10\alpha_{1}}.
\end{align*}
Using the bounds in the display above, we obtain
\begin{align*}
	\int_{-\infty}^{\ell(x) - f(0.5)} \phi_{\sigma}(t) \mathrm{d}t \leq \int_{-\infty}^{\ell(x) + \alpha_{0}} \phi_{\sigma}(t) \mathrm{d}t  \leq  \int_{-\infty}^{0} \phi_{\sigma}(t) \mathrm{d}t  \cdot \exp\left\{-\frac{|\ell(x) + \alpha_{0}|^{2}}{2\sigma^2} \right\} \leq \frac{1}{2} \exp\left\{-\frac{\sigma^{2}}{200 \alpha_{1}^{2} } \right\}, \\
	\int_{L(x)-f(0.5)}^{+\infty} \phi_{\sigma}(t) \mathrm{d}t \leq \int_{L(x) - \alpha_{0}}^{+\infty} \phi_{\sigma}(t) \mathrm{d}t \leq \int_{0}^{+\infty} \phi_{\sigma}(t) \mathrm{d}t \cdot \exp\left\{-\frac{|L(x) - \alpha_{0}|^{2}}{2\sigma^2} \right\} \leq \frac{1}{2} \exp\left\{-\frac{\sigma^{2}}{200 \alpha_{1}^{2}} \right\}.
\end{align*}
Proceeding similarly as in the display above, we obtain
\begin{align*}
	\bigg| \int_{\ell(x) - f(-0.5)  }^{L(x) - f(-0.5)} \phi_{\sigma}(t) \mathrm{d}t - \int_{\ell(x) - f(\theta_0)}^{L(x) - f(\theta_0)} \phi_{\sigma}(t) \mathrm{d}t \bigg| &\leq \int_{\ell(x) - f(-0.5) \vee f(\theta_0)}^{\ell(x)-f(-0.5) \wedge f(\theta_0)} \phi_{\sigma}(t) \mathrm{d}t + \int_{L(x) - f(-0.5) \vee f(\theta_0)}^{L(x)-f(-0.5) \wedge f(\theta_0)} \phi_{\sigma}(t) \mathrm{d}t 
	\\& \leq \int_{-\infty }^{\ell(x) + \alpha_{0}} \phi_{\sigma}(t) \mathrm{d}t + \int_{L(x) - \alpha_{0}}^{+\infty} \phi_{\sigma}(t) \mathrm{d}t 
	\\ &\leq \exp\left\{-\frac{\sigma^{2}}{200 \alpha_{1}^{2}} \right\}.
\end{align*}
Putting all pieces together yields
\[
	 1 - 2 \exp\left\{-\frac{\sigma^{2}}{200 \alpha_{1}^{2}} \right\} \leq \int_{\ell(x)}^{L(x)} \mathcal{S}^{*}(y|x) \mathrm{d}y \leq 1 + 2 \exp\left\{-\frac{\sigma^{2}}{200 \alpha_{1}^{2}} \right\}.
\]
This proves inequality~\eqref{ineq1:S-star-bound-part-2}. Towards to prove inequality~\eqref{ineq2:S-star-bound-part-2}, using expression~\eqref{eq:S-star-expression-case2} and applying the triangle inequality yields
\begin{align*}
	\int_{L(x)}^{+\infty} \big| \mathcal{S}^{*}(y|x) \big| \mathrm{d}y &\leq 
	\int_{L(x)}^{+\infty}  \bigg( \phi_{\sigma}\big(y-f(0.5) \big) + \phi_{\sigma}\big(y-f(-0.5) \big) + \phi_{\sigma}\big(y-f(\theta_0) \big) 
	\\& \qquad \qquad \qquad \qquad + \frac{\alpha_{1}|y-f(x+1/2)| \cdot \phi_{\sigma}\big(y-f(x+1/2) \big) }{\sigma^2} \bigg)  \; \mathrm{d}y 
	\\ & \overset{\1}{\leq} 3 \int_{\frac{\sigma^{2}}{10 \alpha_{1}}}^{+\infty} \phi_{\sigma}(t) \mathrm{d}t + \int_{\frac{\sigma^{2}}{10\alpha_{1}}}^{+\infty} \frac{\alpha_{1}}{\sigma^2} t \cdot \phi_{\sigma}(t) \mathrm{d}t 
	\overset{\2}{\leq} 2 \exp\left\{ -\frac{\sigma^{2}}{200\alpha_{1}^{2}} \right\}, 
\end{align*}
where in step $\1$ we use $L(x) - f(0.5) \vee f(-0.5) \vee f(\theta_0) \vee f(x+1/2) \geq L(x) - \alpha_{0} \geq \sigma^{2}/(10\alpha_{1})$, and in step $\2$ we use
\[
	\int_{\frac{\sigma^{2}}{10\alpha_{1}}}^{+\infty} \phi_{\sigma}(t) \mathrm{d}t \leq \frac{1}{2}\exp\left\{ -\frac{\sigma^2}{200\alpha_{1}^2 } \right\},\quad 
	\int_{\frac{\sigma^{2}}{10\alpha_{1}}}^{+\infty} \frac{\alpha_{1}}{\sigma^2} t \cdot \phi_{\sigma}(t) \mathrm{d}t = \frac{\alpha_{1}}{\sqrt{2\pi}\sigma} \exp\left\{ -\frac{\sigma^2}{200\alpha_{1}^2 } \right\} \leq \frac{1}{2}\exp\left\{ -\frac{\sigma^2}{200\alpha_{1}^2 } \right\}.
\]
Proceeding similarly yields the same bound 
\[
	\int_{-\infty}^{\ell(x)} \big| \mathcal{S}^{*}(y|x) \big| \mathrm{d}y \leq 
	3\int_{-\infty}^{-\frac{\sigma^{2}}{10\alpha_{1}}} \phi_{\sigma}(t) \mathrm{d}t +  \int_{-\infty}^{-\frac{\sigma^{2}}{10\alpha_{1}}} \frac{\alpha_{1}}{\sigma^2} |t| \cdot \phi_{\sigma}(t) \mathrm{d}t \leq 2 \exp\left\{ -\frac{\sigma^{2}}{200 \alpha_{1}^{2}} \right\}.
\]
Putting the pieces together proves inequality~\eqref{ineq2:S-star-bound-part-2}.

\noindent \underline{Case 2: $0< x < \theta_{0} + 1/2$.}
From Eq.~\eqref{S-star-uniform}, we obtain 
\begin{align}\label{eq:S-star-expression-case3}
	\mathcal{S}^{*}(y|x) = & \phi_{\sigma}\big( y - f(0.5)\big) + \phi_{\sigma}\big( y-f(-0.5)\big) - \phi_{\sigma}\big(y-f(\theta_{0})\big) \nonumber \\ & + \phi_{\sigma}\big( y-f(x-1/2)\big) \cdot \frac{y-f(x-1/2)}{\sigma^{2}} \cdot f'(x-1/2).
\end{align}
Note that $|y-f(x-1/2)| \leq \sigma^{2}/(5 \alpha_{1})$ since $\ell(x) \leq y \leq L(x)$. Using expression~\eqref{eq:S-star-expression-case3}, $|f'(x-1/2)| \leq \alpha_1$, and $|y-f(x-1/2)| \leq \sigma^{2}/(5 \alpha_{1})$, we obtain
\begin{align*}
	\mathcal{S}^{*}(y|x) \geq & \phi_{\sigma}\big( y - f(0.5)\big) + \phi_{\sigma}\big( y-f(-0.5)\big) - \phi_{\sigma}\big(y-f(\theta_{0})\big) - \phi_{\sigma}\big( y-f(x-1/2)\big) / 5 
	\\&= \phi_{\sigma}\big( y - f(\theta_{0}) \big) \cdot \bigg(\exp\left\{\frac{2[f(0.5)-f(\theta_0)] \cdot y + f(\theta_0)^{2} - f(0.5)^{2}}{2\sigma^2} \right\}  \\& \qquad \qquad \qquad \qquad \quad + \exp\left\{\frac{2[f(-0.5)-f(\theta_0)] \cdot y + f(\theta_0)^{2} - f(-0.5)^{2}}{2\sigma^2} \right\}  \\ &\qquad \qquad \qquad \qquad \quad - 1 - \exp\left\{\frac{2[f(x-1/2)-f(\theta_0)]\cdot y + f(\theta_0)^{2} - f(x-1/2)^{2}}{2\sigma^2} \right\} \cdot \frac{1}{5}\bigg). 
\end{align*}
Now applying inequality~\eqref{ineq1:S_star-positive-proof-case2} with $t = 0.5$ and $t = -0.5$, inequality~\eqref{ineq2:S_star-positive-proof-case2} with $t = x-1/2$, and  substituting them into the inequality in the display above, we obtain
\begin{align*}
	\mathcal{S}^{*}(y|x) &\geq\phi_{\sigma}\big(y-f(\theta_0)\big) \cdot \Big(2\exp\big( -1/5 - 3/50 \big) -1 - \exp\big( 1/5 + 3/50 \big)/5 \Big) \geq 0.
\end{align*}
We thus conclude that
\begin{align*}
	\mathcal{S}^{*}(y|x) \geq 0 \quad \text{ for all }\; \ell(x) \leq y \leq L(x), \quad \text{where} \quad \ell(x) = f(x-1/2) - \frac{\sigma^2}{5\alpha_{1}},\; L(x) = f(x-1/2) + \frac{\sigma^{2}}{5\alpha_{1}}.
\end{align*}
The proof of inequalities~\eqref{ineq1:S-star-bound-part-2} and~\eqref{ineq2:S-star-bound-part-2} for $ 0 <x < \theta_{0} + 1/2$ is identical to the proof for $\theta_{0}-1/2 < x\leq 0$; we omit the calculations for brevity.
\qed

\subsubsection{Proof of Lemma~\ref{lemma:M-30}} \label{sec:proof-lemma-M-30}
We split the proof into three cases.

\paragraph*{Case 1: $x\leq \theta_{0}-1/2$ or $x\geq \theta_{0}+1/2$.} In this case, using expression~\eqref{eq:closed-form-S-star-case1}, $\mathcal{P}(y|x) = \NORMAL(0,2\sigma^{2})$, and applying the triangle inequality,  we obtain 
\begin{align*}
	\frac{\mathcal{S}^{*}(y|x) \; \vee \;0}{\mathcal{P}(y|x)} \leq \sqrt{2} \exp\left\{\frac{y^{2}}{4\sigma^2} \right\}
	\bigg( \exp\left\{-\frac{|y-f(0.5)|^{2}}{2\sigma^2} \right\} + \exp\left\{-\frac{|y-f(-0.5)|^{2}}{2\sigma^2} \right\} + \exp\left\{-\frac{|y-f(\theta_{0})|^{2}}{2\sigma^2} \right\} \bigg).
\end{align*}
Note that we have
\begin{align*}
	\exp\left\{\frac{y^{2}}{4\sigma^2} \right\}\exp\left\{-\frac{|y-f(0.5)|^{2}}{2\sigma^2} \right\} = 
	\exp\left\{\frac{-|y-2f(0.5)|^{2} + 2f(0.5)^{2}}{4\sigma^{2}} \right\} \leq \exp\left\{\frac{f(0.5)^{2}}{2\sigma^{2}}\right\} \leq 2,
\end{align*}
where in the last step we use $\sigma \geq  |f(0.5)|$. Similarly, we obtain
\[
	\exp\left\{\frac{y^{2}}{4\sigma^2} \right\}\exp\left\{-\frac{|y-f(-0.5)|^{2}}{2\sigma^2} \right\} \; \vee \; \exp\left\{\frac{y^{2}}{4\sigma^2} \right\}\exp\left\{-\frac{|y-f(\theta_{0})|^{2}}{2\sigma^2} \right\} \leq 2.
\]
Putting the pieces together yields $\frac{\mathcal{S}^{*}(y|x) \; \vee \;0}{\mathcal{P}(y|x)} \leq \sqrt{2}(2+2+2) < 30$. 

\paragraph*{Case 2: $\theta_{0}-1/2 < x \leq 0$.} Using Eq.~\eqref{eq:S-star-expression-case2}, $\mathcal{P}(y|x) = \NORMAL(0,2\sigma^2)$, $|f'(x+1/2)| \leq \alpha_{1}$, and applying the triangle inequality, we obtain
\begin{align*}
	\frac{\mathcal{S}^{*}(y|x) \; \vee \; 0}{\mathcal{P}(y|x)} &\leq \sqrt{2} \exp\left\{\frac{y^2}{4\sigma^2} \right\} \bigg( \exp\left\{-\frac{|y-f(0.5)|^{2}}{2\sigma^2} \right\} + \exp\left\{-\frac{|y-f(-0.5)|^{2}}{2\sigma^2} \right\} + \exp\left\{-\frac{|y-f(\theta_{0})|^{2}}{2\sigma^2} \right\} 
	\\ & \hspace{5cm} + \frac{\alpha_{1}}{\sigma^2} \exp\left\{-\frac{|y-f(x+1/2)|^2}{2\sigma^2} \right\} \cdot \big|y-f(x+1/2)
	\big| \bigg).
\end{align*}
Note that
\begin{align*}
	\exp\left\{\frac{y^2}{4\sigma^2} \right\} \exp\left\{-\frac{|y-f(0.5)|^{2}}{2\sigma^2} \right\} = \exp\left\{\frac{-|y-2f(0.5)|^{2} + 2f(0.5)^{2}}{4\sigma^{2}} \right\} \leq \exp\left\{\frac{f(0.5)^{2}}{2\sigma^{2}} \right\} \leq 2,
\end{align*}
where in the last step we use $\sigma \geq |f(0.5)|$. Proceeding similarly as in the display above and using $\sigma \geq \alpha_{0}$ yields
\begin{align*}
&\exp\left\{\frac{y^2}{4\sigma^2} \right\} \exp\left\{-\frac{|y-f(-0.5)|^{2}}{2\sigma^2} \right\} \leq 2, \quad \exp\left\{\frac{y^2}{4\sigma^2} \right\} \exp\left\{-\frac{|y-f(\theta_{0})|^2}{2\sigma^2} \right\} \leq 2,
\\ &\text{and} \quad \exp\left\{\frac{y^2}{4\sigma^2} \right\}\exp\left\{-\frac{|y-f(x+1/2)|^{2}}{2\sigma^2} \right\} \leq 2 \exp \left\{ \frac{-|y-2f(x+1/2)|^{2} }{4\sigma^{2}} \right\}.
\end{align*}
Putting the pieces together yields 
\begin{align}\label{ineq:intermediate-bound-S-star}
	\frac{\mathcal{S}^{*}(y|x) \; \vee \; 0}{\mathcal{P}(y|x)}
	\leq \sqrt{2} \bigg( 6 + 2\cdot \frac{\alpha_{1}}{\sigma^2}  \exp \left\{ \frac{-|y-2f(x+1/2)|^{2} }{4\sigma^{2}} \right\} \big| y-f(x+1/2) \big| \bigg). 
\end{align}
Now note that
\begin{align*}
&\exp\left\{-\frac{|y-2f(x+1/2)|^{2}}{4\sigma^{2}} \right\} |y-f(x+1/2)| \\
&\qquad \leq \exp \left\{-\frac{|y-2f(x+1/2)|^{2}}{4\sigma^{2}} \right\} \Big( |y-2f(x+1/2)| + |f(x+1/2)| \Big) \\
& \qquad \leq \max_{t\geq 0}\; t e^{-\frac{t^2}{4\sigma^2}} + |f(x+1/2)| \overset{\1}{\leq} \sqrt{2e}\sigma + \alpha_{0},
\end{align*}
where in step $\1$ we use $|f(x+1/2)| \leq \alpha_{0}$. Substituting the inequality in the display into inequality~\eqref{ineq:intermediate-bound-S-star} yields
\begin{align}\label{ineq:S-star-bound-case2}
	\frac{\mathcal{S}^{*}(y|x) \; \vee \; 0}{\mathcal{P}(y|x)} \leq \sqrt{2} \bigg( 6 + 2\cdot \frac{\alpha_{1}}{\sigma^2} \big(\sqrt{2e}\sigma + \alpha_{0} \big) \bigg) \leq 30,
\end{align}
where in the last step we use $\sigma \geq 5 \max\{ \alpha_{0}, \alpha_{1}\}$. 

\paragraph*{Case 3: $0<x < \theta_{0}+1/2$.} We have the closed form expression of $\mathcal{S}^{*}(y|x)$ in Eq.~\eqref{eq:S-star-expression-case3}. Note that the expression of $\mathcal{S}^{*}(y|x)$ is similar to the expression~\eqref{eq:S-star-expression-case2}. Thus, proceeding similarly as in the proof of inequality~\eqref{ineq:S-star-bound-case2}, we obtain the same bound
\[
	\frac{\mathcal{S}^{*}(y|x) \; \vee \; 0}{\mathcal{P}(y|x)} \leq 30,\quad \text{ for all }\; 0<x < \theta_{0}+1/2, y\in \real.
\]
\qed

\end{document}